\documentclass[psamsfonts]{amsart}
\usepackage{amsfonts,amssymb,amsaddr}
\usepackage[margin=2.5cm]{geometry}
\usepackage[all,arc]{xy}
\usepackage{tikz}
\usepackage{enumerate}
\usepackage[mathscr]{eucal}
\usepackage{parskip}
\usepackage[dvipdfm, colorlinks]{hyperref}
\usepackage{float}
\restylefloat{table}

\newtheorem{mainthm}{Main Theorem}
\newtheorem{thm}{Theorem}[section]
\newtheorem{cor}[thm]{Corollary}
\newtheorem{prop}[thm]{Proposition}
\newtheorem{lemma}[thm]{Lemma}
\newtheorem{conj}[thm]{Conjecture}

\theoremstyle{definition}
\newtheorem{defn}[thm]{Definition}

\theoremstyle{remark}
\newtheorem{rem}[thm]{Remark}

\numberwithin{equation}{section}
\newcommand{\bra}{\langle}
\newcommand{\ket}{\rangle}

\newcommand{\Wfin}{U(\fing,f)}

\newcommand{\sL}{{L}}
\newcommand{\slam}{{\lam}}
\newcommand{\sW}{W}
\newcommand{\roots}{\Delta}
\newcommand{\sroots}{\roots}
\newcommand{\isomap}{{\;\stackrel{_\sim}{\longrightarrow}\;}}
\newcommand{\fing}{\mathfrak{g}}
\newcommand{\mf}{\mathfrak}
\newcommand{\+}{\oplus}
\newcommand{\lam}{\lambda}
\newcommand{\Lam}{\Lambda}

\newcommand{\affg}{\widehat{\mathfrak{g}}}
\newcommand{\affW}{\widehat{W}}
\newcommand{\affh}{\widehat{\mathfrak{h}}}
\newcommand{\eW}{\widetilde{W}}
\newcommand{\dual}[1]{{#1}^*}
\newcommand{\wh}{\widehat}
\newcommand{\che}{^{\vee}}

\newcommand{\affM}{\widehat{M}}
\newcommand{\affL}{\widehat{L}}

\newcommand{\Vg}[1]{V^{#1}(\fing)} % Here [1] = k
\newcommand{\Vs}[1]{V_k(#1)} % Here [1] = \fing
\newcommand{\sV}{V_k(\fing)}
\newcommand{\finn}{\mathfrak{n}}
\newcommand{\finh}{\mathfrak{h}}
\newcommand{\on}{\operatorname}
\newcommand{\ra}{\rightarrow}
\newcommand{\mc}{\mathcal}
\newcommand{\BGG}{{\mathcal O}}
\newcommand{\HDS}{H_f}
\newcommand{\KL}{\text{\sf KL}}
\renewcommand{\*}{{\otimes}}
\newcommand{\sH}{H^{\mathrm{Lie}}}
\newcommand{\affBGG}{\widehat \BGG_{k}^{\fing_0}}

\newcommand{\cpps}{\check{P}_{+,\text{subreg}}}

\DeclareMathOperator{\var}{Var}
\DeclareMathOperator{\gr}{gr}
\DeclareMathOperator{\ad}{ad}
\DeclareMathOperator{\ann}{Ann}

\DeclareMathOperator{\en}{End}
\DeclareMathOperator{\Gal}{Gal}

\DeclareMathOperator{\prim}{Prim}
\DeclareMathOperator{\prin}{Pr}

\DeclareMathOperator{\str}{STr}
\DeclareMathOperator{\tr}{Tr}

\newcommand{\ppr}{P_{+,\text{reg}}}

\newcommand{\propor}{\sim}
\newcommand{\gachar}{\gamma}
\newcommand{\finW}{W}
\newcommand{\freg}{f_{\text{prin}}}
\newcommand{\fsubreg}{f_{\text{subreg}}}
\newcommand{\Osubreg}{\mathbb{O}_{\text{subreg}}}
\newcommand{\ov}{\overline}
\newcommand{\vac}{{\left|0\right>}}
\newcommand{\vir}{\text{Vir}}
\newcommand{\mbo}{\mathbb{O}}
\newcommand{\cliff}{\mathcal{C}\ell}
\newcommand{\qdim}{\on{qdim}}
\newcommand{\alstar}{\alpha_{*}}

\newcommand{\C}{\mathbb{C}}

\newcommand{\Q}{\mathbb{Q}}
\newcommand{\Z}{\mathbb{Z}}

\newcommand{\CC}{\mathcal{C}}
\newcommand{\CD}{\mathcal{D}}
\newcommand{\HH}{\mathcal{H}}

\newcommand{\W}{\mathscr{W}}

\newcommand{\CF}{\mathcal{F}}

\newcommand{\ma}{\mathfrak{a}}

\newcommand{\g}{\mathfrak{g}}
\newcommand{\h}{\mathfrak{h}}

\newcommand{\al}{\alpha}
\newcommand{\ga}{\gamma}

\newcommand{\D}{\Delta}
\newcommand{\eps}{\epsilon}
\newcommand{\la}{\lambda}
\newcommand{\La}{\Lambda}

\title[]{Rationality and Fusion Rules of Exceptional $\W$-Algebras}
\begin{document}

\begin{center}
{\LARGE \bf Rationality and Fusion Rules of Exceptional $\W$-Algebras} \par \bigskip

\renewcommand*{\thefootnote}{\fnsymbol{footnote}}
{\normalsize
Tomoyuki Arakawa\footnote{email: \texttt{arakawa@kurims.kyoto-u.ac.jp}}\textsuperscript{1},
Jethro van Ekeren\footnote{email: \texttt{jethrovanekeren@gmail.com}}\textsuperscript{2}
}

\par \bigskip

\textsuperscript{1}{\footnotesize Research Institute for Mathematical Sciences, Kyoto, Japan}

\par

\textsuperscript{2}{\footnotesize Instituto de Matem\'{a}tica e Estat\'{i}stica, UFF, Niter\'{o}i RJ, Brazil}

\par \bigskip
\end{center}

\vspace*{10mm}

\noindent
\textbf{Abstract.}
First, we prove 
the Kac-Wakimoto conjecture on modular invariance of characters
of exceptional affine $\W$-algebras. In fact more generally we prove modular invariance of characters of 
all lisse $\W$-algebras obtained through Hamiltonian reduction of admissible affine vertex algebras.
Second,
we prove the rationality of a large subclass of these $\W$-algebras, which includes all exceptional $\W$-algebras of type $A$ and lisse subregular $\W$-algebras in simply laced types. Third, for the latter cases we compute $S$-matrices and fusion rules.
Our results provide the first examples
of
rational 
 $\W$-algebras
associated with non-principal
distinguished nilpotent elements,
and the corresponding fusion rules are rather mysterious.
\vspace*{10mm}

%\tableofcontents

\section{Introduction}

Let $\fing$ be a finite dimensional simple Lie algebra, $f \in \fing$ a nilpotent element, and $k \in \C$ a complex number. The universal affine $\W$-algebra $\W^k(\fing, f)$ of level $k$ is obtained from the universal affine vertex algebra $V^k(\fing)$ through the process of quantized Drinfeld-Sokolov reduction. This construction was introduced in \cite{FeiginFrenkel} for $f$ a principal nilpotent element, and in \cite{KacRoaWak03} for $f$ a general nilpotent element.

Affine $\W$-algebras arise as algebras of symmetries of integrable models \cite{Kaclecture17},
in the geometric Langlands program \cite{Fre07,Gat,AraFre},
the 4d/2d duality \cite{BeeLemLie15,Beem:2015yu,SonXieYan17,AraMT}, 
$\mathcal{N}=4$ super Yang Mills gauge theories \cite{GaiRap19,CreGai},
and as invariants of $4$-manifolds \cite{FeiginGukov}.

It is believed that for appropriate choices of nilpotent element $f \in \fing$ and level $k$ the simple quotient $\W_k(\fing, f)$ of $\W^k(\fing, f)$ is a rational and lisse vertex algebra, and as such gives rise to a rational conformal field theory.
Indeed let $k = -h^\vee + p/q$ be an admissible level for the affine Kac-Moody algebra $\affg$ associated with $\g$. (Recall this means that the simple quotient $V_k(\g)$ of $V^k(\g)$ is admissible 
as a representation of  $\affg$ \cite{KacWak89}.) Then for $f$ a principal nilpotent element the rationality of $\W_k(\fing, f)$ was conjectured by Frenkel, Kac and Wakimoto \cite{FKW} and proved by the first named author in \cite{A2012Dec}. In connection with general nilpotent element $f$ the notion of an \emph{exceptional pair} $(f, q)$, where $f \in \fing$ is nilpotent and $q \geq 1$ is an integer was introduced in \cite{KacWak08} (and later extended in \cite{EKV}). Kac and Wakimoto conjectured that $\W_k(\fing, f)$ is rational whenever $k = -h^\vee+p/q$ is admissible and $(f, q)$ forms an exceptional pair.

In \cite{Ara09b} it was shown that the associated variety \cite{Ara12} of the simple affine vertex algebra $V_k(\fing)$ equals the closure of a certain nilpotent orbit $\mbo_q \subset \fing$ and that $\W_k(\fing, f)$ is nonzero and lisse if $f \in \mbo_q$. At the risk of ambiguity we should like to refer to $(f, q)$ as an \emph{exceptional pair} whenever $f \in \mbo_q$. Restriction to those pairs $(f, q)$ for which $f$ is of standard Levi type recovers the notion of exceptional pair of \cite{EKV}, and further restriction to those pairs for which $q$ is coprime to the lacety $r^\vee$ recovers the original notion of exceptional pair of \cite{KacWak08}. For $\fing$ of type $A$ all these notions coincide.

It was conjectured in \cite{Ara09b} that all exceptional $\W$-algebras (now in the broader sense of exceptional) are rational. Our first main result gives strong evidence for this conjecture, and thus for the conjecture of Kac and Wakimoto.
\begin{mainthm}[Theorem \ref{Th:semisimplicity-of-Zhu},
Theorem \ref{thm:modular-invariance}]\label{MainTh:modulality}
Let $k = -h^\vee + p/q$ be an  admissible level for $\affg$, and let $f \in \mbo_q$ be a nilpotent element. 
Then the Ramond twisted Zhu algebra $A(\W)$
of $\W = \W_k(\fing,f)$ is semisimple.
Let $\{\mathbf{L}_1,\dots, \mathbf{L}_r\}$ be a complete set of representatives of the isomorphism classes of simple 
Ramond twisted
$\W$-modules, and
$S_{\mathbf{L}_i}(\tau \mid u)=\on{Tr}_{\mathbf{L}_i} (u_0 q^{L_0-c/24})$
for $u \in \W$, the associated trace function.
Then $S_{\mathbf{L}_i}(\tau \mid u)$ converges to a holomorphic function on the upper half plane for all $u\in \W$, $i=1,\dots,r$. Moreover there is a representation
$\rho:SL_2(\Z)\ra \en_{\C}(\C^r)$ such that
\begin{align*}
S_{\mathbf{L}_i}\left(\frac{a\tau+b}{c\tau+d} \mid (c\tau+d)^{-L_{[0]}} u \right) = \sum_j \rho(A)_{ij} S_{\mathbf{L}_j}(\tau \mid u)
\end{align*}
for all $u \in \W$.
\end{mainthm}
We note that if $f$ admits a good even grading, as is always the case for $\fing$ of type $A$, a Ramond twisted module is the same thing as an untwisted module in the usual sense.

Our next result establishes rationality of those $\W$-algebras appearing in Main Theorem \ref{MainTh:modulality} for which the set of principle (or coprinciple) admissible weights of level $k$ satisfies a certain integrability condition relative to $f$. 
%This class includes all exceptional $\W$-algebras in type $A$, and all $\W$-algebras associated with subregular nilpotent orbits in simply laced types, which lie outside the class of exceptional $\W$-algebras  for non-type $A$ cases.
Let $k$ be an admissible level for $\affg$. The irreducible highest weight representation $L(\widehat{\lam})$ of $\affg$ with highest weight $\widehat{\lam} = k\Lambda_0 + \lam$ is a $V_k(\g)$-module if and only if $\lam$ belongs to the set $\Pr^k$ of level $k$ principal (or coprinciple) admissible weights \cite{A12-2}. For such $\lam \in \Pr^k$ we consider the annihilating ideal $J_{\la} \subset U(\fing)$ of $L(\la)$ and for a nilpotent orbit $\mbo$ we denote by $\Pr^k_\mbo$ the subset of $\Pr^k$ consisting of those $\lam$ for which $\var(J_\lam) = \ov{\mbo}$. The annihilator $J_{\la}$ depends only on the orbit of $\la$ under the dot action of the finite Weyl group $W$, and we set $[\Pr^k_\mbo] = \Pr^k_\mbo / W\circ(-)$.
%In fact we prove the Zhu algebra is semisimple in greater generality, and we deduce modular invariance of trace functions.
\begin{mainthm}[Theorem \ref{Th:simple-special-cases}]\label{MainTh:rationality}
Let $k = -h^\vee + p/q$ be an  admissible level for $\affg$, and let $f \in \mbo_q$ be a nilpotent element. Suppose $f$ admits a good even grading $\fing = \bigoplus_{j \in \Z} \fing_j$ such that every element of 
$[\Pr^k_{\mbo_q}]$ 
possesses a representative integrable with respect to $\fing_0$. Then the vertex algebra $\W_k(\fing, f)$ is rational and lisse, and all irreducible $\W_k(\fing, f)$-modules are obtained via quantised Drinfeld-Sokolov ``$-$''-reduction of level $k$ admissible $\affg$-modules.
\end{mainthm}

Main Theorem \ref{MainTh:rationality} proves the Kac-Wakimoto rationality conjecture for all exceptional $\W$-algebras of type $A$ (Theorem \ref{thm:rationality-of-type-A}).
It also proves the rationality of all exceptional subregular $\W$-algebras in simply laced types, see Theorem \ref{Th:rationality}. These latter algebras actually lie outside the class of Kac-Wakimoto exceptional $\W$-algebras. The values of $q$ for which $\mbo_q = \mbo_{\text{subreg}}$ are listed in Table \ref{table:Subregular.denom}.

Some special cases of Main Theorem \ref{MainTh:rationality} have already been proved; for $f$ a principal nilpotent element \cite{A2012Dec}, for the Bershadsky-Polyakov algebras $\W_k(\mathfrak{sl}_3, f_{\text{min}})$ \cite{Ara:BPalg}, and for $\W_k(\mathfrak{sl}_4, \fsubreg)$ \cite{CL:sl4}.

We note that
subregular $\W$-algebras in types $D$ and $E$
are \emph{distinguished}  $\W$-algebras,
that is, 
$\W$-algebras associated with distinguished nilpotent elements,
or equivalently,
$\W$-algebras that have  
 zero weight one subspaces.
Distinguished $\W$-algebras play a fundamental role 
among $\W$-algebras.
However,
the representation theory of distinguished $\W$-algebras that are not of principal type are  mysterious even at the level of finite $W$-algebras, since there are no canonical standard modules.
Main Theorem \ref{MainTh:rationality} provides the first examples
of rational
distinguished $\W$-algebras that are not of principal type.

We say a few words about the proofs of the theorems. The first step is to compute the Zhu algebra of $\W_k(\fing, f)$, which is a quotient of the finite $\W$-algebra $U(\g,f)$ \cite{Pre02}.
To do this we compute the Zhu algebra
of the admissible affine vertex algebra (Theorem \ref{Th:Zhu-admissible}),
and then
apply the commutativity \cite{A2012Dec}
of the Zhu algebra functor and the Drinfeld-Sokolov reduction functor. The irreducible $\W_k(\fing, f)$-modules are in bijection with those of the Zhu algebra. The role of the $\fing_0$-integrality condition is to ensure invariance of irreducible modules under the canonical action of the %Premet's 
component group $C(f)$, allowing us to use results of Losev \cite{Los11} on the representation theory of finite $\W$-algebras, and thereby characterise the irreducible $\W_k(\fing, f)$-modules. It remains to rule out nontrivial extensions between irreducible modules, which is done following the same approach as in \cite{A2012Dec}. Here a result of Gorelik-Kac \cite{GorKac0905} on complete reducibility of admissible representations of $\affg$ is used.

By Huang's result \cite{Hua08rigidity},
the module category of a rational, lisse, self-dual vertex algebra is a modular tensor category.
Therefore Main Theorem \ref{MainTh:rationality} provides a huge supply of modular tensor categories. Following the approach of \cite{FKW} and \cite{AvE} we compute the modular $S$-matrix and fusion rules of $\W_k(\fing, f)$ in the cases $\fing$ simply laced and $f$ subregular. We now explain the general features in simplified form.

We recall that the irreducible modules of the simple affine vertex algebra $V_{p-h^\vee}(\fing)$ are parametrised by regular dominant integral weights of level $p$. The $S$-matrix of this vertex algebra is, up to a normalisation, given by
\begin{align*}
K_p^{\la, \mu} = \sum_{w \in W} \epsilon(w) e^{-\frac{2\pi i}{p} (w(\la), \mu)},
\end{align*}
where the indices $\la, \mu$ run over the set of regular dominant integral weights of $\affg$ of level $p$. These coefficients lie in the cyclotomic field $\Q(\zeta_{N})$ where $N$ equals $p$ times the order of the centre of the adjoint group of $\fing$. For $a$ coprime to $N$ we let $\varphi_a \in \Gal(\Q(\zeta_{N})/\Q)$ denote the automorphism defined by $\varphi_a(\zeta_N) = \zeta_N^a$.

The $S$-matrix of $\W = \W_{-h^\vee+p/q}(\fing, \fsubreg)$ is, up to a normalisation, the Kronecker product matrix
\begin{align*}
\varphi_p(C_q) \otimes \varphi_q(K_p),
\end{align*}
where $C_q$ is a sort of degenerate analogue of $K_p$ given explicitly by (Theorem \ref{thm:Smatrix.quasi-general})
\begin{align}\label{eq:intro.Smat}
C_q^{\la, \mu} = \sum_{w(\alstar) \in \D_+} \epsilon(w) \frac{\left<w(\alstar), x\right>}{\left<\alstar, x\right>} e^{-\frac{2\pi i}{q} (w(\la), \mu)}.
\end{align}
Here $\alstar$ is the unique positive root of the Lie subalgebra $\fing_0$ and $x$ is an arbitrary element of $\finh$ not orthogonal to $\alstar$. The indices $\la, \mu$ run over the set of weights $\gamma \in Q$ of level $q$ satisfying $\left<\ga,\al_i\right> \in \Z_+$ for $i=1,\ldots,\ell$ and $\left<\ga,\al_i\right> = 0$ for exactly one $i$.

Since the fusion product multiplicities are integers determined from the $S$-matrix via the Verlinde formula, hence Galois invariant, we deduce that the fusion algebra of $\W$ is the tensor product of the fusion algebra of $V_{p-h^\vee}(\fing)$ with the fusion algebra whose $S$-matrix is $C_q$. This factorisation is quite parallel to the result discovered in \cite{FKW}; that the fusion algebra of the principal $\W$-algebra $\W_{-h^\vee+p/q}(\fing, \freg)$ is more or less the tensor product of the fusion algebras of two simple affine vertex algebras.

For type $A$ the matrix 
$C_q$ is the $1 \times 1$ identity matrix and 
therefore the fusion rules of $\W$ depend only on the numerator $p$ 
and coincide with those of $V_{p-h^\vee}(\fing)$. 
For types $D$ and $E$ the matrix 
$C_q$  is non-trivial and the fusion rules of $\W$ are more interesting.
In most cases (but not quite all, see Conjecture \ref{conj:Dtype} below) $C_q$ is itself naturally identified with the $S$-matrix of a subregular $\W$-algebra. In those cases for which this $\W$-algebra has asymptotic growth less than $1$, that is all cases except $\fing = E_7$, $q=16,17$ and $\fing=E_8$, $q=27,28,29$, we are able to identify it as a simple current extension of a Virasoro minimal model, thus confirming the $S$-matrix computed by (\ref{eq:intro.Smat}). We summarise these results in Table \ref{table:php1}. Most of the rational lisse $\W$-algebras obtained above are not unitary. However 
we conjecture 
that
 $\W_{-117/11}(E_6,{\fsubreg})$ and $\W_{-267/16}(E_7,{\fsubreg})$
 are unitary, giving rise seemingly to two new unitary modular tensor categories.

Finally, let us make some comments on the relations of the present work to 4d/2d duality \cite{BeeLemLie15}.
It has been shown in \cite{Dan, Cre17, SonXieYan17,WanXie}
that the exceptional $\W$-algebras at boundary admissible levels
appear as vertex algebras obtained from 4d $\mc{N}=2$ superconformal field theories (or more precisely from Argyres-Douglas theories), and the corresponding modular tensor categories 
are expected to coincide with those arising from 
the Coulomb branches of the corresponding 
4 dimensional theories,
or the wild Hitchin moduli spaces \cite{FPYY18,DedGukNak}.
Furthermore, the  modular tensor categories
associated with
exceptional distinguished $\W$-algebras 
at boundary admissible levels
are closely connected with  the Jacobian rings of certain 
hypersurface singularities \cite{XieYan}. 
We hope to come back to these points in our future work.

\emph{Acknowledgements.}
Both authors would like to thank Eric Rowell for very interesting discussions, and the referees for many helpful comments and corrections. T. A. is supported in part by JSPS KAKENHI Grant Numbers 17H01086, 17K18724. J. vE. is supported by CNPq Grant Numbers 409598/2016-0 and 303806/2017-6 and by a grant from the Serrapilheira Institute (grant number Serra -- 1912-31433). Part of this work has been presented at the ICM Satellite Workshop on Mathematical Physics at ICTP-SAIFR, S\~{a}o Paulo, August 2018. The authors would like to thank the organisers of this event.

\section{Affine Lie Algebras and Admissible Weights}\label{sec:aff.adm}

Let $\fing$ be a complex simple finite dimensional Lie algebra with a fixed triangular decomposition $\fing=\mf{n}_-\+\finh\+\mf{n}_+$. Let $\Delta$ be the set of roots of $\fing$, let $\Delta_+$ be the set of positive roots and $W$ the Weyl group. Put $Q=\sum_{\alpha\in \Delta}\Z \alpha$ the root lattice and 
$\check{Q}=\sum_{\alpha\in \Delta} \Z \alpha^{\vee}$ the coroot lattice. Here $\alpha^{\vee}=2\alpha/(\alpha,\alpha)$ where $(~,~)$ is the invariant bilinear form normalised so that the highest root $\theta$ satisfies $(\theta,\theta)=2$. The highest short root is denoted by $\theta_s$ and satisfies $(\theta_s,\theta_s)=2/r^{\vee}$, where $r^{\vee}$ is the lacing number of $\fing$. Let $P$ be the weight lattice of $\fing$ and 
{$\check{P}$} the coweight lattice. The simple coroots $\al_i^\vee$ form a basis of $\check{Q}$ and the basis $\{\varpi_i\}$ of $P$ dual to $\{\al_i^\vee\}$ defines the fundamental weights $\varpi_i$ of $\fing$. Similarly $\varpi_i^{\vee}$ will denote the fundamental coweights, which form a basis of $\check{P}$ dual to the simple roots
$\{\alpha_i\}$. Let $\rho$ be the half sum of positive roots of $\fing$ and $\rho^{\vee}$ the half sum of positive coroots. Let $P_+=\sum_{i}\Z_{\geq 0}\varpi_i$ and
$\check{P}_+=\sum_{i}\Z_{\geq 0}\varpi_i^{\vee}$ be
the sets of dominant weights and dominant coweights, respectively, 
and for $n\in \Z_{\geq 0}$ set
\begin{align}
P^{n}_+=\{\lam\in P_+\mid \bra \lam, \alpha^{\vee}\ket\leq n\text{ for }\alpha\in \Delta_+\},
\quad
\check{P}^{n}_+=\{\lam\in \check{P}_+\mid 
\bra \alpha,\lam\ket \leq n\text{ for }\alpha\in \Delta_+\}.
\label{eq:dom-weights}
\end{align}

For $\lam\in \finh^*$, we denote by $M(\lam)$ the Verma $\fing$-module with highest weight $\lam$, and by $L(\lam)$ the unique simple quotient of $M(\lam)$. A weight $\la$ is said to be dominant if $\left<\la+\rho, \al^\vee\right> \notin \Z_{<0}$ for all $\al \in \D_+$ and regular if it has trivial stabiliser under the `dot' action $y \circ \lam = y(\lam+\rho)-\rho$ of $W$.

A primitive ideal in the universal enveloping algebra $U(\fing)$ is by definition the annihilator of some irreducible $\fing$-module. Set
\begin{align*}
J_{\lam}=\on{Ann}_{U(\fing)}L(\lam).
\end{align*}
By Duflo's theorem \cite{Duf77}, any primitive ideal in $U(\fing)$ is of the form $J_{\lam}$ for some $\lam\in \finh^*$.

The centre of $U(\fing)$ is denoted 
$\mc{Z}(\fing)$, and the character $\gachar_\lam : \mc{Z}(\fing) \rightarrow \C$ is defined by $z v_\lam = \gachar_\lam(z) v_\lam$, where $v_\lam$ is a highest weight vector of $L(\lam)$.

Let
\begin{align*}
\affg=\fing[t,t^{-1}] \+ \C K
\end{align*}
be the affine Kac-Moody algebra associated with $\fing$, whose commutation relations are given by
\begin{align*}
& [xt^m,yt^n]=[x,y]t^{m+n}+ m (x,y) \delta_{m+n,0}K,
\quad [K,\affg]=0.
\end{align*}
Let $\affh=\finh\+ \C K$ be the standard Cartan subalgebra of $\affg$, and $\widetilde{\mathfrak{h}}=\finh\+ \C K\+ \C D$ the extended Cartan subalgebra of $\affg$. The dual of $\widetilde{\h}$ is $\widetilde{\h}^*=\finh^*\+ \C \Lam_0\+ \C \delta$, where $\Lam_0(K)=\delta(D)=1$ and $\Lam(\finh+\C D)=\delta(\finh\+ C K)=0$. The dual $\affh^*$ of $\affh$ is identified with the subspace $\finh\+\C \Lam_0\subset \widetilde{\h}^*$.
 
Let $\widehat{\Delta} \subset \widetilde{\h}^*$ be the set of roots of $\affg$ and
\begin{align*}
\widehat{\Delta}^{\text{re}} &= \{\alpha+n\delta\mid \alpha\in \Delta, n\in \Z\}, \\
\text{and} \quad \widehat{\Delta}^{\text{re}}_+ &= \{\alpha + n\delta \mid \alpha \in \Delta_+, n \in \Z_{\geq 0}\} \sqcup \{-\alpha+n\delta\mid \alpha\in \Delta_+, n \in \Z_{\geq 1}\}
\end{align*}
the subsets of real roots and real positive roots, respectively.

We denote by $\affW$ the Weyl group of $\affg$, so $\affW = W \ltimes \check{Q}$, and we denote by $\eW=W\ltimes \check{P}$ the extended affine Weyl group of $\fing$. We have $\eW= \eW_+ \ltimes \affW$, where $\eW_+$ is a finite subgroup of $\eW$ which we now describe. Write $\theta$ as a sum of simple roots: $\theta=\sum_{i=1}^\ell a_i \alpha_i$, and set $J=\{ i \in \{1,\dots,\ell\} \mid a_i=1\}$. Then we have
\begin{align}
\eW_+ = \{\pi_j = t_{\varpi_j} \bar\pi_j \mid j \in J \},
\label{eq:elements-of-eW-of-Dynkin-auto}
\end{align}
where $\bar \pi_j$ is the unique element of $W$ which fixes the set $\{\alpha_1,\dots,\alpha_{\ell},-\theta\}$ and which satisfies $\bar\pi_j(-\theta) = \alpha_j$.
\begin{lemma}[{\cite[Lemma 4.1.1]{FKW}}]\label{lemma:I.choose.Q}
Let $\fing$ be a simply laced Lie algebra, and $\affg$ its affinisation. If $\mu$ is a dominant integral weight of level coprime to $|J|$ then the elements of the $\eW_+$-orbit of $\mu$ represent a complete set of classes of $P/Q$. In particular there exists a unique $\pi \in \eW_+$ such that ${\pi(\mu)} \in Q$.
\end{lemma}

For $\lam\in \dual{\affh}$, let $\wh\Delta(\lam)$ denote its integral root system and $\affW(\lam)$ its integral Weyl group. That is
 \begin{align*}
&\wh\Delta(\lam) = \{ \alpha \in \wh{\Delta}^{\text{re}} \mid \left< \lam+\widehat{\rho},\alpha^\vee \right> \in \Z\},
&\affW(\lam)=\left< s_{\alpha} \mid \alpha \in \wh \Delta(\lam) \right> \subset \affW,
\end{align*}
where $s_{\alpha}$ is the reflection in the root $\alpha$ and $\widehat{\rho}=\rho+h^{\vee}\Lam_0$. Let $\wh\Delta(\lam)_+=\wh{\Delta}(\lam) \cap \wh{\Delta}^{\text{re}}_+$ be the set of positive roots of $\wh{\Delta}(\lam)$ and $\wh\Pi(\lam)\subset \wh{\Delta}(\lam)_+$ its set of simple roots.

A weight $\lam\in \dual{\affh}$ is called {\em admissible} \cite{KacWak89} if
\begin{enumerate}
\item $\lam$ is regular dominant, i.e., $\left< \lam+\widehat{\rho},\alpha\che \right> \not \in \{0,-1,-2,\dots\}$ for all $\alpha\in \widehat{\Delta}^{\text{re}}_+$, and

\item $\Q\wh\Delta(\lam)=\Q \wh{\Delta}^{\text{re}}$.
\end{enumerate}

For $\lam\in \dual{\finh}$ and $k\in \C$, we denote by $\affM_k(\lam)$ the Verma module of $\affg$ with highest weight $\lam+k\Lam_0\in \affh^*$, and by $\affL_k( \lam)$ the unique simple quotient of $\affM_k(\lam)$. The simple highest weight representation $\affL_k(\lam)$ is called admissible if $\widehat\lam$ is admissible. A complex number $k$ is called admissible if $k\Lam_0$ is admissible.

When clear from context we shall write, as above, $\widehat\lam$ for $\lam+k\La_0$. Occasionally we shall need to use the notation $\ov{\hat\lam}$ for the image of $\wh\lam \in \affh^*$ under the natural projection $\affh^* \rightarrow \h^*$.
\begin{prop}[\cite{KacWak89,KacWak08}]
\label{Pro:admissible number}
The number $k$ is admissible if and only if 
\begin{align*}
k+h\che=\frac{p}{q}
\quad \text{with } p,q\in \Z_{\geq 1},\ (p,q)=1,\ p\geq 
\begin{cases}
h\che&\text{if }(r\che, q)=1, \\
h&\text{if }(r\che,q)=r\che,
\end{cases}
 \end{align*}
where $h$ and $h^{\vee}$ are the Coxeter number and the dual Coxeter number of $\fing$, respectively. Furthermore
\begin{align*}
\wh{\Pi}(k\Lam_0) = \{\dot{\alpha_0},\alpha_1,\alpha_2,\dots,\alpha_{\ell}\},
\end{align*}
where 
\begin{align*}
\dot{\alpha_0} = 
\begin{cases}
-\theta+q\delta&\text{if }(r\che,q)=1, \\
 -\theta_s+\frac{q}{r\che}\delta&\text{if }(r\che ,q)=r\che.
\end{cases}
\end{align*}
\end{prop}
Let $k$ be an admissible number.
Then
$\widehat{\Delta} (k\Lam_0)\cong
\begin{cases} \widehat{\Delta}^{\text{re}}
&\text{if }(q|r^{\vee})=1,\\
{}^L\widehat{\Delta}^{\text{re}}&\text{if }(q|r^{\vee})\ne 1,
\end{cases}$
where ${}^L\widehat{\Delta}^{\text{re}}$
is the real roots system of the Langlands dual ${}^L\affg$ 
of $\affg$.
Let $\prin^k\subset \finh^*$
%(resp.\ $\on{CoPr}^k$) 
be the set of weights $\lam$ such that 
$\widehat\lam=\lam+k\Lam_0$ is admissible and 
there exists $y\in \eW$
such that $\widehat{\Delta}(\widehat\lam)=y(\widehat{\Delta}(k\Lam_0))$.
Such admissible weights are said to be \emph{principal admissible} \cite{KacWak89}
if $(q,r^{\vee})=1$
or
 \emph{coprincipal admissible} \cite{AvE}
 if $(q,r^{\vee})\ne 1$.
 Let $\prin^k_{\Z}=\prin^k\cap P$.
% \{\lam\in \prin^k \mid \bar \lam\in P\}$ where $\bar \lam$ denotes the restriction of $\lam$ to $\mf{h}$. 
Then $\prin^k$ decomposes as
\begin{align}
\prin^k = \bigcup_{\stackrel{y \in \eW}{y(\widehat{\Delta}(k\Lam_0)_+) \subset \widehat{\Delta}^{\text{re}}_+}} \prin^k_{y},
\quad \text{where} \quad 
\prin^k_{y} =\{ \lam\mid \widehat{\lam} \in y \circ (\prin^k_{\Z}+k\Lam_0)\}.
\label{eq:set-of-ad-wts}
\end{align}
%where $\bar \lam$ denotes the restriction of $\lam$ to $\mf{h}$.

For later purposes let us assume that
the denominator 
$q$ of $k$ is coprime to $r^{\vee}$ for the rest of this section and describe the set
$\Pr^k$ of principal admissible weights in more detail.
We have {\cite[(3.52)]{KacWak89}}
\begin{align}
\prin^k_{y} \cap  \prin^k_{y'}\ne \emptyset \iff
\prin^k_{y} = \prin^k_{y'} \iff y' = y t_{q\varpi_j} \bar \pi_j \quad \text{for some $j\in J$},
\label{eq:adm-wts-e}
\end{align}
where $q$ is the denominator of $k$.
If $p$ is the numerator of $k$ as in Proposition \ref{Pro:admissible number},
we have
\begin{align}
\prin^k_{\Z}=P_+^{p-h^{\vee}}.
\end{align} 
%that is,
%$\prin^k_{\Z}+(p-h^{\vee})\Lam_0$ is the set of dominant integral weight of $\affg$ of level $p-h^{\vee}$.
The cardinality of $\prin^k$ is $|\prin^k| = q^{\ell} |{P}_+^{p-h^\vee}|$ (\cite[Section 1.5]{FKW},  see also \cite[Proposition 3.2]{A12-2}).

Let $\phi$ denote the isometry of $\affh$ defined by $\phi|_{\finh} = 1$ and $\phi(\La_0) = (1/q)\La_0$. For $\al^\vee \in \finh$ the translate {\cite[Section 1.2]{FKW}} by $\eta \in \finh$ is \
\begin{align*}
t_{\eta}(\al^\vee + sK) = \al^\vee + \left( s - \eta(\al^\vee) \right)K.
\end{align*}
By \eqref{eq:set-of-ad-wts} any element of $\prin^k$ is of the form $\lam$ where
\begin{align}\label{eq:adm.wt.param}
\lam + k\Lambda_0 = \widehat\la = \wh y \phi(\wh\nu) - \wh\rho,
\end{align}
with $\wh\nu = p\La_0 + \nu$,
$\nu \in P_{+, \text{reg}}^p = \{\lam\in P\mid
0< \left<\lam, \alpha^{\vee}\right> < p\text{ for all } \alpha\in \Delta_+\}$, and $\wh y \in \eW$. We may write $\wh y = y t_{-\eta} = t_{\beta} y$ where $y \in \finW$ and $\eta, \beta \in P$, and put $\wh\eta = q\La_0 + \eta$. In this way we associate triples $(y, \eta, \nu)$ and $(y, \wh\eta, \wh\nu)$ to $\lambda \in \prin^k$. The associated triple is unique up to the action of the group $\eW_+$ by
\begin{align}\label{eq:Wplus.action.triples}
\pi : (y, \wh\eta, \wh\nu) \mapsto (y \ov\pi^{-1}, \pi(\wh\eta), \pi(\wh\nu)).
\end{align}

\section{Zhu's Algebras of Admissible Affine Vertex Algebras}\label{sec:zhu.alg}

For $k\in \C$ let $\Vg{k}$ be the universal affine vertex algebra associated with $\fing$ at level $k$. By definition
\begin{align*}
\Vg{k}=U(\affg)\*_{U(\fing[t]\+ \C K)}\C_k
\end{align*}
as a $\affg$-module, where $\C_k$ denotes the one dimensional representation of $\fing[t]\+ \C K$ on which $\fing[t]$ acts trivially and $K$ acts as multiplication by $k$. In this paper we assume that $k\ne -h^{\vee}$, so the vertex algebra $\Vg{k}$ is conformal with Virasoro element $L\in \Vg{k}$ and corresponding field
\begin{align*}
L(z)=\sum_{n\in \Z}L_nz^{-n-2}
\end{align*}
given by the Sugawara construction. The central charge of $V^k(\fing)$ is $\frac{k \dim \fing}{k+h^\vee}$. The unique simple %graded 
quotient $\Vs{\fing}$ of $\Vg{k}$ is called the {\em simple affine vertex algebra} associated with $\fing$ at level $k$. Note that $\Vs{\fing} \cong \affL_k(0)$ as a $\affg$-module.

The Zhu algebra (more precisely the Ramond twisted Zhu algebra) $A(V)$ of the vertex algebra $V$ is the quotient of $V$ by its vector subspace $\{a *_{-2} b \mid a, b \in V\}$, where by definition
\[
a *_n b = \sum_{j \in \Z_+} \binom{\Delta(a)}{j} a_{(n+j)}b,
\]
equipped with the associative product $a \otimes b \mapsto a *_{-1} b$. We denote in general the component of lowest $L_0$-eigenvalue in a graded $V$-module $M$ by $M_{\text{top}}$. This graded piece acquires the natural structure of $A(V)$-module, and it is known that the correspondence $M \mapsto M_{\text{top}}$ is a bijection \cite{Zhu96,De-Kac06} from the set of isomorphism classes of irreducible Ramond twisted $V$-modules to the set of simple $A(V)$-modules. 

For all $k$ there exists a natural isomorphism $A(\Vg{k})\cong U(\fing)$ and hence
\begin{align}
A(\Vs{\fing})\cong U(\fing)/I_k
\end{align}
for some two-sided ideal $I_k$ of $U(\fing)$. Thus $\affL_k(\lam)$ is a $\Vs{\fing}$-module if and only if $I_k \cdot L(\lam)=0$, that is to say if $I_k \subset J_{\lam}$. The following assertion was conjectured in \cite{AdaMil95} and proved by the first named author.
\begin{thm}[{\cite{A12-2}}]\label{Th:rationality-of-admissibleVA}
Let $k$ be an admissible number for $\affg$ and $\lam \in \dual{\finh}$. Then $\affL_k(\lam)$ is a $\Vs{\fing}$-module if and only if $\lam\in \prin^k$.
\end{thm}

\begin{cor}\label{Co:Zhu-modules-in-O}
Any $A(\Vs{\fing})$-module on which $\finn_+$ acts locally nilpotently and $\finh$ acts locally finitely is a direct sum of $L(\lam)$ with $\lam \in \prin^k$.
\end{cor}
\begin{proof}
It is sufficient to show that $\on{Ext}_{\fing}^1(L(\lam),L(\mu))=0$ for $\lam,\mu\in \prin^k$. If $\lam\ne \mu$ this is obvious since any weight in $\prin^k$ is dominant. So suppose there exists a non split exact sequence
\begin{align*}
0\ra L(\lam)\ra M\ra L(\lam)\ra 0.
\end{align*}
Applying the Zhu induction functor to this sequence gives rise to a non-trivial self-extension of $\affL_k(\lam)$, see \cite[Proof of Theorem 10.5]{A2012Dec}. But this contradicts the fact \cite{GorKac0905} that admissible representations of $\affg$ do not admit non-trivial self-extensions.
\end{proof}

\begin{lemma}[{\cite[Proposition 2.4]{AraFutRam17}}]\label{Lem:pri-ideals}
Let $\lam, \mu \in \prin^k$. Then
\begin{enumerate}
\item $J_{\lam}$ is the unique maximal two-sided ideal containing $U(\fing)\ker \gachar_{\lam}$. In particular $U(\fing)/J_{\lam}$ is a simple algebra.

\item We have $J_{ \lam}=J_{ \mu}$ if and only if there exists $w\in W$ such that $\mu=w\circ \lam$.
\end{enumerate}
\end{lemma}

Set
\begin{align*}
[\prin^k] = \prin^k/\sim,
\end{align*}
where $\lam\sim \mu \iff \mu\in W \circ \lam$. By Lemma \ref{Lem:pri-ideals}, the primitive ideal $J_{ \lam}$ depends only on the class of $\lam\in {\prin^k}$ in $[\prin^k]$.

We are now in a position to state the main result of this section.
\begin{thm}\label{Th:Zhu-admissible}
Let $k$ be an admissible number for $\affg$. The Zhu algebra $A(\sV)$ is isomorphic to the product of the simple algebras $U(\fing)/J_{ \lam}$ as $\lam$ runs over $[\prin^k]$. That is
\begin{align*}
A(\sV) \cong \prod_{\lam \in [\prin^k]}U(\fing)/J_{\lam}.
\end{align*}
\end{thm}

\begin{rem}{\ }
\begin{enumerate}
\item
Theorem \ref{Th:Zhu-admissible} implies, in particular, that any $ A(\sV)$-module admits a central character. Thus the image $[L]$ of the conformal vector $L$ is semisimple in $A(\sV)$.

\item For $\lam\in \prin^k$, the algebra $U(\fing)/J_{ \lam}$ is finite dimensional if and only if $\lam \in P$. 
%$$\lam\in Pr_{k,\Z}:=\{\lam \in Pr_k\mid \lam(\alpha^{\vee})\in \Z \text{ for }\alpha\in \Delta\}.$$
(Note that the latter condition implies that $\lam$ is a dominant integral weight, since any element of $\prin^k$ is regular dominant.) In this case we have $U(\fing)/J_{ \lam}\cong L(\lam)\* L( \lam)^*$.

\item In the special case of $k\in \Z_+$, or equivalently $V_k(\fing)$ integrable, $\prin^{k}$ is the projection
$P_+^k$ to $\finh^*$ of the set of dominant integral weights $\widehat{P}^{k}_+$ of $\affg$ of level $k$ and we have
\begin{align*}
A(\sV) \cong \prod_{\lam\in P_k^+} L(\lam) \* L(\lam)^*,
\end{align*}
which is well-known \cite{FreZhu92}.
\end{enumerate}

\end{rem}
\begin{proof}[Proof of Theorem \ref{Th:Zhu-admissible}]
Fix an admissible number $k$ and put $A = A(\sV) = U(\fing)/I_k$. For any $\lam \in \prin^k$ we have
\begin{align}
A \*_{U(\fing)}M(\lam)\cong L(\lam)
\label{eq:A-tensor-M}
\end{align}
as $A$-modules. Indeed $A\*_{U(\fing)}M(\lam)$ is a quotient of $M(\lam)$, and by Corollary \ref{Co:Zhu-modules-in-O} the only quotient $A$-module of $M(\lam)$ is $L(\lam)$.

Let $\mc{Z}(\fing)$ denote the center of $U(\fing)$ and put $Z=\mc{Z}(\fing)/\mc{Z}(\fing)\cap I_k$. Since the associated variety of $\sV$ is contained in the nilpotent cone $\mc{N}$ of $\fing$ \cite{Ara09b}, it follows that $Z$ is finite dimensional (see the proof of \cite[Corollary 5.3]{AM15}). Hence
\begin{align*}
Z=\prod_{\gachar\in \on{Specm}(Z)}Z_{\gachar}, \quad \text{where} \quad Z_{\gachar}=\{z\in Z \mid \text{$(z-\gachar(z))^r=0$ for 
$r \gg 0$} \}.
\end{align*}
This gives a decomposition
\begin{align*}
A = \prod_{\gachar\in \on{Specm}(Z)}A_{\gachar}, \quad \text{where} \quad A_{\gachar}=A\*_Z Z_{\gachar}.
\end{align*}
Moreover, by Theorem \ref{Th:rationality-of-admissibleVA} and Lemma \ref{Lem:pri-ideals}, we have
\begin{align*}
\on{Specm}(Z)=\{\gachar_{\lam}\mid \lam\in [\prin^k]\}.
\end{align*}
%where $\gachar_{\lam}:Z\ra \C$ is the evaluation at $L(\lam)$.

Let $\lam\in \prin^k$ and denote by $\C_{\gachar_{ \lam}}$  the one-dimensional representation of $Z$ corresponding to $\gachar_{\lam}$. Then $A\*_{Z}\C_{\gachar_{ \lam}}$ is a quotient algebra of $U(\fing)$, and we claim that
\begin{align}
A\*_Z \C_{\gachar_{ \lam}}\cong U(\fing)/J_{\lam}.
\label{eq:U/J}
\end{align}
Indeed $A\*_{Z}\C_{\gachar_{ \lam}}\cong U(\fing)/I$ for some two-sided ideal $I$, and clearly $I$ contains $U(\fing)\ker \gachar_{\lam}$. Since $M(\lam)/IM(\lam)$ is an $A$-module we have $M(\lam)/IM(\lam) \cong L(\lam)$ by Corollary \ref{Co:Zhu-modules-in-O}. 
On the other hand $\lam$ is dominant, and by \cite{Joseph:1979kq} one knows that the correspondence $I\mapsto IM(\lam)$ is an order preserving bijection from the set of two sided ideals of $U(\fing)$ containing $U(\fing)\ker \gachar_{ \lam}$ to the set of submodules of $M(\lam)$
for a dominant $\lam$. It follows that $I=J_{\lam}$.

There exists a finite filtration
\begin{align*}
0=Z_0\subset Z_1\subset Z_2\subset \dots \subset Z_r=Z_{\gachar_{\lam}}
\end{align*}
of $Z_{\gachar_{\lam}}$ as a $Z_{\gachar_{\lam}}$-module such that each successive quotient $Z_i/Z_{i-1}$ is isomorphic to the one-dimensional representation $\C_{\gachar_{ \lam}}$ of $Z_{\gachar_{\lam}}$. We put
$A_i=A\*_{Z}Z_i$ and obtain
\begin{align*}
0=A_0\subset A_1\subset A_2\subset \dots \subset A_r=A_{\gachar_{\lam}}
\end{align*}
and
\begin{align*}
A_i/A_{i-1}=A\*_{Z}(Z_i/Z_{i-1})=A\*_{Z}\C_{\gachar_{ \lam}}\cong U(\fing)/J_{\lam}.\end{align*}
as an $A$-bimodule. If the filtration is trivial, that is, if $r=1$, then $A_{\gachar_\lam} \cong U(\fing)/J_{\lam}$ and we are done. So we now assume that $r>1$ and deduce a contradiction.

Note that each exact sequence $0\ra Z_{i}\ra Z_{i+1}\ra Z_{i+1}/Z_i\ra 0$ is non split since otherwise we obtain a contradiction to \eqref{eq:U/J}. So we consider the non split exact sequence of $Z$-modules
\begin{align}
0\ra Z_1\ra Z_2 \ra Z_2/Z_1\ra 0.
\label{eq:non-spliting}
\end{align}
Since $Z_1=Z_2/Z_1=\C_{\gachar_{ \lam}}$ and $\lam$ is regular, it follows from the proof of \cite[Lemma 10.6]{A2012Dec} that (\ref{eq:non-spliting}) is obtained from an exact sequence $0\ra \C_{\lam}\ra E\ra \C_{ \lam}\ra 0$ of $\finh$-modules via the Harish-Chandra homomorphism $\mc{Z}(\fing)\ra S(\finh)$. We induce it to obtain an exact sequence 
\begin{align*}
0\ra M( \lam)\ra N\ra M( \lam)\ra 0
\end{align*}
of $\fing$-modules. Here $N=U(\fing)\*_{U(\finh\+\finn)}E$ where $\finn$ acts trivially on $E$. Applying the functor $A\*_{U(\fing)}(-)$ yields the exact sequence
\begin{align}
L(\lam)=A\*_{U(\fing)} M( \lam)\overset{\varphi_1}{\ra} A\*_{U(\fing)} N\ra A\*_{U(\fing)} M( \lam)=L(\lam)\ra 0.
\label{eq:extension}
\end{align}
Let $v_{\lam}$ be a highest weight vector of $M(\lam)$. Then $1\* v_{\lam}$ is a highest weight vector of $ A\*_{U(\fing)} M( \lam)=L(\lam)$. By construction,  this vector is mapped to a nonzero vector of $A\*_{U(\fing)}N$. Thus, the map $L(\lam)\ra A\*_{U(\fing)}N$ is an injection and \eqref{eq:extension} is a non-trivial self-extension of $L(\lam)$. But this contradicts Corollary \ref{Co:Zhu-modules-in-O}.
\end{proof}

For a two-sided ideal $I$ of $U(\fing)$ we denote by $\on{Var}(I)$ the zero locus of $\gr I$ in $\fing^*$ (or rather its image in $\g$ under the isomorphism induced by the Killing form), where $\gr I$ is the associated graded of $I$ with respect to the filtration induced from the PBW filtration of $U(\fing)$. Joseph's Theorem \cite{Jos85} asserts that $\on{Var}(I)=\overline{\mathbb{O}}$ for some nilpotent orbit $\mathbb{O}$ of $\fing$. Thus for each nilpotent orbit $\mathbb{O}$, we denote by $\prim_{\mathbb{O}}$ the set of those primitive ideals $I \subset U(\fing)$ such that $\var(I) = \ov{\mathbb{O}}$.

Recall the ideal $I_k$ such that $A(\Vs{\fing})=U(\fing)/I_k$. In \cite[Theorem 9.5]{A2012Dec} the first author has shown that, for an admissible number $k$, $\on{Var}(I_k)$ coincides with the associated variety \cite{Ara12} $X_{\Vs{\fing}}$ of $\Vs{\fing}$, and hence \cite{Ara09b},
\begin{align*}
\on{Var}(I_k)=\overline{\mathbb{O}}_q,
\end{align*}
where $\mathbb{O}_q$ is some nilpotent orbit of $\fing$ that depends only on the denominator $q$ of $k$. More precisely we have
\begin{align*}
\overline{\mathbb{O}}_q
= \begin{cases}
\{x\in \fing \mid (\ad x)^{2q}=0\}&\text{if }(r\che, q)=1,\\
\{x\in \fing \mid \pi_{\theta_s}(x)^{2q/r\che}=0\} &\text{if }(r\che,q)=r\che,
\end{cases}\end{align*}
where $\pi_{\theta_s}$ is the simple finite dimensional representation of $\fing$ with highest weight $\theta_s$.

For a nilpotent orbit $\mathbb{O}$ we write
\begin{align*}
\prin^k_{\mathbb{O}}=\{\lam\in \prin^k\mid \on{Var}(J_{\lam})=\overline{\mathbb{O}}\},
\end{align*}
and we write $[\prin^k_{\mathbb{O}}]$ for the image of $\prin^k_{\mathbb{O}}$ in $[\prin^k]$. We have $\var(J_\la) \subset \var(I_k) = \ov\mbo_q$ for all $\la \in \prin^k$, so we may write
\begin{align}
\prin^k=\bigsqcup_{\mathbb{O}\subset  \overline{\mathbb{O}}_q}\prin^k_{\mathbb{O}} \quad \text{and} \quad
[\prin^k]=\bigsqcup_{\mathbb{O}\subset  \overline{\mathbb{O}}_q}[\prin^k_{\mathbb{O}}].
\label{eq:dec-ad-weights}
\end{align}
Finally we put 
\begin{align}
\prin^k_\circ = \prin^k_{\mathbb{O}_q} \quad \text{and} \quad [\prin^k_\circ] = [\prin^k_{\mathbb{O}_q}].
\end{align}

\begin{thm}\label{thm:Joseph}
For an admissible number we have
\begin{align*}
\on{Pr}^k_\circ=\{\lam\in \on{Pr}^k\mid |\Delta(\lam)|=\dim \mc{N}-\dim \overline{\mathbb{O}}_q
\},
\end{align*}
where $\Delta(\lam)=\{\alpha\in \Delta\mid \langle \lam+\rho,\alpha^{\vee}\rangle \in \Z\}$.
\end{thm}
\begin{proof}
Let $\la \in \prin^k$.
Since
$\var(J_\la) \subset \var(I_k) = \ov\mbo_q$, it follows that
$\lam\in \on{Pr}^k_\circ$ if and only of
\begin{align}
\dim \on{Var}(J_{\la})=\dim \mathbb{O}_q,
\end{align}
where $q$ is the dominator of $k$.
Since all the elements in  $\prin^k$ are regular dominant,
we have
\begin{align*}
\dim \on{Var}(J_{\la})=\dim \mc{N}-|\Delta(\la)|
\end{align*}
for such $\lambda$ by \cite[Corollary 3.5]{Jos78}. The result follows.
\end{proof}

\section{Semisimplicity of Zhu Algebras of \texorpdfstring{$\W$}{W}-Algebras}\label{sec:w.alg}
Let $f$ be a nilpotent element of $\g$.
Recall \cite{ElaKac05} that a $\tfrac{1}{2}\Z$-grading
 $\fing = \bigoplus_{j \in \tfrac{1}{2}\Z} \fing_j$ is 
called \emph{good} for $f$ if $f \in \fing_{-1}$,
  $\ad(f) : \fing_j \rightarrow \fing_{j-1}$ is injective for $j \geq 1/2$ and surjective for $j \leq 1/2$. 
%  The grading itself is said to be ``good'' if it possesses at least one good element. 
  The grading is \emph{even} if $\fing_j = 0$ for $j \notin \Z$. 
Any nilpotent $f$ can be embedded into an $\mathfrak{sl}_2$-triple $\{e, h, f\}$, and $\fing$ thereby acquires a $\tfrac{1}{2}\Z$-grading by eigenvalues of $\ad(h/2)$.
%with respect to which the element $f$ is good.
 Such a grading is called a Dynkin grading. All Dynkin gradings are good, but not all good gradings are Dynkin.

We denote by $\W^k(\fing,f)$ the affine $\W$-algebra associated with $\fing$,
$f$ at level $k$ and a good grading $\g=\bigoplus_{j\in \frac{1}{2}\Z}\g_j$ \cite{KacRoaWak03}. It is defined by the generalized {\em quantized Drinfeld-Sokolov reduction}:
\begin{align*}
\W^k(\fing,f)=\HDS^0(V^k(\fing)),
\end{align*}
where $\HDS^\bullet(M)$ denotes the cohomology of the BRST complex associated with $\fing$, $f$ and $k$. This vertex algebra carries a conformal structure with central charge {\cite[Theorem 2.2]{KacRoaWak03}}
\begin{align}\label{eq:cc.formula}
\dim(\fing_0) - \frac{1}{2}\dim(\fing_{1/2}) - \frac{12}{k+h^\vee} |\rho - (k+h^\vee)x_0|^2,
\end{align}
where $x_0$ is the semisimple element of $\g$ that defined the grading,
that is,
$\g_j=\{x\in \g\mid [x_0,x]=jx\}$.
We note that the vertex algebra structure of $\W^k(\g,f)$ does not depend of the choice of the good grading
of $\g$ \cite[3.2.5]{AKM} but the conformal structure does.

Recall \cite{A07.rep,De-Kac06} that
\begin{align*}
A(\W^k(\fing,f))\cong U(\fing,f),
\end{align*}
where $U(\fing,f)$ is the {\em finite $W$-algebra} associated with $(\fing,f)$ \cite{Pre02}, whose construction we now briefly recall.

Let $\mc{HC}$ be the category of Harish-Chandra $U(\fing)$-bimodules, that is, the full subcategory of the category of $U(\fing)$-bimodules consisting of objects that are finitely generated as a $U(\fing)$-bimodule and on which the adjoint action of $\fing$ is locally finite. For an object $M$ of $\mc{HC}$ one defines a finite dimensional analogue of the quantized Drinfeld-Sokolov reduction which, by abuse of notation, we write as $\HDS^0(M)$, see \cite[Section 3]{A2012Dec}. In particular
\begin{align*}
U(\fing,f)=\HDS^0(U(\fing)),
\end{align*}
and in general the space $\HDS^0(M)$ is a bimodule over the finite $W$-algebra $U(\fing,f)$.

Let $\cliff$ be the Clifford algebra associated with the vector space $\fing_{>0} \+ \fing_{>0}^*$ equipped with the canonical symmetric bilinear form. Let $\chi : \fing_{\geq 1} \rightarrow \C$ be defined by $\chi(x) = (f, x)$, and let $\CD = U(\fing_{>0})/I_{>0, \chi}$ where $I_{>0, \chi}$ is the two sided ideal $U(\fing_{>0})\left<x - \chi(x) | x \in \fing_{\geq 1}\right>$. Note that if the good grading is even then $\CD$ is one dimensional. Now for $M \in \mc{HC}$ we put
\begin{align*}
C(M) = M \otimes \CD \otimes \cliff,
\end{align*}
This inherits a $\Z$-grading from $\cliff$ by assigning $\deg(\fing_{>0}) = -1$ and $\deg(\fing_{>0}^*) = +1$. Let $\{x_i\}$ be a homogeneous basis of $\fing_{>0}$, denote by $\ov{x}_i$ the canonical image of $x_i$ in $\CD$, and let $\{x_i^*\}$ be the dual basis of $\fing_{>0}^*$. Put
\begin{align*}
d = \sum_i \left( x_i \otimes 1 + 1 \otimes \ov{x}_i \right) \otimes x_i^* - 1 \otimes 1 \otimes \frac{1}{2} \sum_{i,j,k} c_{ij}^k x_i^* x_j^* x_k \in C^1(U(\fing)),
\end{align*}
where $c_{ij}^k$ are the structure constants $[x_i, x_j] = \sum_k c_{ij}^k x_k$. Then $\ad(d)$ is a differential on $C(M)$, and one defines
\begin{align*}
\HDS^\bullet(M) = H^\bullet(C(M), \ad(d)).
\end{align*}
The functor $\HDS^0(-)$ is identical to $(-)_{\dagger}$ in Losev's notation \cite{Los11} (see \cite[Remark 3.5]{A2012Dec}).

The space $\g_{1/2}$ is a symplectic vector space
with respect to the form $(x,y)\mapsto \chi([x,y])$.
Let $\mf{l}$ be a Lagrangian subspace  of $\g_{1/2}$,
and let $\mf{m}=\mf{l}\+\bigoplus_{j\geq 1}\g_j$.
Then $\mf{m}$ is a nilpotent subalgebra
of $\g$ and $\chi:\mf{m}\ra \C$ defines a character.
Put $Y = U(\fing) \otimes_{U(\mf{m})} \C_\chi$.
We have the algebra isomorphism \cite{DAnDe-De-07}
\begin{align*}
U(\g,f)\cong \on{End}_{U(\g)}(Y)^{\text{op}}.
\end{align*}

Let $E$ be a left $U(\fing, f)$-module, then $E \mapsto Y \otimes_{U(\fing, f)} E$ is a $U(\fing)$-module on which $x - \chi(x)$ acts locally nilpotently for all $x \in \mf{m}$, moreover this assignment defines an equivalence of categories known as the Skryabin equivalence \cite{Pre02}.

By \cite[Theorem 4.15]{Ara09b} $\HDS^0(\Vs{\fing})$ is a quotient vertex algebra of $\W^k(\fing,f)$, provided it is nonzero. By \cite[Theorem 8.1]{A2012Dec} we have
\begin{align*}
A(\HDS^0(\Vs{\fing}))\cong \HDS^0(A(\Vs{\fing})),
\end{align*}
and thus by Theorem \ref{Th:Zhu-admissible} we have, for any admissible number $k$,
\begin{align}
A(\HDS^0(\Vs{\fing}))\cong \prod_{[\lam]\in [\prin^k]}\HDS^0(U(\fing)/J_\lam).
\label{eq:decomposition-Zhu}
\end{align}

We recall the construction of the component group $C(f)$. Let us embed $f$ into an $\mf{sl}_2$ triple $\{e,h,f\} \subset \fing$, and let $G^{\natural}$ be the centralizer of the corresponding copy of $\mathfrak{sl}_2$ in the simply connected algebraic group $G$ with Lie algebra $\fing$. The restriction of the adjoint action of $G$ defines actions of $G^\natural$ on $U(\fing, f)$ and $\W^k(\fing,f)$. The Lie algebra $\fing^{\natural}$ of $G^{\natural}$ embeds naturally into $U(\fing, f)$ and therefore the action of the unit component $(G^{\natural})^{\circ}$ of $G^{\natural}$ is trivial. Hence the action of $G^{\natural}$ descends to the component group $C(f)=G^{\natural}/(G^{\natural})^{\circ}$.

Let $\lam\in \finh^*$. Then $\HDS^0(U(\fing)/J_{\lam})$ is naturally an algebra and the exact sequence
$0\ra J_{\lam}\ra U(\fing)\ra U(\fing)/J_{\lam}\ra 0$ induces an exact sequence
\begin{align*}
0\ra \HDS^0(J_{\lam})\ra \HDS^0(U(\fing))\ra \HDS^0(U(\fing)/J_{\lam})\ra 0
\end{align*}
(\cite{Gin08,Los11}, see also \cite[Section 3]{A2012Dec}). 
Thus, $\HDS^0(U(\fing)/J_{\lam})$ is a quotient algebra of $U(\fing,f)= \HDS^0(U(\fing))$.

Let $I \in \prim_{G \cdot f}$, that is $I \subset U(\g)$ is a primitive ideal and $\var(I) = \ov{G \cdot f}$. For such $I$ we denote by $\on{Fin}_{I}(U(\fing,f))$ the set of isomorphism classes of finite dimensional simple $U(\fing,f)$-modules $E$ such that $\on{Ann}_{U(\fing)}(Y \*_{U(\fing,f)}E)=I$.
\begin{thm}[\cite{Los11}]\label{Th:Losev}
Let $\lam\in \finh^*$ and let $f \in \fing$ be a nilpotent element.
\begin{enumerate}
\item If $\on{Var}(J_{\lam}) \subsetneq \overline{G.f}$ then $\HDS^0(U(\fing)/J_{\lam})=0$.

\item If $\on{Var}(J_{\lam}) = \overline{G.f}$ then 
\begin{align*}
\HDS^0(U(\fing)/J_{\lam}) \cong \prod_{E\in  \on{Fin}_{J_{\lam}}(U(\fing,f))}E\*E^*.
\end{align*}
In particular, $\HDS^0(U(\fing)/J_{\lam})$ is a semisimple algebra.

\item The natural action of $C(f)$ on $\on{Fin}_{J_{\lam}}(U(\fing,f))$ is transitive.
% See Losev expository paper Proposition 5.2. It asserts that the upper-dagger functor coincides with the annihilator of the Skryabin functor, when the primitive ideal comes from a module.
\end{enumerate}
\end{thm}

\begin{thm}\label{Th:semisimplicity1}
Let $k$ be an admissible number with denominator $q \in \Z_{\geq 1}$, and let $f \in \mathbb{O}_q$. Then
\begin{align*}
A(\HDS^0(\Vs{\fing}))\cong \prod_{[\lam]\in [\prin^k_{\circ}]} \left( \prod_{E\in  \on{Fin}_{J_{\lam}}(U(\fing,f))}E\*E^* \right).
\end{align*}
In particular $A(\HDS^0(\Vs{\fing}))$ is semisimple. Moreover, if $S[\lam]$ is a complete set of isomorphism classes of $\on{Fin}_{J_{\lam}}(U(\fing,f))$,
then 
 $\bigsqcup\limits_{[\lam]\in [\prin^k_{\circ}]}S[\lam] $ give a complete set of
isomorphism classes  simple $A(\HDS^0(\Vs{\fing})$-modules.
\end{thm}
\begin{proof}
The first statement follows immediately from Theorem \ref{Th:Zhu-admissible} and Theorem \ref{Th:Losev}.
To see the second statement,
we recall that the center of $U(\g,f)$ is isomorphic to $\mc{Z}(\g)$ (\cite{Pre07}),
and
hence $E\in  \on{Fin}_{J_{\lam}}(U(\fing,f))$ and $E'\in \on{Fin}_{J_{\mu}}(U(\fing,f))$ have distinct central characters
if $[\lam], [\mu]$ are distinct in  $[\prin^k_{\circ}]$.
\end{proof}

Let $\W_k(\fing,f)$ be the unique simple quotient of $\W^k(\fing,f)$.
\begin{thm}\label{Th:semisimplicity-of-Zhu}
Let $k$ be an admissible number with denominator $q\in \Z_{\geq 1}$, and let $f\in \mathbb{O}_q$. Then the Zhu algebra $A(\W_k(\fing,f))$ is semisimple.
\end{thm}
\begin{proof}
Since $\HDS^0(\Vs{\fing})$ is a quotient of $\W^k(\fing,f)$, we have that $\W_k(\fing,f)$ is a quotient of $\HDS^0(\Vs{\fing})$. Hence $A(\W_k(\fing,f))$ is a quotient of $A(\HDS^0(\Vs{\fing}))$. Since the latter is semisimple by Theorem \ref{Th:semisimplicity1}, so is the former.
\end{proof}

 \begin{rem}
 Conjecturally \cite{KacWak08}, 
 $\HDS^0(\Vs{\fing})$ is either zero or isomorphic to $ \W_k(\fing,f)$.
 This will be proved in Theorem \ref{Th:simplicity-vacumme}
  for 
the exceptional $\W$-algebra $\W_k(\fing,f)$ in the case that 
$f$ admits a good even grading.
\end{rem}

\section{Modular Invariance of Trace Functions}\label{sec:mod.inv}
For a vertex algebra $V$
let $R_V=V/C_2(V)$ denote Zhu's $C_2$ algebra \cite{Zhu96},
which is a Poisson algebra.
The associated variety  $X_V$ of $V$ is by definition
the Poisson variety
 $\on{Specm}R_V$
 \cite{Ara12}.
 By \cite{Ara09b}
 we have
 $X_{H^0_{f}(V_k(\g))}\cong X_{V_k(\g)}\cap \mathcal{S}_f$,
 where $\mathcal{S}_f$ is the Slodowy slice $f+\g^e$ at $f$.

Let $k$ be an admissible number with denominator $q$ and let $f\in \mathbb{O}_q$.
Since $X_{V_k(\g)}=\overline{\mathbb{O}}_q$,
$X_{H^0_{f}(V_k(\g))}=\{f\}$ by the transversality of $\mathcal{S}_f$ to the $G$-orbits.
Therefore $\HDS^0(V_k(\g))$ is lisse,
and so is $\W_k(\fing,f)$ \cite{Ara09b}. The vertex algebra $\W_k(\fing,f)$ is $\frac{1}{2}\Z_{\geq 0}$-graded, as are untwisted irreducible $\W_k(\fing,f)$-modules. On the other hand {\em Ramond twisted} irreducible $\W_k(\fing,f)$-modules (see \cite{KacWak08,Ara08-a}) are $\Z_{\geq 0}$-graded. Note that if $\W_k(\fing,f)$ is $\Z_+$-graded then a Ramond twisted module is the same thing as an untwisted module.

For a simple $A(\W_k(\fing,f))$-module $E$, we denote by $\mathbf{L}(E)$ the corresponding irreducible Ramond twisted $\W_k(\fing,f)$-module. The module $\mathbf{L}(E)$ is the unique simple quotient of the Verma module
\begin{align*}
\mathbf{M}(E)=\mathcal{U}(\W^k(\fing,f))\*_{ \mathcal{U}(\W^k(\fing,f))_{\geq  0}}E,
\end{align*}
where $\mathcal{U}(\W^k(\fing,f))=\bigoplus_{d\in \Z}\mathcal{U}(\W^k(\fing,f))_d$ is the Ramond twisted current algebra of $\W^k(\fing,f)$ \cite{FreZhu92,MatNagTsu05} and $\mathcal{U}(\W^k(\fing,f))_{\geq  0}=\bigoplus_{d\geq 0} \mathcal{U}(\W^k(\fing,f)))_{d}$.

Let $\on{Irr}_\Z(\W_k(\fing,f))$ be the complete set of isomorphism classes of irreducible Ramond twisted representations of $\W_k(\fing,f)$. For $M\in \on{Irr}_\Z(\W_k(\fing,f))$ and $u \in \W^k(\fing,f)$, set
\begin{align*}
S_M(\tau\mid u)=\on{Tr}_M (u_0 q^{L_0-c/24}),
\end{align*}
where $c$ is the central charge, $q=e^{2\pi i\tau}$, and $u_0$ is the degree preserving Fourier mode of the field $u(z)$. Set
\begin{align*}
L_{[0]} = L_0 - \sum_{j=1}^\infty \frac{(-1)^j}{j(j+1)} L_j.
\end{align*}

\begin{thm}\label{thm:modular-invariance}
Let $k$ be an admissible number with denominator $q\in \Z_{\geq 1}$, and let $f\in \mathbb{O}_q$. Let $\{E_1,\dots,E_r\}$ be a complete set of representatives of the isomorphism classes of simple modules over the Zhu algebra $A(\W)$ of $\W = \W_k(\fing,f)$. Then $S_{\mathbf{L}(E_i)}(\tau \mid u)$ converges to a holomorphic function on the upper half plane for all $u\in \W$, $i=1,\dots,r$. Moreover there is a representation
$\rho:SL_2(\Z)\ra \en_{\C}(\C^r)$ such that
\begin{align*}
S_{\mathbf{L}(E_i)}\left(\frac{a\tau+b}{c\tau+d} \mid (c\tau+d)^{-L_{[0]}} u \right) = \sum_{j=1}^r \rho(A)_{ij} S_{\mathbf{L}(E_j)}(\tau \mid u)
\end{align*}
for all $u \in \W$.
\end{thm}
\begin{proof}
As remarked by the second author in \cite{Van-Ekeren:2013xy}, in Zhu's result \cite{Zhu96}  on the modular invariance of the trace functions and its generalizations \cite{DonLiMas00,Van-Ekeren:2013xy}, the assumption of rationality of the vertex algebra can be replaced by semisimplicity of the Zhu algebra. Therefore the theorem follows immediately from Theorem \ref{Th:semisimplicity-of-Zhu} and the fact that $\W_k(\fing,f)$ is lisse \cite{Ara09b}.
\end{proof}

Since $S_M(\vac,\tau)$ is just the normalized character
\begin{align*}
\chi_{M}(\tau):=\on{Tr}_M ( q^{L_0-c_V/24})
\end{align*}
we obtain the following 
\begin{cor}\label{Co:modular-inv-character}
Let $k$ be an admissible number with denominator $q\in \Z_{\geq 1}$, and let $f\in \mathbb{O}_q$. Let $\{E_1,\dots,E_r\}$ be a complete set of representatives of the isomorphism classes of simple modules over the Zhu algebra $A(\W)$ of $\W = \W_k(\fing,f)$. Then the character $\chi_{\mathbf{L}(E_i)}(\tau)$ converges to a holomorphic function on the complex upper half plane, and the span of the $\chi_{\mathbf{L}(E_i)}(\tau)$, as $i$ runs over $\{1,\dots,r\}$, is invariant under the action of $SL_2(\Z)$.
\end{cor}

The exceptional $\W$-algebras introduced by Kac and Wakimoto  \cite{KacWak08}
are exactly the $\W$-algebras $\W_k(\g,f)$ with $f$  of standard Levi type and an admissible number $k$ whose denominator $q$ is coprime to $r^{\vee}$ \cite{Ara09b}.
Hence Corollary \ref{Co:modular-inv-character} in particular proves the modular invariance of the characters of modules over the exceptional $\W$-algebras,
which 
 was conjectured by Kac and Wakimoto \cite{KacWak08}.

We now restrict to the $\Z_{\geq 0}$-grade case and discuss fusion products. The category of modules of a rational lisse self-dual simple vertex algebra $V = \bigoplus_{n \in \Z_+} V_n$ of CFT-type carries the structure of a modular tensor category (MTC) under the fusion product $X \boxtimes Y$ of modules \cite{Huang08Contemp}. The fusion product is the one defined, in this level of generality, in \cite{HLtensorIandII}. Duals are given by the usual contragredient construction, and twists are given in terms of conformal dimensions. The $S$-matrix of the MTC coincides with the matrix $S = \rho(\begin{smallmatrix}0&-1\\1&0\\\end{smallmatrix})$ of Zhu's theorem. In particular the Verlinde formula asserts that the decomposition multiplicities or \emph{fusion rules} ${N}$ in $X \boxtimes Y \cong \bigoplus_Z {N}_{X,Y}^Z \cdot Z$ are given by
\begin{align*}
{N}_{X,Y}^Z = \sum_{W \in \on{Irr}(V)} \frac{S_{X,W}S_{Y,W}S_{Z',W}}{S_{V,W}}.
\end{align*}
Furthermore the \emph{charge conjugation matrix} $S^2$ is the permutation matrix which exchanges each module with its contragredient.

Let $\CC$ be an MTC. The integral Grothendieck group $\CF(\CC)$ acquires a commutative ring structure corresponding to the tensor product, and a distinguished $\Z$-basis corresponding to simple objects which comes equipped with an involution corresponding to duality. We refer to this structure as the fusion ring of $\CC$, and we also write $\CF(V)$ for the fusion ring of the category of representations of a rational lisse self-dual vertex algebra $V$. Since the structure of $\CF(\CC)$ is completely encoded in its $S$-matrix, we shall sometimes abuse notation and write $\CF(S)$.

\section{Self-Duality of \texorpdfstring{$\W$}{W}-Algebras}
Let $V$ be a vertex algebra 
of CFT-type,
that is,
$V$ is $\Z_{\geq 0}$-graded where $V_0=\C$. Then 
$V$ is called self-dual if 
$V \cong V'$ as $V$-modules,
where $M'$ denotes the contragredient dual \cite{Frenkel:1993aa} of the $V$-module $M$. Equivalently 
$V$ is self-dual if and only if 
it admits a non-degenerate symmetric invariant bilinear form.
According to Li \cite{Li94}, 
the space of symmetric invariant bilinear forms on $V$ is naturally isomorphic to the linear dual of $V_0/L_1V_1$.

The condition of self-duality 
depends on the choice of the conformal vector and is
necessary to apply Huang's result \cite{Huang08Contemp}
on the Verlinde formula. In this section we consider the question of self-duality of affine $\W$-algebras.
\begin{prop}\label{prop:sym-inv}
Suppose that $f$ admits a good even grading.
The simple $\W$-algebra
$\W_k(\g,f)$  is self-dual
if and only if 
\begin{align*}
(k+h^{\vee})(x_0|v)-\frac{1}{2}\on{tr}_{\g_{>0}}(\ad v)=0\quad\text{for all $v\in \g^f_0$.}
\end{align*}
\end{prop}
\begin{proof}
By Li's result $\W^k(\g,f)$ 
admits a nonzero symmetric invariant bilinear form
if and only if 
$L_1\W^k(\g,f)_1=0$.
By \cite{KacRoaWak03},
$\W^k(\g,f)_1$ is spanned by
the vectors
$J^{\{v\}}$, $v\in \g^f_0$,
defined in \cite[p.\ 320]{KacRoaWak03}.
Hence 
from 
\cite[Theorem 2.4 (b)]{KacRoaWak03}
it follows that
$\W^k(\g,f)$
admits a symmetric invariant bilinear form
such that 
$(\mathbf{1},\mathbf{1})=1$
if and only if
$(k+h^{\vee})(x_0|v)-\frac{1}{2}\on{tr}_{\g_{>0}}(\ad v)=0$
for all $v\in \g^f_0$.
If this is the case
the form induces a non-degenerate 
symmetric invariant bilinear form on the simple quotient
$\W_k(\g,f)$ and thus the latter is self-dual.
Conversely,
suppose that 
$(k+h^{\vee})(x_0|v)-\frac{1}{2}\on{tr}_{\g_{>0}}(\ad v)\ne 0$
for some $v\in \g^f_0$.
Then the image of $J^{\{v\}}$ in $\W_k(\g,f)_1$ is nonzero,
and $L_1 J^{\{v\}}=\mathbf{1}$ up to nonzero constant multiplication.
Therefore
 $\W_k(\g,f)$ is not self-dual
 according to Li's criterion.
\end{proof}

\begin{rem}
The notion of the contragredient dual 
 naturally extends to modules over $\tfrac{1}{2}\Z$-graded vertex algebras
 and the proof of Li's criterion
 applies without any change. See {\cite[Proposition 2.4]{Yamauchi}} for example.
 Hence Proposition \ref{prop:sym-inv}
 is valid without the assumption that $f$ admits a good even grading.
\end{rem}

Recall that a nilpotent element $f$ is called \emph{distinguished}
if $\g^f_0=0$ for the Dynkin grading.
For example all principal nilpotent elements are distinguished, and so are 
subregular nilpotent elements in types $D$ and $E$. 
All distinguished nilpotent elements are even. For a distinguished nilpotent element the only good grading is the Dynkin grading. The following assertion is a direct consequence of Proposition \ref{prop:sym-inv}.
\begin{prop}
Let $f$ be a distinguished nilpotent element.
Then 
the simple $\W$-algebra $\W_k(\g,f)$ is 
self-dual.
\end{prop}
\begin{rem}
More generally,
one can show that
$\W_k(\g,f)$  is self-dual for the Dynkin grading for any nilpotent element,
see \cite{AvEM}.
\end{rem}

For $\g=\mf{sl}_n$,
only principal nilpotent elements are distinguished.
A subregular nilpotent element $\fsubreg \in \mf{sl}_n$ is even if and only if 
$n$ is even.
\begin{prop}\label{prop:An.selfdual}
Let $\g=\mf{sl}_n$.
The simple subregular $\W$-algebra
$\W_k(\g,\fsubreg)$
is self-dual
if and only if
either
(1)
$n$ is even and the grading is Dynkin,
or
(2)
$k+n=n/(n-1)$.
\end{prop}
\begin{proof}
By \cite{ElaKac05} the 
good even gradings of $\fsubreg$ are classified 
by the pyramids corresponding to the partition $(n-1,1)$.
So we may take
\begin{align}
f=\sum_{i=1}^{m-2}E_{i+1, i}+E_{m+1,m-1}+
\sum_{i=m+1}^{n-1}E_{i+1, i} \quad \text{and} \quad x_0=\sum_{i=1}^{m-1}(m-i)E_{i,i}-\sum_{i=m+2}^{n}(i-m-1)E_{i,i}
\end{align}
for some $m=1,\dots,n$, to obtain $\g_0^f=\C v_m$, where $v_m=E_{m, m}-1/n\sum_{i=1}^n E_{i,i}$. 
It follows that
\begin{align*}
(k+h^{\vee})(x_0|v)-\frac{1}{2}\on{tr}_{\g_{>0}}(\ad v)=\frac{(n-1)(n-2k)}{2n}\left(k+h^{\vee}-\frac{n}{n-1}\right).
\end{align*}
The assertion follows
from 
Proposition \ref{prop:sym-inv}
noting that that 
$n=2m$ if and only if 
$n$ is even and the corresponding pyramid is symmetric,
or the grading is Dynkin.
\end{proof}

\begin{rem}
For $k+n=n/(n-1)$ we have $\W_{k}(\mf{sl}_n,\fsubreg)=\C$.
\end{rem}

\section{The ``\texorpdfstring{$-$}{-}''-Reduction Functor Revisited}\label{sec:reduc.revis}

From now on \emph{we assume that the nilpotent element $f$ admits a good even grading $\fing=\bigoplus_{j\in \Z}\fing_j$},
so that $\W^k(\g,f)$ is $\Z_{\geq 0}$-graded.
Note that this condition is satisfied by all nilpotent elements in type $A$ and subregular nilpotent elements in simply laced types. Without loss of generality we assume $\finh \subset \fing_0$ and that the root system $\Delta$ is compatible with the grading,
that is,
$\Delta_+ = \Delta_{0,+}\sqcup \bigsqcup_{j\geq 1}\Delta_j$, where $\Delta_j=\{\alpha\in \Delta\mid \fing_{\alpha}\subset \fing_j\}$
and $\Delta_{0,+}\subset \Delta_+$ is a set of positive roots of $\g_0$. We write $\fing_{\geq 0}=\bigoplus_{j\geq 0}\fing_j$ and $\fing_{<0}=\bigoplus_{j< 0}\fing_j$.

Let $\chi_- : \fing_{<0} \rightarrow \C$ be the character defined by $\chi_-(e_{-\al}) = \chi(e_{\al})$. As in {\cite[Section 5]{Ara08-a}} we write
\begin{align*}
\sH_\bullet(M) = H_\bullet(\fing_{<0},M\* \C_{\chi_-})
\end{align*}
for the Whittaker coinvariants functor, where $H_\bullet(-)$ denotes the usual Lie algebra homology functor.

Let $\BGG$ be the Bernstein-Gelfand-Gelfand category of $\fing$, and let $\BGG^{\fing_0}$ be the full subcategory of $\BGG$ consisting of those objects that are integrable as $\fing_0$-modules. We put
\begin{align}
P_{0,+} = \{\lam \in \dual{\finh} \mid \text{$\left< \lam,\alpha^{\vee}\right> \in \Z_{\geq 0}$ for all $\alpha \in \Delta_{0,+}$}\}.
\end{align}
Then $\{ L(\lam) \mid \lam \in P_{0,+} \}$ is a complete set of isomorphism classes of simple objects in $\BGG^{\fing_0}$.

Let $\on{Dim}M$ be the Gelfand-Kirillov dimension of the $\fing$-module $M$. For $M \in \BGG^{\fing_0}$ one has
\begin{align*}
\on{Dim}M \leq \dim \fing_{<0} = \frac{1}{2}\dim \overline{G.f},
\end{align*}
and in the case of equality we shall say that $\on{Dim}M$ is \emph{maximal}. We recall that $M$ is said to be \emph{holonomic} if $\on{Dim}M = \tfrac{1}{2}\dim\on{Var}\ann(M)$.
\begin{thm}[{\cite{Mat90,Matumoto:1990jk}, see also \cite[Theorem 5.1.1]{Ara08-a}}]{\ }
\label{thm:fdsi}
\begin{enumerate}
\item If $M\in \BGG^{\fing_0}$ then $\sH_0(M)$ is finite dimensional.

\item If $M\in \BGG^{\fing_0}$ then $\sH_i(M)=0$ for $i>0$.

\item If $\lam\in P_{0,+}$ then $\sH_0(L(\lam))\ne 0$ if and only if $\on{Dim}L(\lam)$ is maximal.
\end{enumerate}
\end{thm}

\begin{prop}\label{prop:GKdim}
For $\lam\in P_{0,+}$ we have $\on{Var}(J_\lam)=\overline{G.f}$ if and only if $\on{Dim}L(\lam)$ is maximal.
\end{prop}
\begin{proof}
We begin by recalling that $L(\lam)$ is holonomic, by a result of Joseph \cite{Jos78}. Now suppose $\on{Dim}L(\lam)$ to be maximal,
so that 
$\on{Dim}L(\lam)=\frac{1}{2}\dim \overline{G.f}$.
Then $\sH_0(L(\lam))\ne 0$ and thus there exists a vector $v\in L(\lam)$ such that its image $[v]$ in $\sH_0(L(\lam))$ is nonzero. As $\sH_0(L(\lam))$ is a $H_{f}^0(U(\fing)/J_\lam)$-module and $[v]=[1].[v]$, the image $[1]$ of $1$ in $H_{f}^0(U(\fing)/J_\lam)$ is nonzero, and hence $H_{f}^0(U(\fing)/J_\lam)\ne 0$. But this implies that $\on{Var}(J_\lam) \supset \overline{G.f}$. 
By the holonomicity of $L(\lam)$ we have 
$\dim\on{Var}(J_\lam)=\dim\overline{G.f}$ and thus $\on{Var}(J_\lam)=\overline{G.f}$ as required.

On the other hand if $\on{Dim}L(\lam)<\frac{1}{2}\dim \overline{G.f}$, then $\dim \on{Var}(J_\lam)<\dim \overline{G.f}$. This completes the proof.
\end{proof}

Let $\mathbb{O}=G \cdot f$ and for $I\in \prim_{\mathbb{O}}$ let $\{E_{I}[i] \mid i=1,\dots, N_I\}$ denote the complete set of isomorphism classes of irreducible finite dimensional representations of $\HDS^0(U(\fing)/I)$. Recall that the group $C(f)$ acts transitively on the set $\{E_{I}[i] \mid i=1,\dots, N_I\}$.
\begin{thm}\label{thm:Whittaker-cov-func}{\ }

\begin{itemize}
\item Let $\lam\in \h^*$,
$\mathbb{O}=G \cdot f$ and suppose that $J_\lam \in \prim_{\mathbb{O}}$. As $U(\fing,f)$-modules we have
\begin{align*}
\sH_p(L(\lam))\cong \bigoplus_{i=1}^{N_{J_{\lam}}} E_{J_{\lam}}[i]^{\oplus n_{p,i}}
\end{align*}
for some collection of $n_{p,i}\in \Z_{\geq 0}\cup \{\infty\}$ for each $p\geq 0$.

\item Let $\lam\in P_{0,+}$ and suppose $\on{Dim}L(\lam)$ is maximal. As $U(\fing,f)$-modules we have
\begin{align*}
\sH_p(L(\lam)) \cong 
\begin{cases}\bigoplus_{i=1}^{N_{J_{\lam}}}E_{J_\lam}[i]^{\+ n_{\lam}}
&\text{for }p=0,\\
0&\text{otherwise}
\end{cases}
\end{align*}
for some $n_{\lam}\in \Z_{>0}$.
\end{itemize}
\end{thm}

\begin{proof}
By Theorem \ref{Th:Losev}
\begin{align*}
\HDS^0(U(\fing)/J_{\lam}) \cong \bigoplus\limits_{i=1}^{N_{J_{\lam}}} E_{J_\lam}[i]\*E_{J_\lam}[i]^*
\end{align*}
is a finite-dimensional semisimple algebra. We note that $\sH_p(L(\lam))$ is a module over this algebra, and therefore is a direct sum of $E_{J_{\lam}}[i]$ with $i=1,\dots, N_{J_{\lam}}$. This proves the first part. 
For the second part,
 there is a natural inclusion $G^\natural \subset G^0$ the simply connected algebraic group of $\fing^0$. Since $L(\lam)$ is integrable with respect to $\fing^0$, we have $\sH_0(L(\lam))$ invariant under the action of $G^\natural$ and hence $C(f)$.
\end{proof}

We also need 
the following result of  Matumoto.

\begin{thm}[\cite{Mat90}]\label{Masumoto}
Suppose that $\left< \lam+\rho,\alpha\che \right> \not\in \Z_{\geq 1}$ for all $\alpha\in \Delta\backslash \Delta_0$. Then $\sH_0(L(\lam))$ is a (nonzero) simple $U(\fing,f)$-module.
\end{thm}

\begin{thm}\label{Th:generic-finite}
Let $\lam\in P_{0,+}$ such that $\left< \lam+\rho,\alpha^{\vee} \right> \not \in \Z_{\geq 1}$ for all $\alpha\in \Delta_{> 0}$. Then $\on{Dim}L(\lam)$ is maximal,
$J_{\lam}\in \prim_{G.f}$, and $\HDS^0(U(\fing)/J_{\lam})$ has a unique simple module $E_{J_\lam}$. Furthermore $\sH_0(L(\lam))\cong E_{J_\lam}$.
\end{thm}
\begin{proof}
By Theorem \ref{Masumoto},
$\sH_0(L(\lam))$ is nonzero simple $U(\fing,f)$-module.
Hence by Theorem \ref{thm:fdsi}
$\on{Dim}L(\lam)$ is maximal,
and so 
$J_{\lam}\in \prim_{G.f}$  by 
Proposition \ref{prop:GKdim}.
Finally,
 it follows from Theorems \ref{thm:Whittaker-cov-func}
 and \ref{Masumoto} that 
 $\sH_0(L(\lam))$ is the unique element of $\on{Fin}_{J_\la}(U(\fing,f))$.
\end{proof}

We recall the definition of (the ``$-$'' variant of) the quantized Drinfeld-Sokolov reduction functor $H_{f,-}^0(-)$ \cite{KacRoaWak03} \cite{Ara08-a}. For a vector space $\ma$ we denote by $L\ma$ the superalgebra $(\ma\+\ma^*)[t,t^{-1}]\+\C\mathbf{1}$ with 
whose even part is $\C\mathbf{1}$ 
and odd part is $(\ma\+\ma^*)[t,t^{-1}]$
and the commutation relation
$[at^m, bt^n] = (a, b)\delta_{m,-n}\mathbf{1}$, where $(\cdot, \cdot)$ is the canonical symmetric bilinear form $(a, \phi) = (\phi, a) = \phi(a)$ for $a \in \ma$, $\phi \in \ma^*$.

Let $\bigwedge^{\frac{\infty}{2}+\bullet}$ denote the Fock $L\fing_{<0}$-module with highest weight vector $\vac$, subject to the relations $\varphi_{\al, n\geq 1}\vac = 0$, $\varphi^*_{\al, n \geq 0}\vac = 0$. Here $\varphi_\al \equiv e_{-\al}$ for $\al \in \D_{>0}$ and $\{\varphi_\al^*\}$ is the dual basis of $\fing_{<0}^*$. Assigning $\deg(\varphi)=-1$ and $\deg(\varphi^*)=+1$ makes $\bigwedge^{\frac{\infty}{2}+\bullet}$ into a $\Z$-graded vertex superalgebra, with generating fields $\varphi_\al(z) = \sum_n \varphi_{\al, n} z^{-n}$ and $\varphi_\al^*(z) = \sum_n \varphi^*_{\al, n} z^{-n-1}$. For any $V^k(\fing)$-module we put
\begin{align}\label{eq:C.minus.def}
C_-^\bullet(M) = M \otimes \textstyle{\bigwedge^{\frac{\infty}{2}+\bullet}},
\end{align}
and we introduce the operator $Q_-$ on $C_-^\bullet(M)$ by
\begin{align*}
Q_- = \sum_{\al \in \D_{>0},n\in \Z} e_{-\al}t^{-n} \otimes \varphi_{\al,n}^* - \frac{1}{2} \sum_{\al, \beta,\ga \in \D_{>0}
\atop m,n\in \Z} c_{\beta\ga}^\al :\varphi_{\al,-m}\varphi_{\beta,-n}^*\varphi_{\ga,m+n}^*
+\sum_{\al \in \D_{>0}} \chi(e_{-\al})\varphi^*_{\al,0}.
%:, \quad \text{and} \quad p^f = \sum_{\al \in \D_{>0}} \chi(e_{-\al})\varphi^*_{\al}.
\end{align*}
%the element $Q_- \in C_-^1(V^k(\fing))$, defined to be the sum $Q^{\text{st}}_- + p^f$, where
%\begin{align*}
%Q^{\text{st}}_- = \sum_{\al \in \D_{>0}} e_{-\al} \otimes \varphi_{\al}^* - \frac{1}{2} \sum_{\al, \beta,\ga \in \D_{>0}} c_{\beta\ga}^\al :\varphi_\al\varphi_\beta^*\varphi_\ga^*:, \quad \text{and} \quad p^f = \sum_{\al \in \D_{>0}} \chi(e_{-\al})\varphi^*_{\al}.
%\end{align*}
Then we have $(Q_-)^2 = 0$, and we define
\begin{align*}
H_{f,-}^\bullet(M) = H^\bullet(C_-^\bullet(M), Q_-).
\end{align*}
If $M$ is any $V^k(\fing)$-module, then {\cite[Section 4.3]{Ara08-a}} (see also {\cite[Section 2.2]{FKW}}) the space $H_{f,-}^0(M)$ carries the structure of a Ramond twisted $\W^k(\fing, f)$-module
(which, under our assumption that $f$ admits a good even grading, is nothing but the usual untwisted module structure).

Let $\affBGG$ be the full subcategory of the category of left $\affg$-modules consisting of objects $M$ such that:
\begin{itemize}
\item $K$ acts as multiplication by $k$ on $M$,
\item $M$ admits a weight space decomposition with respect to the action of $\affh$,
\item there exists a finite subset $\{\mu_1,\dots,\mu_n\}$ of $\dual{\h}_{k}$ such that $M=\bigoplus\limits_{\mu\in \bigcup_i \mu_i-\widehat Q_+}M^{\mu}$,
\item for each $d\in \C$, $M_d$ is a direct sum of finite dimensional $\fing_0$-modules.
\end{itemize}
For $\lam\in P_{0,+}$,
put $\widehat{M}_{k,0}(\lam)=U(\affg)\*_{U(\fing[t]\+ \C K)}L(\lam)\in \widehat{\BGG}_{k}^{\fing_0}$, where $L(\lam)$ is considered as a $\fing[t]\+ \C K$-module on which $\fing[t]t$ acts trivially and $K$ acts as a multiplication by $k$. 
The modules $\widehat{L}_k( \lam)$, as $\lam$ ranges over $P_{0,+}$, form a complete set of simple objects of $\widehat{\BGG}_{k}^{\fing_0}$.

\begin{thm}\label{Th:simplicity}
 Let $k$ be any complex number.
\begin{enumerate}
\item ({\cite[Theorem 5.5.4]{Ara08-a}}) Let $M \in \affBGG$, then $H_{f,-}^i(M)=0$ for all $i\ne 0$.
In particular the functor
$\affBGG\ra \W^k(\g,f)\on{-Mod}$,
$M\mapsto H_{f,-}^0(M)$,
is exact.
\item (\cite[Theorem 5.5.4]{Ara08-a})
Let $\lam \in P_{0,+}$, then $H_{f,-}^0(\widehat L_k( \lam)) \neq 0$ if and only if $\on{Dim}L(\lam)$ is maximal.

\item Let $\lam \in P_{0,+}$ and suppose $\on{Dim}L(\lam)$ is maximal, then
\begin{align*}
H_{f,-}^0(\widehat M_{k,0}( \lam)) \cong \bigoplus_{i=1}^{N_{J_{\lam}}}\mathbf{M}(E_{J_\lam}[i])^{\oplus n_{\lam}}
\text{ and }
H_{f,-}^0(\widehat L_k( \lam)) \cong \bigoplus_{i=1}^{N_{J_{\lam}}}\mathbf{L}(E_{J_\lam}[i])^{\oplus n_{\lam}},
\end{align*}
\end{enumerate}
where 
$n_{\lam}$ is the multiplicity of $E_{J_\lam}[i]$ in $\sH_0(L(\lam))$ as in 
Theorem \ref{thm:Whittaker-cov-func}.
\end{thm}
\begin{proof}[Proof of (3)]
We have 
$H_{f,-}^0(\widehat M_{k,0}(\lam))_{\text{top}}\cong H_{f,-}^0(\widehat L_k( \lam))_{\text{top}}\cong \sH_0(L(\lam))$
(see \cite{Ara08-a}),
and the latter space is isomorphic to 
$\bigoplus_{i=1}^{N_{J_{\lam}}}E_{J_\lam}[i]^{\+ n_{\lam}}$ by 
Theorem \ref{thm:Whittaker-cov-func}.
On the other hand,
it was shown in \cite{Ara08-a}
that $H_{f,-}^0(\widehat M_{k,0}( \lam))$ is almost highest weight,
that is,
$H_{f,-}^0(\widehat M_{k,0}( \lam))$  is generated by
$H_{f,-}^0(\widehat M_{k,0}( \lam))_{\text{top}}$.
Therefore,
there is a surjective homomorphism
\begin{align*}
\bigoplus_{i=1}^{N_{J_{\lam}}}\mathbf{M}(E_{J_\lam}[i])^{\+ n_{\lam}}\twoheadrightarrow H_{f,-}^0(\widehat M_{k,0}( \lam)).\end{align*}
of $\W^k(\fing, f)$-modules.
But this must be an isomorphism
since their characters coincide (see \cite{Ara08-a}).
The exactness result of part (1) now implies that 
there is a surjective homomorphisms
\begin{align}\label{eq:s-map}
\bigoplus_{i=1}^{N_{J_{\lam}}}\mathbf{M}(E_{J_\lam}[i])^{\+ n_{\lam}}\twoheadrightarrow H_{f,-}^0(\widehat L_k( \lam))
\end{align}
On the other hand
it was also shown in \cite{Ara08-a}
that
$H_{f,-}^0(\widehat L_k( \lam))$ is almost irreducible,
that is,
any non-trivial submodule of $H_{f,-}^0(\widehat L_k( \lam))$
intersects $H_{f,-}^0(\widehat L_k(\lam))_{\text{top}}$ non-trivially.
It follows that \eqref{eq:s-map}
factors through the isomorphism
$\bigoplus_{i=1}^{N_{J_{\lam}}}\mathbf{L}(E_{J_\lam}[i])^{\+ n_{\lam}}\isomap  H_{f,-}^0(\widehat L_k( \lam))$.
This completes the proof.
\end{proof}

By the definition of the ``-'' reduction \cite{Ara08-a},
it follows that
the conformal dimension of the $\W^k(\fing, f)$-module $\mathbf{L}(E_{J_\lam}[i])$ is
\begin{align}\label{eq:confdim.formula}
h_{\lam}:=\frac{|\la+\rho|^2 - |\rho|^2}{2(k+h^\vee)} - \frac{k+h^\vee}{2}|x_0|^2 + (x_0, \rho),
\end{align}
see \cite[(3.1.6)]{FKW}. A simple but useful observation is that this expression is invariant under the dot action of $\finW$ on $\la$.

\begin{thm}\label{Th:irr}
Let $k$ be an admissible number for $\affg$. Let $\lam\in \prin^k\cap  P_{0,+}$ such that
$\on{Dim}L(\lam)$ is maximal. Then $\HDS^0(U(\fing)/J_{\lam})$ possesses a unique simple module, which we denote $E_{J_\lam}$, and 
\begin{align*}
H_{f,-}^0(\widehat M_{k,0}( \lam))\cong \mathbf{M}(E_{J_\lam})
\quad \text{and} \quad
H_{f,-}^0(\widehat L_k( \lam))\cong \mathbf{L}(E_{J_\lam}).
\end{align*}
\end{thm}
\begin{proof}
By Theorem \ref{Th:generic-finite},
$H^0_f(U(\g)/J_{\lam})$ has a unique simple module $E_{J_{\lam}}$.
The rest of the statements follows from Theorem \ref{Th:simplicity}.
\end{proof}

\begin{thm}\label{Th:simplicity-vacumme}
Let $k$ be an admissible number for $\affg$ with denominator $q$ and let $\lam\in\Pr^k_{\Z}$.
For 
 $f\in \overline{\mathbb{O}}_q$, 
% admitting a good even grading we have
we have $\lam-\frac{p}{q}x_0\in \prin^k\cap P_{0,+}$
% $\HDS^0(U(\fing)/J_{\lam})$ admits a unique simple module $E_{J_{\lam-\frac{p}{q}x_0}}$
and
\begin{align*}
H_f^0(\affL_k(\lam))\cong \mathbf{L}(E_{J_{\lam-\frac{p}{q}x_0}}).
\end{align*}
In particular,
\begin{align*}
\W_k(\g,f)\cong H_f^0(V_k(\g))\cong \mathbf{L}(E_{J_{-\frac{p}{q}x_0}}).
\end{align*}
\end{thm}
\begin{proof}
First, we have
$\lam-\frac{p}{q}x_0\in \Pr^k\cap  P_{0,+}$.
Indeed,
it is clear that $\lam-\frac{p}{q}x_0\in  P_{0,+}$.
Also,
$x_0\in  P^{\vee}_+$
 since we have assumed that $f$ admits a good even grading.
We have
$$\begin{cases}
\theta(x_0)<q &\text{if }(q|r^{\vee})=1,\\
\theta_s(x_0)<q/r^{\vee} % \frac{q}{r^{\vee}}
&\text{if }(q|r^{\vee})\ne 1,
\end{cases}$$
see \cite[5.7]{Ara09b}.
It follows that 
$t_{-x_0}(\widehat{\Delta}(k\Lam_0)_+)\subset \widehat{\Delta}_+^{\text{re}}$,
and
hence,
$\lam-\frac{p}{q}x_0\in \prin_{t_{-x_0}}^k\subset \prin^k$.
%$t_{-x_0}\circ \widehat{\lam}=\lam-\frac{p}{q}x_0+k\Lam_0$ is an admissible weight.
Moreover 
by Janzten's criterion \cite{Jan77}
we have
$L(\lam-\frac{p}{q}x_0)\cong U(\g)\*_{U(\g_{\geq 0})}L_{\g_0}(\lam-\frac{p}{q}x_0)$,
where 
$\mf{g}_{\geq 0}=\bigoplus_{j\geq 0}\g_j$
and $L_{\g_0}(\lam)$ is the irreducible highest weight representation of
$\g_0$ with highest weight $\lam$.
Hence $\on{Dim} L(\lam-\frac{p}{q}x_0)=\dim \mf{g}_{\geq 0}$
is maximal.
Therefore,
by Theorem \ref{Th:irr},
$\HDS^0(U(\fing)/J_{\lam})$ has a unique simple module $E_{J_{\lam-\frac{p}{q}x_0}}$
and 
$H_{f,-}^0(\widehat L_k( \lam-\frac{p}{q}x_0))\cong \mathbf{L}(E_{J_\lam-\frac{p}{q}x_0})$.

Now
recall that the center $Z(U(\g,f))$ of $U(\g,f)$ is isomorphic to the center 
$\mc{Z}(\g)$ of $U(\g)$ \cite{Pre07}.
By definition,
for $\lam, \mu\in \prin^k$,
$\chi_{\lam}=\chi_{\mu}$ if and only if 
$[\lam]=[\mu]$ in $[\prin^k]$,
where  the central character
$\chi_{\lam}:\mc{Z}(\g)\ra \C$ is the evaluation at $L_{\g}(\lam)$.
Hence
$E_{J_{-\frac{p}{q}x_0}}$ is the unique simple 
$A(H_f^0(V_k(\g)))$-module having central character $\chi_{\lam-\frac{p}{q}x_0}$.
On the other hand,
$Z(U(\g,f))$
acts  on
$H_f^0(\affL_k(\lam))_{\text{top}}$
by the central character $\chi_{\lam-\frac{p}{q}x_0}$
as well (see \cite[\S 5]{A2012Dec}).
%Since $H_f^0(\affL_k(\lam))$ is a $H_f^0(V_k(\g))$-module,
%the multiplicity of 
%$\mathbf{L}(E_{J_{\lam-\frac{p}{q}x_0}})$ in $H_f^0(\affL_k(\lam))$ is nonzero.
Because $H_f^0(\affL_k(\lam))$
and 
$\mathbf{L}(E_{J_{\lam-\frac{p}{q}x_0}})\cong 
H_{f,-}^0(\affL_k(\lam-\frac{p}{q}x_0))$
have the same character (see \cite[Proposition 5.12]{Ara09b}),
we obtain that
$H_f^0(\affL_k(\lam))\cong \mathbf{L}(E_{J_{\lam-\frac{p}{q}x_0}})$
as required.
 \end{proof}

\begin{thm}\label{Th:simple-special-cases}
Let $k$ be an admissible number for $\affg$ with denominator $q$ and let $f\in \mathbb{O}_q$. Suppose that each element of $[\prin^k_{\circ}]$ can be represented by an element $\lam \in \prin^k \cap P_{0,+}$. Then
\begin{enumerate}
\item for each $[\lam] \in [\prin^k_{\circ}]$ the algebra $\HDS^0(U(\fing)/J_{\lam})$ possesses a unique simple module, which we denote $E_{J_{\lam}}$,

\item the complete set of simple $\W_k(\fing,f)$-modules is
\begin{align*}
\left\{\mathbf{L}(E_{J_{\lam}}) \mid [\lam] \in [\prin^k_{\circ}]\right\},
\end{align*}

\item $\W_k(\fing,f)$ is rational.
\end{enumerate}
\end{thm}

\begin{proof}
Let $\{\lam_1,\dots,\lam_r\}$ be a subset of $P_{0,+}$ such that $[\prin^k_{\circ}]=\{[\lam_1],\dots,[\lam_r]\}$. By Theorem \ref{Th:irr} we have $H_{f,-}^0(\affL_k( \lam_i))\cong \mathbf{L}(E_{J_{\lam_i}})$ for $i=1,\ldots,r$, where $E_{J_{\lam_i}}$ is the unique simple $\HDS^0(U(\fing)/J_{\lam_i})$-module.

By Theorem \ref{Th:semisimplicity1} $\{E_{J_{\lam_i}} \mid i = 1,\ldots, r\}$ is a complete set of representatives of isomorphism classes of simple $A(\HDS^0(\Vs{\fing}))$-modules. Thus by Theorem \ref{Th:simplicity-vacumme} $\{\mathbf{L}(E_{J_{\lam_i}}) \mid i = 1,\ldots, r\}$ is a complete set of representatives of isomorphism classes of simple $\W_k(\fing,f)$-modules.

We already know that $\W_k(\fing,f)$ is lisse. It thus remains to show that
\begin{align*}
\on{Ext}_{\W_k(\fing,f)}^1(\mathbf{L}(E_{J_{\lam_i}}),\mathbf{L}(E_{J_{\lam_j}}))=0
\end{align*}
 for all $i,j$. Let
\begin{align}
0\ra \mathbf{L}(E_{J_{\lam_j}})\ra M\ra \mathbf{L}(E_{J_{\lam_i}})\ra 0
\label{eq:ext=0}
\end{align}
be an exact sequence of $\W_k(\fing,f)$-modules.

Let $h_{\lam_i}$ denote the conformal dimension, i.e., lowest $L_0$-eigenvalue, of $\mathbf{L}(E_{J_{\lam_i}})$. If $h_{\lam_i}=h_{\lam_j}$, then \eqref{eq:ext=0} is obtained by applying the induction functor to the sequence
\begin{align*}
0\ra \mathbf{L}(E_{J_{\lam_j}})_{\text{top}} \ra M_{\text{top}} \ra \mathbf{L}(E_{J_{\lam_i}})_{\text{top}} \ra 0
\end{align*}
of $A(\W_k(\fing,f))$-modules, and is therefore split because $A(\W_k(\fing,f))$ is semisimple.

Now suppose that $h_{\lam_i}<h_{\lam_j}$. There is a $\W^k(\fing,f)$-module homomorphism $\mathbf{M}(E_{J_{\lam_i}})\ra M$ such that the following diagram commutes.
\begin{align*}
\xymatrix{
&  \mathbf{M}(E_{J_{\lam_i}})\ar[d]^{} \ar[dl] \\
M\ar[r]^{} &\mathbf{L}(E_{J_{\lam_i}}).
}&
\end{align*}
If \eqref{eq:ext=0}
is non split 
then
$M$ must coincide with a homomorphic image of 
$\mathbf{M}(E_{J_{\lam_i}})$.
In particular $[\mathbf{M}(E_{J_{\lam_i}}) : \mathbf{L}(E_{J_{\lam_j}})] \neq 0$.
By Theorem \ref{Th:simplicity}
this occurs only if 
there exists $\mu\in P_{0,+}$ such that
$[\widehat{M}_k({\lam}_i): \widehat{L}_k({\mu})] \neq 0$
and $E_{J_{\lam_j}}$ is a direct summand of $\sH_0(L(\mu))$.
The second of these conditions implies that 
$\mu\in W\circ\lam_j$. But since $\widehat{\lam}_i$ and $\widehat{\lam}_i$ are dominant,
the first condition is only satisfied if $\widehat{\lam}_i = \widehat{\lam}_i$, which contradicts $h_{\lam_i}<h_{\lam_j}$.

Finally, the case
$h_{\lam_i}<h_{\lam_j}$
follows from the case $h_{\lam_i}>h_{\lam_j}$
by applying the duality functor to \eqref{eq:ext=0}.

\end{proof}

%\begin{rem}\label{rem:weights-of-singular-vector}
%Let $k=p/q-h^{\vee}$ be an admissible number for $\affg$
%and let $f\in \mathbb{O}_q$. 
%By Theorem \ref{Th:simplicity-vacumme}
%and the exactness of the functor $H_f^0(-)$ \cite{Ara09b},
%we have the exact sequence
%$0\ra H_{f}^0(N_k)\ra \W^k(\g,f)\ra \W_k(\g,f)\ra 0$,
%where  $N_k$ is the maximal submodule of $V_k(\g)$.
%Since
%$N_k$ is generated by a singular vector of weight
%$s_{\dot{\alpha}_0}\circ k\Lam_0=\begin{cases}
%(p-h^{\vee}+1)(\theta-q\delta)+k\Lam_0&\text{if }(q|r^{\vee})=1,\\
%(p-h+1)(\theta_s-\frac{q}{r^{\vee}}\delta)+k\Lam_0&\text{if }(q|r^{\vee})\ne 1,
%\end{cases}
%$
%the conformal dimension $H_{f}^0(N_k)$
%is given by
%$\begin{cases}
%(p-h^{\vee}+1)(q-\langle \theta,x_0\rangle)&\text{if }(q|r^{\vee})=1,\\
%(p-h+1)(\frac{q}{r^{\vee}}-\langle \theta_s,x_0\rangle)&\text{if }(q|r^{\vee})\ne 1.
%\end{cases}
%$
%\end{rem}

\begin{rem}\label{rem:weights-of-singular-vector}
Let $k=p/q-h^{\vee}$ be an admissible number for $\affg$ and let $f\in \mathbb{O}_q$. 
By Theorem \ref{Th:simplicity-vacumme}
and the exactness of the functor $H_f^0(-)$ \cite{Ara09b},
we have the exact sequence
$0\ra H_{f}^0(N_k)\ra \W^k(\g,f)\ra \W_k(\g,f)\ra 0$,
where $N_k$ is the maximal proper submodule of $V^k(\g)$. The submodule $N_k$ is generated by a singular vector $\sigma_k$ of weight $s_{\dot{\alpha}_0}\circ k\Lam_0$. From the relations $(\rho, \theta) = h^\vee-1$ and $(\rho, \theta_s) = h-1$ we compute
\begin{align*}
s_{\dot{\alpha}_0}\circ k\Lam_0=\begin{cases}
(p-h^{\vee}+1)(\theta-q\delta)+k\Lam_0&\text{if }(q|r^{\vee})=1,\\
(p-h+1)(\theta_s-\frac{q}{r^{\vee}}\delta)+k\Lam_0&\text{if }(q|r^{\vee})\ne 1.
\end{cases}
\end{align*}
From this we compute the action of the conformal vector $L_0 - (x_0)_0$ on the corresponding singular vector $\sigma_k \otimes \vac$ in $\W^k(\g,f)$ (cf. {\cite[Remark 2.3]{KacRoaWak03}}) to be given by
\begin{align*}
\begin{cases}
(p-h^{\vee}+1)(q-\langle \theta,x_0\rangle)&\text{if }(q|r^{\vee})=1,\\
(p-h+1)(\frac{q}{r^{\vee}}-\langle \theta_s,x_0\rangle)&\text{if }(q|r^{\vee})\ne 1.
\end{cases}
\end{align*}
\end{rem}

\section{Rationality of \texorpdfstring{$\W$}{W}-Algebras in Type \texorpdfstring{$A$}{A}}

For $\fing$ of type $A$ all nilpotent elements are standard Levi type, and so
$\W_k(\g,f)$ is exceptional in the sense of Kac and Wakimoto \cite{KacWak08}
if and only if $k$ is admissible and $f\in \mathbb{O}_q$,
where $q\in \Z_{\geq 1}$  is the denominator of $k$.
In this section we prove the rationality of all exceptional $\W$-algebras in type $A$.
 Throughout this section $\fing=\mf{sl}_n = A_{n-1}$. The Coxeter number of $\fing$ is $h^\vee = h = n$. It is known that the component group $C(f)$ is trivial for every nilpotent element $f \in \fing$, see e.g. \cite[6.1]{ColMcG93}. Therefore for any primitive ideal $I$ of $U(\fing)$ such that $\on{Var}(I)=\overline{G.f}$, the set $\on{Fin}_{I}(U(\fing,f))$ contains a single element, which we denote $E_{I}$. Hence
\begin{align*}
\HDS^0(U(\fing)/I)\cong E_I\*E_I^*.
\end{align*}
Moreover the correspondence $I\mapsto E_I$ gives a bijection from the set of primitive ideals of $U(\fing)$ satisfying $\on{Var}(I)=\overline{G.f}$ to the set of isomorphism classes of irreducible finite dimensional $U(\fing,f)$-modules. The module $E_I$ is described as follows \cite{BruKle08}.

As usual, we write
\begin{align*}
\sroots=\{\alpha_{i,j} \mid 1\leq i,j\leq n\},
\quad \text{and} \quad
\sroots_+=\{\alpha_{i,j}\mid 1\leq i<j\leq n\}.
\end{align*}
The nilpotent orbits are indexed by partitions of $n$. Indeed let $Y=(p_1\leq p_2\leq \dots \leq p_r)$ be a partition of $n$, then as in \cite{BruKle08} we identify $Y$ with the Young diagram having $p_i$ boxes in the $i^{\text{th}}$ row, and we number the boxes of $Y$ by $1,2,\ldots , n$ down columns from left to right. Let $\on{row}(i)$ and $ \on{col}(i)$ denote the row and column number of the $i^{\text{th}}$ box. Now put 
\begin{align*}
f=f_Y=\sum e_{j,i},
\end{align*}
where the sum runs over $(i, j)$ satisfying $\on{row}(i)=\on{row}(j)$ and $\on{col}(i)=\on{col}(j)-1$. Here $e_{i,j}$ stands for the $i,j$-matrix unit. Then $f$ is a nilpotent element of Jordan type $Y$. Declaring $\deg(e_{i,j}) = \on{col}(j)-\on{col}(i)$ equips $\fing$ with an good even grading
for $f_Y \in \fing_{-1}$ \cite{ElaKac05}.
%, relative to which $f_Y \in \fing_{-1}$ is good \cite{ElaKac05}. 
The subsets of roots
\begin{align}
\sroots_0 &= \{\alpha \in \sroots \mid e_{\al} \in \fing_0\} \\
\text{and} \quad
\sroots^{f} &= \{\alpha\in \sroots \mid \text{$\alpha(h)=0$ for all $h \in \finh^f$}\},
\end{align}
now become
\begin{align*}
\sroots_{0} &= \left\{\alpha_{i,j}\in \sroots \mid \text{the $i^{\text{th}}$ and $j^{\text{th}}$ boxes belong to the same column}\right\} \\
\text{and} \quad \sroots^f &= \{\alpha_{i,j}\in \sroots \mid \text{the $i^{\text{th}}$ and $j^{\text{th}}$ boxes belong to the same row}\}.
\end{align*}
Let $\sroots^f_+=\sroots^f\cap \sroots_+$ and
\begin{align}
\sW^f=\{w\in \sW \mid \text{$w(h)=h$ for all $h \in \finh^f$}\}.
\end{align}
Then $\sW^f$ is the subgroup of $W\cong\mathfrak{S}_n$ generated by $s_{\alpha}$ for $\alpha \in \sroots^f$. Finally we put
\begin{align}
P_{0,+}=\{\lam\in \dual{\finh}\mid \text{$\left< \lam,\alpha^{\vee}\right> \in \Z_{\geq 0}$ for all $\alpha\in \Delta_{0,+}$}\},
\end{align}
where $\Delta_{0,+}=\Delta_{0}\cap \Delta_{+}$.
\begin{thm}[{Brundan and Kleshchev \cite{BruKle08}}]
\label{Th:Brundan-Kleshchev}{\ }
\begin{enumerate}
 \item Let $\lam \in P_{0, +}$. Then $\sH_0(L(\lam))\ne 0$ if and only if $\left< \slam+\rho,\alpha\che\right>\not\in \Z_{\geq 1}$ for all $\alpha\in \sroots^f_+$. In this case $\sH_0(L(\slam))$ is  an irreducible $U(\fing,f)$-module. Furthermore every irreducible finite dimensional representation of $\Wfin$ arises in this way.

\item Let $\lam,\mu \in P_{0, +}$ and suppose $\sH_0(\sL(\slam))$ and $\sH_0(\sL( {\mu}))$ are nonzero. Then $\sH_0(\sL(\slam)) \cong \sH_0(\sL( {\mu}))$ if and only if $\mu = w \circ \la$ for some $w \in \sW^f$.
\end{enumerate}
\end{thm}

 \begin{cor}\label{cor:para-prim}
The assignment $\lam \mapsto J_{\lam}$ sets up a bijection
\begin{align*}
\begin{array}{ccc}
\{\lam\in P_{0,+}\mid \text{$\left< \lam+\rho,\alpha^{\vee}\right>\not\in \Z_{\geq 1}$ for all $\alpha\in \sroots^f_+$}\}/\sim \isomap  \prim_{G.f},
\end{array}
\end{align*}
where $\lam\sim \mu$ if and only if $\mu = w \circ \la$ for some $w \in \sW^f$. Furthermore $E_{J_{\lam}}\cong \sH_0(\sL(\slam))$.
 \end{cor}
 \begin{proof}
 By Losev's result \cite{Los11} and the fact that $C(f)$ is trivial,
there is a bijection between
$\on{Prim}_{G.f}$ and the isomorphism classes of simple
$U(\g,f)$-modules,
where the simple $U(\g,f)$-module corresponding 
to $J\in\on{Prim}_{G.f}$
is the unique simple module of the simple algebra $\HDS^0(U(\fing)/J)$,
see Theorem \ref{Th:Losev}.
The assertion is obtained by
 comparing this with Theorem \ref{Th:Brundan-Kleshchev}.
 \end{proof}

In the present case Theorem \ref{Th:simplicity} becomes
\begin{thm}[{\cite[Theorem 5.7.1]{Ara08-a}}]\label{Th:Main:typeA}
 Let $k$ be any complex number and let $\lam \in P_{0,+}$.
 Then
  \begin{align*}
H_{f,-}^0(\widehat L_k( \lam))\cong \begin{cases}
\mathbf{L}(E_{J_\lam})&\text{if }\left< \lam+\rho,\alpha\che\right>
\not\in \Z_{\geq 1}
\text{ for all }\alpha\in \sroots^f_+,\\
0&\text{otherwise}.
\end{cases}
\end{align*}
\end{thm}

\begin{lemma}\label{lem:represented-by}
Any element of $[\on{Pr}^k_\circ]$ can be represented by an element
of $\on{Pr}^k_\circ\cap P_{0,+}$.
\end{lemma}
\begin{proof}
By Corollary \ref{cor:para-prim},
for any $\lam\in \on{Pr}^k_\circ$ 
there exists $\mu\in P_{0,+}$ such that
$J_\lam=J_{\mu}$.
But then  $\widehat{L}_k(\mu)$ is a $V_k(\g)$-module by Theorem \ref{Th:Zhu-admissible},
and hence, $\mu\in \on{Pr}^k$ by Theorem \ref{Th:rationality-of-admissibleVA}.
\end{proof}
The following assertion follows immediately from
 Theorem \ref{Th:simple-special-cases} and
 Lemma \ref{lem:represented-by}.
\begin{thm}\label{thm:rationality-of-type-A}
Let $k$ be admissible, $f_q\in \mathbb{O}_q$.
Then $\W_k(\g,f_q)$ is rational
and
the complete set of simple $\W_k(\fing,f)$-modules is
given by
$
\left\{\mathbf{L}(E_{J_{\lam}}) \mid [\lam] \in [\prin^k_{\circ}]\right\}
$.
\end{thm}
We now describe the set $[{\prin^k_{\circ}}]$ more precisely.
Let $k = -n + p/q$ be an admissible number for $\affg$, so $p$ and $q$ are coprime and $p \geq n$. 
Since $[{\prin^k_{\circ}}]$ is described in \cite{FKW} in the cases where $q\geq n$,
we assume that 
  $q< n$,
so that $\mathbb{O}_q$ is non-principal.
Let $n = rq + s$ where $r, s \geq 0$ and $0 \leq s < q$. Then $\mathbb{O}_q$ is the nilpotent orbit corresponding to the partition 
%$(s,\overbrace{q,q,\dots,q}^r)$
$Y = (s,q,q,\dots,q)$ of $n$, in which $q$ appears $r$ times. We have $\fing_j=\{u\in \fing\mid [x_0,u]=ju\}$, where
 \begin{align}
x_0:=\sum_{i=1}^s \varpi_{i(r+1)}^{\vee}+\sum_{i=s+1}^{q-1} \varpi_{s(r+1)+(i-s)r}^{\vee}.
\label{eq:x0}
\end{align}
%\begin{prop}[{\cite[Proposition 2.8]{AraFutRam17}}]\label{Pro:AFR}
%Let $k = -n + p/q$ be an admissible number as above. 
%We have
%\begin{align*}
%[\prin^k]=\bigcup_{\substack{\eta\in P_+^{\vee}},\  {\theta(\eta)\leq q-1}}[\prin^k_{{t_{-\eta}}}],
%\end{align*}
%where $P_+^{\vee}$ is the set of dominant integral coweights.
%\end{prop}
For $\eta\in \check{P}$ we set ${\Delta}_{\eta}=\{\alpha\in {\Delta}\mid \left<\eta, \alpha\right>=0\} \subset {\Delta}$,
and define
\begin{align}
\check{P}_{+,f}^{q}=\{\eta\in \check{P}^{q}_+\mid \left<\eta, \theta\right> \leq q-1\text{ and }{\Delta}_{\eta}\cong \Delta_0\text{ as root systems}
\}.
\end{align}
Note that $x_0\in \check{P}_{+,f}^{q}$.
\begin{lemma}\label{lem:explicit.P.check}
A weight $\eta \in \check{P}_+$ belongs to $\check{P}_{+,f}^{q}$
if and only if 
there exists
a permutation
$(m_1,\dots,m_q)$  of $(\overbrace{r+1,\dots,r+1}^s,\overbrace{r,\dots,r}^{q-s})$
such that
$\eta=\sum_{j=1}^{q-1}\varpi_{\sum_{a=1}^jm_a}^{\vee}$.
\end{lemma}

\begin{proof}
From the description of $\Delta_{0}$ given at the beginning of this section, applied to the partition $Y$ associated with the orbit $\mbo_q$, it follows that $\Delta_0$ is the direct product of $q$ irreducible root systems; $s$ of them of type $A_{r+1}$ and $q-s$ of type $A_{r}$ (treating the case $r=0$ as $A_0 = \emptyset$). If $\Delta_\eta$ is to be isomorphic to $\Delta_0$ then we must have $\left<\eta, \alpha\right> > 0$ for at least $q-1$ simple roots $\alpha$, and now the condition $\left<\eta, \theta\right> \leq q-1$ implies that $\left<\eta, \alpha\right>$ equals $1$ for these simple roots and $0$ for all others. The simple roots $\alpha_{i_1}, \alpha_{i_2}, \ldots, \alpha_{i_{q-1}}$ for which $\left<\eta, \alpha_{i_j}\right>=1$ are of the form $i_j = \sum_{a=1}^j m_a$ for some permutation $(m_1,\dots,m_q)$ as in the statement of the lemma and so we are done.
\end{proof}

We define an equivalence relation in the set $P_+^{p-n} \times \check{P}_{+,f}^{q}$
by
\begin{align*}
(\lam,\eta)\sim (\lam',\eta')\quad \iff \quad
\begin{cases}
\lam'=\bar \pi_i\circ \lam+p\varpi_i=\bar \pi_i\lam+(p-n)\varpi_i \\
\eta'=\bar \pi_i\eta+q\varpi_i,
\end{cases}
\text{ for some $i=1,\dots,n-1$}
\end{align*}
(cf.\ \eqref{eq:adm-wts-e}). We may also describe the equivalence relation in terms of the description of $\check{P}_{+,f}^{q}$ given in Lemma \ref{lem:explicit.P.check}. Suppose $(\lam,\eta) \sim (\lam',\eta')$ in $P_+^{p-n} \times \check{P}_{+,f}^{q}$. Writing $\eta=\sum_{j=1}^{q-1}\varpi^\vee_{\sum_{a=1}^jm_a}$ and $\eta'=\sum_{j=1}^{q-1}\varpi^\vee_{\sum_{a=1}^j m_a'}$, setting $m_q=n-\sum_{j=1}^{q-1} m_j$ and $m_q'=n-\sum_{j=1}^{q-1} m_j'$ and, for convenience, considering the index modulo $n$, we then have $m'_{a}=m_{a+b}$ for all $a=1,\dots, q$, for some fixed $b$.

\begin{thm}\label{thm:weight.param}
Let $k = -n+p/q$ with $q\leq n$.
We have a bijection
\begin{align*}
\left(P_+^{p-n} \times \check{P}_{+,f}^{q}\right)/\sim \isomap [{\prin^k_{\circ}}],\quad (\lam,\eta)\mapsto [\lam-\frac{p}{q}\eta].
\end{align*}
\end{thm}

\begin{proof}
First we recall that
$\dim {\mathbb{O}_q}$ coincides with the maximal 
 Gelfand-Kirillov dimension of objects of $\BGG^{\g_0}$,
 and thus,
 \begin{align*}
\dim {\mathbb{O}_q}=|\Delta|-|\Delta_0|.
\end{align*}
Hence, 
by Theorem \ref{thm:Joseph},
%by Proposition \ref{Pro:AFR}, 
 $\lam\in {\prin}^{k}$ belongs to
${\prin}^{k}_{\circ}$ 
if and only if $|\Delta(\lam)|=|\Delta_0|$. 
\\
Let
  $\eta\in  \check{P}_{+,f}^{q}$.
  Since $\eta \in \check{P}_+$ and $\left<\eta, \theta\right> \leq q-1$ we have $\left<\eta, \alpha\right><q$ for all $\alpha\in \Delta_+$,
and therefore
$t_{-\eta}(\widehat{\Delta}(k\Lam_0)_+) \subset \widehat{\Delta}^{\text{re}}_+$. Hence for all $\lam\in P^{p-n}_+$ we have
$t_{-\eta}\circ \widehat{\lam}=\lam-\frac{p}{q}\eta +k\Lam_0$
with $\lam-\frac{p}{q}\eta \in \prin^k$.
Moreover, $\Delta(\lam-\frac{p}{q}\eta)=\Delta(\eta) \cong \Delta_0$ since $\eta \in \check{P}_{+, f}^q$
and hence
 $\lam-\frac{p}{q}\eta \in \prin^k_\circ$ by the above criterion. 
%On the other hand let us suppose that $(\lam, \eta) \sim (\lam, \eta)$ in $P^{p-n}_+ \times \check{P}_{+, f}^q$, in particular
% $\eta'=\bar \pi_i\eta+q\varpi_i$,
% then
% $\bar \pi_{n-i}t_{-\eta'}=t_{-\eta}t_{q\varpi_{n-i}}\bar \pi_{n-i}$,
% and so
% $t_{-\eta}\circ \widehat{\lam}=t_{-\eta}t_{q\omega_{n-i}}\bar \pi_{n-j}t_{q\omega_{i}}\bar \pi_{j}\circ \widehat{\lam}
% =\bar \pi_{n-i}t_{-\mu'}\circ \widehat{\lam'}$.
% Here we have used that $\bar \pi_i \bar \pi_{n-i}=1$.
% Thus,
% $[\lam-\frac{p}{q}\mu]=[\bar \pi_{n-i}\circ (\lam'-\frac{p}{q}{\mu'})]=[\lam'-\frac{p}{q}{\mu'}]$ in 
% $[\prin^k_{\circ}]$.
%We have shown that  the map 
%in Theorem is well-defined.
On the other hand let us suppose that $(\lam', \eta') \sim (\lam, \eta)$ in $P^{p-n}_+ \times \check{P}_{+, f}^q$, so that $\lam'=\bar \pi_i\circ\lam+p\varpi_i$ and $\eta'=\bar \pi_i\eta+q\varpi_i$. Then
\[
\bar \pi_{i}\circ (\lam-\frac{p}{q}{\eta}) = \bar \pi_{i}\circ\lam - \frac{p}{q}\bar\pi_i{\eta} = \lam' - p\varpi_i - \frac{p}{q}(\eta'+q\varpi_i) = \lam'-\frac{p}{q}\eta'.
\]
We have shown that  the map in the theorem statement is well-defined.

Next we show the injectivity of the map. We shall use the relation $\bar\pi_{n-i}(\varpi_i) = -\varpi_{n-i}$, which is easily proved. Let $(\lam, \eta), (\lam',\eta')\in  P^{p-n}_+\times \check{P}_{+,f}^{q}$
 and suppose that 
 $[\lam'-\frac{p}{q}\eta']=[\lam-\frac{p}{q}\eta]$ in $[\prin^k_{\circ}]$, i.e., that there exists $w\in W$ such that
 $\lam'-\frac{p}{q}\eta'=w\circ (\lam-\frac{p}{q}\eta)$.
 Since $\lam'-\frac{p}{q}\eta'\in \prin^k$,
 this means that $wt_{-\eta}(\widehat{\Delta}(k\Lam_0)_+) \subset \widehat{\Delta}^{\text{re}}_+$
 and $\lam'-\frac{p}{q}\eta' \in \prin^k_{wt_{-\eta}} \cap \prin^k_{t_{-\eta'}}$.
 By \eqref{eq:adm-wts-e},
 we get that
 $wt_{-\eta}=t_{-\eta'}t_{q\varpi_j}\bar \pi_j
 =\bar \pi_j t_{-(\bar \pi_{n-j}\eta'+q\varpi_{n-j})}$ for some $j$.
 Thus
\[
w=\bar \pi_{j}, \qquad \text{and} \qquad \eta=\bar \pi_{n-j}\eta'+q\varpi_{n-j}.
\] 
From $\lam'-\frac{p}{q}\eta'=w\circ (\lam-\frac{p}{q}\eta)$, and the equalities above, it follows that $\lam=\bar\pi_{n-i}\circ (\lam'-\frac{p}{q}\eta')+\frac{p}{q}\eta=
 \bar\pi_{n-i}\circ \lam+p\varpi_{n-i}$.
We have shown that
$(\lam',\mu')\sim (\lam,\mu)$.
%\\
% To show the surjectivity,
% we first show that
% any element of $[\on{Pr}^k_\circ]$ can be represented by an element
%of $\on{Pr}^k_\circ\cap P_{0,+}$.
%Indeed,
%by Corollary \ref{cor:para-prim},
%for any $\lam\in \on{Pr}^k_\circ$ 
%there exits $\mu\in P_{0,+}$ such that
%$J_\lam=J_{\mu}$.
%But then  $\widehat{L}_k(\mu)$ is a $V_k(\g)$-module by Theorem \ref{Th:Zhu-admissible},
%and hence $\mu\in \on{Pr}^k$ by Theorem \ref{Th:rationality-of-admissibleVA}.

Finally we prove surjectivity.
Let $[\lam] \in [ \on{Pr}^k_\circ]$.
By Lemma \ref{lem:represented-by},
we may choose a representative 
$\lam\in \on{Pr}^k_\circ\cap P_{0,+}$ of $[\lam]$. Clearly we have $\Delta(\lam)=\Delta_0$.
 By \cite[Proposition 2.8]{AraFutRam17} and its proof,
$$\on{Pr}^k=\bigcup\prin^k_{{yt_{-\eta}}},$$
where the union is take over pairs $(y,\eta)\in W\times \check{P}_+^{q}$
such that
\begin{align}
\text{%$\theta(\mu)\leq q$ and
$\left<\alpha, \eta\right> \geq 1$ for all $\alpha\in \Delta_+\cap y^{-1}(\Delta_-)$.}
\label{eq:des-ad-weights}
\end{align}
Let us therefore take such a pair $(y,\eta) \in W\times \check{P}_+^{q}$
satisfying
\eqref{eq:des-ad-weights}
such that
 $\lam\in \on{Pr}^k_{yt_{-\eta}}$.
Then we may write $\lam=y\circ \lam_1$ and $\lam_1=\lam_0-\frac{p}{q}\eta$,
where $\lam_1\in \on{Pr}^k_{t_{-\eta}}$ and
$\lam_0\in  \on{Pr}^k_{\Z}$. 
We have on the one hand
 $\Delta(\lam_1)=y^{-1}(\Delta(\lam))=y^{-1}(\Delta_0)$.
On the other hand $\Delta(\lam_1) = \Delta(\tfrac{p}{q}\eta)$ and, since $0 \leq \left<\alpha, \eta\right> \leq \left<\theta, \eta\right> \leq q-1 < h^\vee$ for all $\alpha \in \D_+$, we have $\left<\tfrac{p}{q}\eta, \alpha^\vee\right> \in \Z$ if and only if $\left<\eta, \alpha^\vee\right> = 0$. Hence $\Delta(\lam_1)=\Delta_\eta$. Therefore $\Delta_\eta\cong \Delta_0$, and we obtain
$[\lam]=[\lam_0-\frac{p}{q}\eta]$ where $(\lam_0,\mu)\in P_+^{p-n}\times \check{P}^{q}_{+,f}$. This establishes surjectivity as required.
\end{proof}

\begin{rem}
Let $f \in \mathfrak{sl}_n$ be a nilpotent element associated with the partition $Y=(p_1\leq p_2\leq \dots \leq p_r)$ of $n$. According to \cite{ElaKac05} the good gradings $\g=\bigoplus_{j\in (\tfrac{1}{2})\Z} \g_j$ for $f$ are in natural bijection with combinatorial structures known as \emph{pyramids} associated with $Y$. A pyramid is an arrangement of $n$ boxes of dimensions $1 \times 1$ into $r$ rows, the $i^{\text{th}}$ row consisting of a contiguous block of $p_i$ boxes. The projection to the horizontal axis of the $i^{\text{th}}$ row must be contained within the projection of the $(i+1)^{\text{th}}$ for each $i$, and the center of each box is required to lie directly above the center or else the edge of the box below.

The symmetric pyramid is the one in which the centers of all rows lie on a single vertical line.

The good grading associated with a pyramid is obtained by labeling its boxes $1, 2, \ldots, n$, letting $\on{col}(i)$ and $\on{row}(i)$ denote the horizontal and vertical coordinates of the $i^{\text{th}}$ box, respectively, and declaring $\on{deg}(e_{ij})=\on{col}(j)-\on{col}(i)$. The nilpotent $f = f_Y$ is recovered as $\sum e_{j,i}$ where the sum runs over $(i, j)$ satisfying $\on{row}(i)=\on{row}(j)$ and $\on{col}(i)=\on{col}(j)-1$.

The grading is even if the center of each box lies above the centers (and not the edges) of other boxes. Throughout this section we have worked with the even grading associated with the Young diagram, in which all rows are left justified. In particular the element $x_0$ given in \eqref{eq:x0} corresponds to this grading. All statements in this section hold upon replacing this grading with any other even grading. In particular if $f$ is even, so that the Dynkin grading is a good even grading, the results of this section apply to the Dynkin grading. This will be important for applications to the computation of fusion rules in subsequent sections, for instance by Proposition \ref{prop:An.selfdual} the subregular $\W$-algebra of type $A$ is only self-dual for the Hamiltonian reduction associated with the Dynkin grading.
\end{rem}

\section{Rationality of Simply Laced Subregular \texorpdfstring{$\W$}{W}-Algebras}

Let $\fsubreg$ be a subregular nilpotent element of $\fing$. We recall that $\Osubreg = G\cdot\fsubreg$ is by definition the unique nilpotent orbit of $\fing$ of dimension $\dim \fing-\on{rank}\fing-2$.
The corresponding partition or the Bala-Carter label of
 $\fsubreg$ can be found in the third column of Table \ref{table:Subregular.denom} below.

\begin{lemma}\label{Lem:condition1}
Let $\fing$ be a \textup{(}not necessarily simply laced\textup{)} simple Lie algebra, and let $k$ be an admissible number with denominator $q$ such that $\mathbb{O}_{q}={\Osubreg}$. For $\lam\in \prin^k$ we have
\begin{enumerate}
\item $\Delta(\lam)$ is nonempty,

\item $\lam \in \prin^k_\circ = \prin^k_{\Osubreg}$ if and only if $|\Delta(\lam)|=2$.
\end{enumerate}
 \end{lemma}
 \begin{proof}
If $\Delta(\lam)$ were empty then we would have $L(\lam)=M( \lam)$ and so $\var(J_{ \lam})=\mc{N} \supsetneq \overline{\mathbb{O}_q}$, which contradicts \eqref{eq:dec-ad-weights}. This proves the first part. The second part follows from Theorem \ref{thm:Joseph}.
 \end{proof}

By \cite{Ara09b} the condition $\mathbb{O}_q=\Osubreg$ holds for precisely those values of $q$ listed in the following table.
\begin{table}[h]%[tp]%
\caption{Subregular Denominators} \label{table:Subregular.denom}\centering% 
\begin{tabular}{|l|c|c|l|}
\hline
Type &$h^\vee$ &$\fsubreg$   & $q$ \\
\hline
$A_n$ &$n+1$ & $[n,1]$ &  $n$ \\
$B_n$ &$2n-1$ & $[2n-1,1^2]$  & $2n-1, 2n$ \\
$C_n (n \geq 3)$ & $n+1$ & $[2n-2,2]$ & $2n-1, 4n-6, 4n-4$ \\
$D_n$ &$2n-2$ &$[2n-3,3]$  & $2n-4, 2n-3$ \\
$E_6$ &$12$ &$E_6(a_1)$  & $9, 10, 11$ \\
$E_7$ &$18$ &$E_7(a_1)$ & $14, 15, 16, 17$ \\
$E_8$ & $30$&$E_8(a_1)$  & $24, 25, 26, 27, 28, 29$ \\
$F_4$ & $9$ &$F_4(a_1)$ & $9, 11, 12, 14, 16$ \\
$G_2$&$4$ &$G_2(a_1)$  & $4, 5, 6, 9$ \\
\hline
\end{tabular}
\end{table}

\begin{rem}
There are typos in \cite[Tables 6, 7]{Ara09b}. The central charge of $\W_k(\fing,{\fsubreg})$ for type $F_4$ at level $k=-h^\vee + p/q$ should read
\begin{align*}
-\frac{6(12p-13q)(5p-6q)}{pq}.
\end{align*}
\end{rem}

 In the rest of this section {\em we assume $\fing$ is of simply laced type} and that $k = -h^\vee + p/q$ is an admissible number with denominator $q$ such that $\mathbb{O}_{q}={\Osubreg}$. Let $\fsubreg \in \Osubreg$.

We recall that if $\fing$ is of type $D$ or $E$ then
$f_{\on{subreg}}$ is distinguished,
and thus
 there is exactly one even grading with respect to which $\fsubreg$ is good, namely its Dynkin grading. Up to conjugacy this grading is given by $\fing_0 = \finh + \fing_{\alstar} + \fing_{-\alstar}$ and $\fing_1 \supset  \bigoplus_{\alpha \in \Pi \backslash \{\alpha_*\}} \fing_\al$, where $\alstar$ is the simple root corresponding to the trivalent node in the Dynkin diagram of $\fing$. If $\fing$ is  of type $A$ then we fix a simple root $\alstar$ arbitrarily and define $\fing_0$ and $\fing_1$ as above. This defines $n$ distinct good even gradings on $\fing$, all of subregular type. If $n$ is odd then one of these gradings is Dynkin, if $n$ is even then none of them are. We have $\D_{0, +} = \{\alstar\}$. Let
\begin{align*}
x_0=\sum_{i \neq *} \varpi_i
\end{align*}
denote the grading element: $\fing_j = \{x \in \fing | [x_0, x] = jx \}$, and let $h = 2x_0$.

\begin{lemma}\label{Lem:eq-class}
Let $\fing$ be simply laced. Then every class of $[\prin_{\circ}^k]$ contains a representative $\lam\in \prin^k_{\circ}$ such that $\Delta(\lam)_+=\{\alstar\}$.
\end{lemma}
\begin{proof}
Let $\lam\in \prin^k_{\circ}$. By Lemma \ref{Lem:condition1}, $\Delta(\lam)_+=\{\alpha\}$ for some $\alpha\in \Delta_+$. Choose $w \in W$ such that $\alpha=w(\alstar)$. Then $w\circ \lam\in \prin^k$ and $\Delta(w\circ \lam)_+=\{\alstar\}$ as required.
\end{proof}

\begin{thm}\label{Th:rationality}
Let $\fing$ be simply laced and let $k$ be an admissible number for $\fing$ with denominator $q$ such that $\mathbb{O}_q=\Osubreg$. Then $\W_k(\fing,f_{subreg})$ is rational and the complete set of isomorphism classes of irreducible $\W_k(\fing,f_{subreg})$-modules is
\begin{align*}
\{\mathbf{L}(E_{J_{\lam}})\mid [\lam]\in [\prin^k_{\circ}]\},
\end{align*}
where $E_{J_{\lam}}$ is the unique simple module of $\HDS^0(U(\fing)/J_{\lam})$.
\end{thm}

\begin{proof}
By Lemma \ref{Lem:eq-class} the conditions of Theorem \ref{Th:simple-special-cases} are satisfied. Thus the assertion follows immediately from that theorem.
\end{proof}

We give a more explicit description of the set $[\prin^k_{\circ}]$ in the subregular case.
\begin{defn}
%Denote by $\check{P}_{+}^q$ the set of dominant integral coweights of $\affg$ of level $q$, and 
Put
\begin{align*}
\cpps^q = \left\{\eta \in \check{P}_+^q \,\, | \,\, 
\text{$\left<\al_i, \eta+qD\right> = 0$ for exactly one $i \in \{0,1,\ldots,\ell\}$}\right\},
\end{align*}
where $\al_0, \ldots, \al_\ell \in \widehat{\Pi}$ are the simple roots of $\affg$.
\end{defn}
The finite group $\eW_+$ acts on the set $P_+^{p-h^{\vee}}\times \cpps^q$ by $\pi_j(\lam,\eta)=(\bar \pi_j\lam+(p-h^{\vee})\varpi_j,\bar \pi_j \eta+q\varpi_j^{\vee})$ for $j\in J$.

\begin{thm}\label{thm:Tomoyuki.thm.8.5}
Suppose $\fing$ is of simply laced type. For each $\eta \in \cpps^q$ there exists $y_{\eta} \in \finW$ such that $y_{\eta}t_{-\eta}(\widehat{\Delta}(k\Lam_0))\subset \widehat{\Delta}_+^{\text{re}}$. Moreover, there is a well-defined bijection
\begin{align*}
\frac{ P_+^{p-h^{\vee}}\times \cpps^q }{\eW_+} \isomap [\prin^k_\circ], \quad (\lam,\eta)\mapsto %[\overline{y_{\eta}t_{-\eta}\circ \widehat{\lam}}]=
[y_{\eta}\circ (\lam-\frac{p}{q}\eta)].
\end{align*}
\end{thm}

\begin{proof}
Let $\lam\in \prin^{k}_{\circ}\cap
\prin^k_{\widehat{y}}$ where $\widehat{y}=y t_{-\eta} \in \eW$ with $y\in W$ and $\eta\in P^{\vee}$. It is straightforward to see that the condition $\widehat y(\widehat{\Delta}(k\Lam_0))\subset \widehat{\Delta}_+^{\text{re}}$ is equivalent to the following: that
\begin{align}
\begin{cases}
0\leq \alpha(\eta)\leq q-1 & \text{for all $\al \in \D_+$ such that $y(\alpha)\in \Delta_+$},\\
1\leq \alpha(\eta)\leq  q  & \text{for all $\al \in \D_+$ such that $-y(\alpha)\in \Delta_+$}.
\end{cases}
\end{align}
It is also clear that 
\begin{align*}
\Delta(\lam)_+=\{y(\alpha)\mid  \alpha\in \Delta_+,\ \alpha(\eta)=0\}\sqcup \{-y(\alpha)
\mid \alpha\in \Delta_+,\ \alpha(\eta)=q\}.
\end{align*}
Therefore the condition $|\Delta(\lam)_+|=1$ implies that $\eta$ satisfies one of the following two conditions:
\begin{enumerate}
\item[(i)] $0\leq \alpha(\eta)<q$ for all $\alpha\in \Delta_+$ and there exists a unique simple root $\alpha_i$  of $\fing$ such that $\alpha_i(\eta)=0$,

\item[(ii)] $0< \alpha(\eta)\leq q$ for all $\alpha\in \Delta_+$ and 
$\alpha(\eta)=q$ if and only if $\alpha=\theta$.
\end{enumerate}
But this is equivalent to the statement $\eta \in \cpps^q$.

Now let $ \eta \in \cpps^q$. For $y\in W$, we have $y t_{-\eta}\circ P^{p-h^{\vee}}_+ \subset \prin^k_\circ$ if and only if
\begin{align}\label{eq:y.mu.condition}
y t_{-\eta}(\widehat{\Delta}(k\Lam_0))\subset \widehat{\Delta}_+^{\text{re}}.
\end{align}
If $\eta$ satisfies (i) above then we may take $y=1$ and condition (\ref{eq:y.mu.condition}) is satisfied. If $\eta$ satisfies (ii) above then we may take $y=w^\circ$ the longest element of $W$, and condition (\ref{eq:y.mu.condition}) is again satisfied.

Finally let $\eta, \eta' \in \cpps^q$ and suppose that $\eta=\pi_j(\eta')$ for some $j\in J$. Then
\begin{align*}
(y t_{-\eta})(t_{q\varpi_j}\bar \pi_j)
= y \bar \pi_j t_{-{\pi_j^{-1}(\eta)}}
= y \bar \pi_jt_{-\eta'}.
\end{align*}
The assertion now follows from \eqref{eq:adm-wts-e}.

\end{proof}

We now record some properties of the subregular rational $\W$-algebras.
\begin{table}[h]%[tp]%
\caption{Subregular $W$-Algebras in Simply Laced Types} \label{Subregular}\centering% 
\begin{tabular}{|c| c |c|}
\hline
Type & Conformal weights of generators & $c(p/q)$\\
\hline
$A_{n}$ ($n$ odd) & {$1, 2, \dots, \tfrac{n-1}{2}, (\tfrac{n+1}{2})^3, \tfrac{n+3}{2} \dots, n$} & \\
$D_4$ &$2^3,3,4^2$ &$ -\frac{6 (4 p-7 q) (3 p-4 q)}{p q}$  \\
$D_n$ ($n$ odd) &$2^2,4,6,8,\dots,n-3,n-2,(n-1)^2,n,n+1,n+3,\dots,2n-4$ & \\
$D_n$ ($n$ even, $n\geq 6)$ &$2^2,4,6,8,\dots,n-4,(n-2)^2,n^2,n+2,n+4,\dots,2n-4$ & \\
$E_6$ &${2, 3, 4, 5, 6^2, 8, 9}$&$-\frac{8 (9 p-13 q) (7 p-9 q)}{p q}$ \\
$E_7$&$2,4,6^2,8,9,10,12,14$&$-\frac{9(14p-19q)(11p-14q)}{pq}$ \\
$E_8$&$2,6,8,10,12,14,15,18,20,24$ &$-\frac{10 (24 p-31 q) (19 p-24 q)}{p q}$ \\
\hline
\end{tabular}
\end{table}

The conformal structure, and in particular the central charge, of the subregular $W$-algebra $\W^k(\fing,{\fsubreg})$ of type $\fing = A_n$ depends of the choice of good grading. In Table \ref{Subregular} we list conformal weights of generators relative to the Dynkin grading, which corresponds to $x_0 = \rho^\vee - \varpi_m^\vee$, where $n=2m+1$. The central charge at level $k=-h^\vee+p/q$ is 
\[
c = -n(n^2-1) \frac{p}{q} - (1+n+3n^2+2n^3) - n (2 + 3n + n^2) \frac{q}{p}.
\]
The central charge of the subregular $W$-algebra $\W^k(\fing,{\fsubreg})$ of type $\fing = D_n$ at $k=-h^\vee+p/q$ is given by
\begin{align*}
c = -2n(13-9n+2n^2)\frac{p}{q} + (26+17n-24n^2+8n^3) - 2n(1-3n+2n^2)\frac{q}{p}.
\end{align*}

\begin{rem}
If $\g$ is of type $A$ and $q=h^\vee-1$ then $\cpps^q =\check{P}_{+,f_{\on{subreg}}}^{q} \sqcup \{\rho^{\vee}\}$ and Theorem \ref{thm:Tomoyuki.thm.8.5} agrees with 
Theorem \ref{thm:weight.param}. Furthermore the action of $\eW_+$ on $\cpps^q$ is simply transitive and the bijection becomes
\begin{align*}
P_+^{p-h^\vee}\isomap [\prin^k_\circ], \quad \lam\mapsto \lam-\frac{p}{q}x_0.
\end{align*}
See also Lemma \ref{cor:Smat.An} below.
\end{rem}

\section{Characters of Admissible Highest Weight Modules}

For a weight $\widehat \la=\lam+k\Lam_0 \in \affh^*$ 
of level $k$ we write $\chi_{\la}$ for the formal character $\sum_{\mu \in \affh^*} \dim{\affL_k(\la)_\mu} e^\mu$ of the irreducible $\affg$-module $\affL_k(\la) = \bigoplus_{\mu \in \affh^*} \affL_k(\la)_{\mu}$. For $\la \in \prin^k$ the formula
\begin{align}\label{eq:gen.adm.char}
\chi_{\la} = \frac{1}{R} \sum_{w \in \affW(\widehat \la)} \eps(w) e^{w \circ \widehat \la}
\end{align}
was proved by Kac and Wakimoto \cite{KWPNAS} and used to deduce modular properties of the set of characters of modules of admissible highest weight. Here $R$ is the Weyl-Kac denominator for $\affg$, and $\affW(\la) = \left< r_\al | \al \in \Pi^\vee(\la) \right>$ is the integral Weyl group of $\la$. We may consider $\chi_\la$ as a meromorphic function of $(\tau, z) \in \HH \times \finh$, more precisely $\chi_\la(\tau, z) = \left<\chi_\la, -2\pi i (\tau \La_0 - z) \right>$.

Let us now take $k = -h^\vee + p/q$ and assume that $\fing$ is simply laced for convenience. In particular all admissible weights are principal. For $\mu \in P \pmod{pq Q}$ we write
\begin{align*}
\Theta_\mu(\tau, z) = \sum_{\al \in \frac{\mu}{pq} + Q} e^{\pi i pq \tau (\al, \al)} e^{2\pi i pq (\al, z)}.
\end{align*}
Up to a change of variable these are exactly the theta functions associated with the discriminant form $L^\vee / L = P/pqQ$ of the integral lattice $L = Q(pq)$.

Let $\la \in \prin^k$ and let $(y, \eta, \nu)$ be a triple associated with $\la$ as in (\ref{eq:adm.wt.param}) and subsequent remarks, and let $\beta = -y(\eta)$ so that $y t_{-\eta} = t_{\beta} y$. The function
\begin{align*}
B_\la(\tau, z) = \eps(y) \sum_{w \in \finW} \eps(w) \Theta_{q w(\nu) + p \beta}(\tau, z/q),
\end{align*}
depends on $\la$ and not the choice of triple $(y, \eta, \nu)$. A linear change of coordinates identifies the sum over $\affW(\widehat\la)$ in (\ref{eq:gen.adm.char}) with a sum over $\finW$ of theta functions, and in this way one obtains $\chi_\la = B_\la / R$. Modular properties of theta functions now yield the following result.
\begin{prop}[{\cite{KWPNAS}}]\label{KWPNASprop}
The set of functions $B_\la$, as $\la$ ranges over $\prin^k$, is $SL_2(\Z)$-invariant. Furthermore
\begin{align}\label{eq:B.trans}
B_\la\left( -1/\tau, z/\tau \right) &= (-i\tau)^{\ell/2} e^{\pi i (k+h^\vee) (z, z)/\tau} \sum_{\la' \in \prin^k} a^B(\la, \la') \chi_{\la'}(\tau, z),
\end{align}
where
\begin{align}\label{eq:a.formula}
a^B(\la, \la')
= \frac{1}{|P/pqQ|^{1/2}} \cdot \eps(y) \eps(y') e^{-2\pi i [(\nu, \beta') + (\nu', \beta)]} e^{-2\pi i \frac{p}{q}(\beta, \beta')} \sum_{w \in \finW} \eps(w) e^{-2\pi i \frac{q}{p} (w(\nu), \nu')},
\end{align}
where $(y, \eta, \nu)$ is a triple associated with $\la$ as above, $\beta = -y(\eta)$, and $(y', \eta', \nu')$ and $\beta'$ are defined similarly.
\end{prop}

\newcommand{\BigTheta}{\Theta}
\newcommand{\prodchar}{\Psi}
\newcommand{\Wsr}{W^{\text{sr}}}

In \cite{KWPNAS} the modular transformations of the $\chi_\la$ are given by coefficients denoted $a$. These are related to the $a^B$ of (\ref{eq:B.trans}) by $a = i^{|\ov\D_+|} a^B$.

\begin{lemma}\label{lem:transfer.w}
Let $\la \in \prin^k$ and $w \in \finW$ such that $w(\Pi^\vee(\la)) = \Pi^\vee(w \circ \la)$. Then
\begin{align}\label{eq:w.invar}
B_{w \circ \la}(\tau, z) = B_\la(\tau, w^{-1}(z)).
\end{align}
In particular if $\la$ is regular then (\ref{eq:w.invar}) holds for all $w \in \finW$, and if $\D_+(\la) = \{\alstar\}$ then (\ref{eq:w.invar}) holds for all $w \in \finW$ such that $w(\alstar) \in \D_+^\vee$.
\end{lemma}

\begin{proof}
We write $\la = \wh{y} \phi(\nu) - \rho$ where $\Pi^\vee(\la) = \wh{y}(\Pi^\vee_{(q)}) \subset \D^{\vee, \text{re}}_+$, and $\wh y = t_{\beta} y$. By assumption $\Pi^\vee(w \circ \la) = w \wh y(\Pi^\vee_{(q)}) \subset \D^{\vee, \text{re}}_+$, and therefore
\begin{align*}
t_{w(\beta)} (wy) \phi(\nu) - \wh\rho = w \wh{y} \phi(\nu) - \wh\rho = w \circ \la,
\end{align*}
so $( w y, \nu,  w(\beta))$ is an admissible triple. In general $\Theta_{w(\mu)}(\tau, z) = \Theta(\tau, w^{-1}(z))$, and one immediately deduces (\ref{eq:w.invar}).
\end{proof}

Let us write
\begin{align*}
\BigTheta(\tau, v) = \frac{\theta_{1,1}(\tau, -v)}{\eta(\tau)} = q^{1/12} e^{-\pi i v} \prod_{n=1}^\infty (1 - e^{-2\pi i v}q^{n-1})(1 - e^{+2\pi i v} q^{n}).
\end{align*}
Then $\BigTheta$ obeys
\begin{align*}
\BigTheta(-1/\tau, v/\tau) = -i e^{\pi i v^2/\tau} \BigTheta(\tau, v).
\end{align*}

Now let $f$ be a subregular nilpotent element in $\fing$. We fix a good even grading for $f$, and we consider the characters of Hamiltonian reductions $H_{f, -}^0(L(\la))$, i.e., of the cohomology of (\ref{eq:C.minus.def}). By the Euler-Poincar\'{e} principle the supercharacter of $H_{f, -}^\bullet(M)$ equals that of $C_-^\bullet(M) = M \otimes \textstyle{\bigwedge^{\frac{\infty}{2}+\bullet}} $. For $\la \in \prin^k$ we write
\begin{align*}
\prodchar_\la(\tau, z \mid u) = \str_{C_-^\bullet(\affL_k(\la))} u_0 e^{2\pi i (z_0 - (x_0, z))} q^{L_0-c/24}
\end{align*}
where $u$ is a $d$-closed element of $C_-^\bullet(V_k(\fing))$. The central charge of $\textstyle{\bigwedge} = \textstyle{\bigwedge^{\frac{\infty}{2}+\bullet}}$ is $-2|\D_{>0}|$ and
\begin{align*}
\str_{\textstyle{\bigwedge^{\frac{\infty}{2}+\bullet}}} e^{2\pi i (F_0^z + (x_0, z))} q^{L_0+|\D_{>0}|/12} = \prod_{\al \in \D_{>0}}\BigTheta(\tau, \al(x)).
\end{align*}
The modular transformations of the $\prodchar_{\la}$ may be derived from Theorem 6.4 of \cite{AvE}.
\begin{prop}\label{prop:mod.inv.BRSTcomplex}
For all $\la \in \prin^k$ we have
\begin{align*}
\prodchar_\la\left( \frac{-1}{\tau}, \frac{z}{\tau} \mid \tau^{-L_{[0]}} \exp\left[\frac{1}{\tau} \sum_{n > 0}\frac{(-1)^{n}}{n} z_{(n)} \right]u \right) = e^{\pi i (k+h^\vee) |z^2|/\tau} \sum_{\la' \in \prin^k} i \cdot a^B(\la, \la') \,\, \prodchar_{\la'}(\tau, z \mid u).
\end{align*}
\end{prop}
In fact the proof of this proposition is the same as that of {\cite[Theorem 8.1]{AvE}}, which deals with the case of $f$ a principal nilpotent element. The only difference is that $\dim(\fing_{>0})$ changes from $|\D_+|$ to $|\D_+|-1$ which causes the factor of $i$ to appear in Proposition \ref{prop:mod.inv.BRSTcomplex} above.

For $\la \in \prin^k \cap P_{0,+}$ the Hamiltonian reduction $H^0_{f,-}(\affL_k(\la))$ is an irreducible $\W_k(\fing, f)$-module. Let us define
\begin{align*}
\psi_\lam(\tau \mid u) = \tr_{H^0_{f,-}(\affL_k(\la))} u_0 q^{L_0 - c/24}.
\end{align*}
By Theorem \ref{Th:simplicity} part (1) we have
\begin{align}\label{eq:tr.func.W}
\psi_\lam(\tau \mid u) = \lim_{z \rightarrow 0} \prodchar_\la(\tau, z \mid u).
\end{align}

From now on we assume that $\mbo_q = \mbo_{\text{subreg}}$, i.e., that $q$ is one of the denominators listed in Table \ref{table:Subregular.denom}. In particular $|\D_0| = 2$ and we write $\alstar$ for the unique element of $\D_{0, +}$. Now we define
\begin{align*}
\Wsr = \{w \in W \mid w(\alstar) \in \D_+\}.
\end{align*}
Since $\fing$ is simply laced we may fix a set $\mathbb{X} \subset P_{0,+}$ of representatives of $[\prin^k_\circ]$, so that $\prin^k = \{ y \circ \la | \text{$\la \in \mathbb{X}$ and $ y \in \Wsr$}\}$. Since $\W_k(\fing, f)$ is rational and lisse, by Theorem \ref{Th:simple-special-cases}, the set of trace functions (\ref{eq:tr.func.W}), as $\la$ ranges over $\mathbb{X}$, is modular invariant.

We now compute the $S$-matrix of $\W_k(\fing, f)$ using Proposition \ref{prop:mod.inv.BRSTcomplex}. Since the restrictions $\prodchar_\la(\tau, z \mid \mathbf{1})$ are linearly independent it suffices to work with them. The Weyl denominator $R(\tau, x)$ is essentially $\prod_{\al \in \D_+}\BigTheta(\tau, \al(x))$ and so, since $\D_{>0} = \D_+ \backslash \{\alstar\}$, we have
\begin{align*}
\prodchar_\la(\tau, z \mid \mathbf{1}) = \chi_\la(\tau, z) \cdot \prod_{\al \in \D_+ \backslash \alstar} \BigTheta(\tau, \al(z)) = \frac{1}{\eta(\tau)^\ell} \cdot \frac{B_\la(\tau, z)}{\BigTheta(\tau, \alstar(z))}.
\end{align*}
Note that $\BigTheta(\tau, \alstar(z))$ has a zero along the hyperplane $\alstar(z) = 0$. If $\D_+(\la) = \{\gamma\}$ then $B_\la(\tau, z)$ has a zero along the hyperplane $\gamma(z) = 0$. So unless $\la \in P_{0, +}$ the function $\prodchar_\la(\tau, z \mid \mathbf{1})$ has an indeterminate value at $z=0$. We make an arbitrary choice of $x \in \h^*$ not orthogonal to $\alstar$ and we put $z = tx$. The limit in (\ref{eq:tr.func.W}) becomes a limit as $t \rightarrow 0$. We then apply l'H\^{o}pital's rule to the formula of Theorem \ref{prop:mod.inv.BRSTcomplex}. Since for $\la \in P_{0,+}$ we have $\gamma = \alstar$ and hence $\prodchar_\la(\tau, z \mid \mathbf{1})$ regular, the final result does not depend on the auxiliary parameter $x$.

Fix $x \in \finh$ not orthogonal to $\alstar$, and let $\la \in \mathbb{X}$. We have
\begin{align*}
\lim_{t \rightarrow 0} \frac{1}{\eta(-1/\tau)^\ell} \cdot \frac{B_\la(-1/\tau, tx/\tau)}{\BigTheta(-1/\tau, t\alstar(x)/\tau)}
&= \frac{1}{(-i\tau)^{\ell/2} \eta(\tau)^\ell} \cdot  \lim_{t \rightarrow 0} \frac{(-i\tau)^{\ell/2} e^{\frac{\pi i t^2}{\tau} (k+h^\vee) |x|^2} \sum_{\xi \in \prin^k} a^B(\la, \xi) B_{\xi}(\tau, tx)}{(-i) e^{\frac{\pi i t^2}{\tau}|\alstar(x)|^2}\BigTheta(\tau, t\alstar(x))} \\
&= \frac{i}{\eta(\tau)^\ell} \cdot \lim_{t \rightarrow 0} e^{\frac{\pi i t^2}{\tau}\left[(k+h^\vee)|x|^2-|\alstar(x)|^2\right]} \sum_{\xi \in \prin^k} a^B(\la, \xi) \frac{B_{\xi}(\tau, tx)}{\BigTheta(\tau, t\alstar(x))}.
\end{align*}
The exponential factor tends to $1$ in the limit. It follows from Lemma \ref{lem:transfer.w} that
\begin{align*}
\sum_{\xi \in \prin^k} a^B(\la, \xi) \frac{B_{\xi}(\tau, tx)}{\BigTheta(\tau, t\alstar(x))}
&= \sum_{\la' \in \mathbb{X}} \left[ \sum_{ y \in \Wsr} a^B(\la,  y\circ \la') \frac{B_{\la'}(\tau, t  y^{-1}(x))}{\BigTheta(\tau, t\alstar(x))} \right].
\end{align*}
Using the product formula for $\BigTheta$, l'H\^{o}pital's rule, and the fact that $\la' \in P_{0,+}$, we compute
\begin{align*}
\lim_{t \rightarrow 0} \frac{B_{\la'}(\tau, t  y^{-1}(x))}{\BigTheta(\tau, t\alstar(x))}
= \lim_{t \rightarrow 0} \frac{1-e^{-t\alstar( y^{-1}(x))}}{1-e^{-t\alstar(x)}} \cdot \lim_{t \rightarrow 0} \frac{B_{\la'}(\tau, t  y^{-1}(x))}{\BigTheta(\tau, t\alstar( y^{-1}(x)))}
= \frac{\alstar( y^{-1}(x))}{\alstar(x)} \cdot \lim_{t \rightarrow 0} \frac{B_{\la'}(\tau, tx)}{\BigTheta(\tau, t\alstar(x))}.
\end{align*}

Thus we have proved the following theorem.
\begin{thm}\label{thm:S.mat.red}
Let $f$ be a subregular nilpotent element of the simply laced simple Lie algebra $\fing$, and let $k = -h^\vee + p/q$ with denominator $q$ such that $\mbo_q = \mbo_{\text{subreg}}$. Let $\mathbb{X}$ be a set of representatives of $[\prin^k_\circ]$ in $P_{0, +}$. For all $\la \in \mathbb{X}$ we have
\begin{align*}
\psi_\la(-1/\tau \mid \tau^{-L_{[0]}}u) = \sum_{\la' \in \mathbb{X}} S_{\la, \la'} \psi_{\la'}(\tau \mid u),
\end{align*}
where
\begin{align}\label{eq:ov.a.formula}
S_{\la, \la'} = i \sum_{ y \in \Wsr} \frac{\left< y(\alstar), x\right>}{\left<\alstar, x\right>} a^B(\la,  y \circ \la').
\end{align}
\end{thm}
% If we used $a$ instead of $a^B$ there would appear a $i^{1-|\D_+|}$.

A remarkable feature of the sum appearing in formula (\ref{eq:ov.a.formula}) is its independence of $x$. Of course this follows \emph{a posteriori} but clearly an elementary proof would be desirable. Having fixed $\beta, \beta'$ let us put
\begin{align*}
r(x) = \sum_{y(\alstar) \in \D_+^\vee} \eps(y) \frac{\left<y(\alstar), x\right>}{\left<\alstar, x\right>} e^{-\frac{2\pi i}{q} (\beta, y(\beta'))}.
\end{align*}
Let $\kappa_x$ denote the $1$-form of pairing with $x$, i.e., $\left<x, -\right>$. We compute the gradient of $f$ to be
\begin{align*}
\nabla r(x) = \frac{1}{\left<\alstar, x\right>^2} \iota_{\kappa_x}(\alstar \wedge \xi), \quad \text{where} \quad \xi = \sum_{y(\alstar) \in \D_+^\vee} \eps(y) e^{-\frac{2\pi i}{q} (\beta, y(\beta'))} y(\alstar).
\end{align*}
To show that $r$ is independent of $x$, it therefore suffices to show that $\xi$ is proportional to $\alstar$. This may be established by a direct computation for any fixed root system, though we do not know a uniform proof.

\section{Modular Transformations of Simple Affine Vertex Algebras}

%\section{Some lemmas on $\eW_+$ and $S$-matrices}

In this section we collect some results on $S$-matrices and fusion rules of the rational vertex algebra $V_{p-h^\vee}(\fing)$ where $p \geq h^\vee$.
\begin{prop}[{\cite{KP84}}]
Let $\fing$ be a simple Lie algebra and $p \geq h^\vee$ an integer. The irreducible $V_{p-h^\vee}(\fing)$-modules are precisely the irreducible highest weight $\affg$-modules $L(\widehat\kappa)$ where $\widehat\kappa = \kappa + (p-h^\vee)\Lambda_0$ and $\kappa$ runs over $P^{p-h^\vee}_+$. The span of the characters $\chi_\kappa(\tau, z)$ is modular invariant. In particular
\begin{align*}
\chi_\kappa\left(-1/\tau, z/\tau\right) = \sum_{\kappa' \in P_+^{p-h^\vee}} K^{\kappa, \kappa'} \chi_{\kappa'}(\tau, z),
\end{align*}
where
\begin{align}\label{eq:Smatrix.WZW}
K^{\kappa, \kappa'} = \frac{i^{|\D_+|}}{|P/pQ|^{1/2}} \sum_{ w \in \finW} \eps( w) e^{-\frac{2\pi i}{p} ( w(\kappa+\rho),\kappa+\rho)}.
\end{align}
\end{prop}

We recall from {\cite[Equation (4.2.6)]{FKW}} the relation
\begin{align}\label{eq:FKW4.2.6}
K^{\ov{w(\widehat\kappa)}, \kappa'} = e^{-2\pi i \left( \ov{w(\La_0)}, \kappa' \right)} K^{\kappa, \kappa'},
\end{align}
for all $w \in \eW_+$.

We denote by $K_{p}$ the matrix of coefficients $K^{\kappa, \kappa'}$ given by (\ref{eq:Smatrix.WZW}) as $\kappa$, $\kappa'$ run over $P_+^{p-h^\vee}$. We also denote by $K_p^{\text{int}}$ the submatrix obtained by restricting $\kappa, \kappa' \in Q$ and by $K_p^{\Z}$ the submatrix obtained by restricting $\kappa, \kappa' \in -\rho + Q$.

\begin{lemma}\label{lemma:KZ.Kint}
The subgroup $\CF(V_{p-h}(\fing))^{\text{int}} \subset \CF(V_{p-h^\vee}(\fing))$, spanned by $[\mu]$ where $\mu$ runs over $Q$, is a subring closed under duality. Furthermore $\CF(V_{p-h^\vee}(\fing))^{\text{int}} \cong \CF(K_p^\Z)$.
\end{lemma}

\begin{proof}
Combining equation (\ref{eq:FKW4.2.6}) with the Verlinde formula yields
\begin{align*}
N_{\la, \mu}^\nu
% %
% &= \sum_{\xi} \frac{K_{\la, w(\xi)} K_{\mu, w(\xi)} \ov K_{\nu, w(\xi)}}{K_{0, w(\xi)}} \\
% %
% &= e^{-2\pi i(w(\La_0), \la+\mu-\nu)} \sum_{\xi} \frac{K_{\la, \xi} K_{\mu, \xi} \ov K_{\nu, \xi}}{K_{0, \xi}} \\
% %
&= e^{-2\pi i(w(\La_0), \la+\mu-\nu)} N_{\la, \mu}^\nu \quad \text{for all $w \in \eW_+$}.
\end{align*}
Since $w_j(\La_0) = \varpi_j + \La_0$ for $j \in J$, it follows that $N_{\la, \mu}^\nu = 0$ unless $(\la+\mu-\nu, \varpi_j) \in \Z$ for all $j \in J$. But the set of fundamental weights $\{\varpi_j\}_{j \in J}$ forms a system of representatives of the quotient $P/Q$. Therefore
\begin{align}\label{beautiful.condition}
N_{\la, \mu}^\nu = 0 \quad \text{unless} \quad \la+\mu-\nu \in Q.
\end{align}
This shows that $\CF(V_{p-h^\vee}(\fing))^{\text{int}}$ is closed under the fusion product. For duality it is well known that $L(\kappa)^\vee \cong L(-w^\circ(\kappa))$, and if $\kappa \in Q$ then so is $-w^\circ(\kappa)$. This proves the first part.

Let $s \in J$ be defined by the condition $\varpi_s - \rho \in Q$. Then for $\mu \in Q$ we have $w_s(\mu) \in -\rho+Q$. (Here we have used that $2\rho \in Q$.) For $\mu, \mu' \in Q$ we have, by (\ref{eq:FKW4.2.6}),
\begin{align*}
K^{w_s(\mu),w_s(\mu')} = e^{-2\pi i (\rho, \rho)} K^{\mu, \mu'}.
\end{align*}
This says that $K_p^\Z$ becomes proportional to $K_p^{\text{int}}$ upon conjugation by $w_s$. It therefore suffices to show that the fusion algebra determined by $K_p^{\text{int}}$ is isomorphic to $\CF(V_{p-h}(\fing))^{\text{int}}$. But this follows easily from equation (\ref{eq:FKW4.2.6}) and the Verlinde formula.

\end{proof}

\section{\texorpdfstring{$S$}{S}-Matrices of Subregular \texorpdfstring{$\W$}{W}-Algebras}\label{sec:Smat.subreg}

Let $\fing$ be a finite dimensional simple Lie algebra of simply laced type. In this section we compute the fusion rules of rational subregular $W$-algebras obtained via Hamiltonian reduction of $\fing$. Let $k = -h^\vee + p/q$ where $q$ is one of the denominators listed in Table \ref{table:Subregular.denom} and $p \geq h^\vee$ is coprime to $q$. The irreducible $\W_k(\fing, \fsubreg)$-modules are parametrised by $[\prin^k_\circ]$, and by Theorem \ref{thm:Tomoyuki.thm.8.5} we have a bijection
\begin{align*}
\frac{\ppr^p \times \cpps^q}{\eW_+} &\rightarrow [\prin^k_\circ],
\end{align*}
taking $(\nu, \eta)$ to the $\finW \circ (-)$-orbit of
\begin{align*}
t_{-\eta} \phi(\widehat\nu) - \widehat\rho = \left(\nu-(p/q)\eta-\rho\right) + k\La_0.
\end{align*}
The fusion rules are computed by applying the Verlinde formula to the $S$-matrix of $\W_k(\fing, \fsubreg)$. Let $\wh\la = y t_{-\eta} \phi(\widehat\nu) - \widehat\rho$ and $\wh\la' = y' t_{-\eta'} \phi(\widehat\nu') - \widehat\rho$. 
% then $\ov y \circ \la' = \ov y t_{-\ov\eta'} \phi(\nu') - \rho$ and
By Proposition \ref{KWPNASprop} and Theorem \ref{thm:S.mat.red} the $S$-matrix is given by
\begin{align*}
S_{\la, \la'} = i \cdot \frac{\eps(y) \eps(y')}{|P/pqQ|^{1/2}} \cdot \sum_{w \in \Wsr} \left[ \eps(w) \frac{\left<w(\alstar^\vee), x\right>}{\left<\alstar^\vee, x\right>} e^{-2\pi i [(\nu, w(\beta')) + (\nu', \beta)]} e^{-2\pi i \frac{p}{q}(\beta, \beta')} \sum_{u \in \finW} \eps(u) e^{-2\pi i \frac{q}{p} (u(\nu), \nu')} \right],
\end{align*}
where $\beta = -y(\eta)$ and $\beta' = -y'(\eta')$. We wish to reduce this to a product of sums over $w \in \Wsr$ and $u \in \finW$, and to do this it is desirable to eliminate the cross terms
\begin{align*}
e^{2\pi i [(\nu, w(\beta')) + (\nu', \beta)]}.
\end{align*}
If $q$ is coprime to $|J|$ then by Lemma \ref{lemma:I.choose.Q} we are able to choose, for each $\la \in [\prin^k_\circ]$, a representative $(\nu, \eta)$ for which $\eta \in Q$. Indeed in this case the action of $\eW_+$ on $\cpps^q$ is transitive and we obtain a bijection
\begin{align}\label{eq:q.coprime.bijec}
\ppr^p \times (\cpps^q \cap Q) \rightarrow [\prin^k_\circ].
\end{align}
If $q$ is not coprime to $|J|$ then $p$ is coprime to $|J|$ and we have a bijection
\begin{align}\label{eq:p.coprime.bijec}
(\ppr^p \cap Q) \times \cpps^q \rightarrow [\prin^k_\circ].
\end{align}
In either case, for each $\eta$, we fix an element $y^{(\eta)} \in \finW$ such that
\begin{align}\label{eq:yeta.condition}
\wh\la = y^{(\eta)} t_{-\eta} \phi(\widehat\nu) - \widehat\rho \in P_{0, +} \quad \text{for all $\widehat\nu \in \ppr^p$}.
\end{align}
The existence of such elements $y^{(\eta)} \in \finW$ is ensured by the following lemma.
\begin{lemma}
Let $\eta \in \cpps^q$, then there exists $y = y^{(\eta)} \in \finW$ such that
\begin{align*}
\left<\al_*^\vee, y(\nu - (p/q)\eta)\right> \in \Z_+,
\end{align*}
for all $p \in \Z_+$ and $\nu \in P_+^p$.
\end{lemma}

\begin{proof}
Let $\eta = \sum_{i=1}^\ell c_i \varpi_i \in \cpps^q$, then either $c_k=0$ for some $k \in \{1,2,\ldots,\ell\}$ or else $\left<\eta, \theta^\vee\right> \geq \left<\rho, \theta^\vee\right> = h^\vee-1$ and hence $q \geq h^\vee-1$. Since in fact $q \leq h^\vee-1$ we deduce that either $c_k=0$ for some $k \in \{1,2,\ldots,\ell\}$ or else $\eta = \rho$.

It clearly suffices to show that there exists $y \in \finW$ such that
\begin{align}\label{eq:pos}
\left<\al_*^\vee, y( \varpi_i - \tfrac{a_i^\vee}{q}\eta) \right> \in \Z_+, \quad i = 1,\ldots, \ell.
\end{align}
Here $a_i^\vee = \left<\theta^\vee, \varpi_i\right>$. If $c_k = 0$ then we have $\left<\eta, \al_k^\vee\right> = 0$ and $\left< \varpi_i - \tfrac{a_i^\vee}{q}\eta, \al_k^\vee \right> = \delta_{ik}$. On the other hand if $\eta = \rho$ then $q = h^\vee-1$ and $\left< \varpi_i - \tfrac{a_i^\vee}{q}\eta, \theta^\vee \right> = a_i^\vee - a_i^\vee = 0$. Let us now take $y \in \finW$ such that $y(\al_k^\vee) = \al_*^\vee$, respectively $y(\theta^\vee) = \al_*^\vee$. We obtain (\ref{eq:pos}) as required.
\end{proof}

\begin{thm}\label{thm:Smatrix.quasi-general}
Let $\fing$ be a simple Lie algebra of simply laced type and let $k = -h^\vee + p/q$ where $q$ is a subregular denominator for $\fing$ and $p \geq h^\vee$ is coprime to $q$. Let $\wh\la = y^{(\eta)} t_{-\eta} \phi(\widehat\nu) - \widehat\rho$ as in the preceding remarks, and similarly $\wh\la' = y^{(\eta')} t_{-\eta'} \phi(\widehat\nu') - \widehat\rho$. Put $\beta = -y^{(\eta)}(\eta)$ and $\beta' = -y^{(\eta')}(\eta')$. Then the $S$-matrix of $\W_k(\fing, f_{\text{sr}})$ is given by
\begin{align}\label{eq:Smatrix.quasi-general}
S_{\la, \la'} = i \cdot \frac{\eps(y^{(\eta)})\eps(y^{(\eta')})}{|P/pqQ|^{1/2}} \cdot \left( \sum_{w \in \Wsr} \eps(w) \frac{\left<w(\alstar^\vee), x\right>}{\left<\alstar^\vee, x\right>}  e^{-2\pi i \frac{p}{q}(\beta, w(\beta'))} \right) \left( \sum_{u \in \finW} \eps(u) e^{-2\pi i \frac{q}{p} (u(\nu), \nu')} \right).
\end{align}
\end{thm}

\begin{proof}
Having made the choices prescribed above, the result follows immediately from Theorem \ref{thm:S.mat.red}.
\end{proof}

Let $q$ be a subregular denominator for $\fing$, i.e., one of the denominators listed in Table \ref{table:Subregular.denom}, and $p \geq h^\vee$ coprime to $q$. We denote by $C_q$ the matrix of coefficients
\begin{align}\label{eq:Smatrix.quasi-general.constant}
C^{\eta, \eta'} = \sum_{w \in \Wsr} \eps(w) \frac{\left<w(\alstar^\vee), x\right>}{\left<\alstar^\vee, x\right>}  e^{-\frac{2\pi i}{q}(\beta, w(\beta'))},
\end{align}
where $\widehat\eta, \widehat\eta'$ run over $\cpps^q$, and by $C_q^\Z$ the submatrix obtained by restricting to $\eta, \eta' \in Q$. We denote by $S_{\text{sr}}^{p,q}$ the $S$-matrix of the rational vertex algebra $\W_{-h^\vee+p/q}(\fing, \fsubreg)$.
\begin{lemma}\label{lemma:form.denominators}
The weight lattice $P$ of $\fing$ satisfies $(P, P) \subset \frac{1}{|J|}\Z$, where $(\cdot, \cdot)$ is the invariant bilinear form on $\fing$ so normalised that $(\theta, \theta)=2$. If $\fing = D_n$ and $n$ is even then $(P, P) \subset \frac{1}{2}\Z$ in fact.
\end{lemma}

Let $\Q(\zeta_N)$ denote the cyclotomic field obtained by adjoining to $\Q$ a primitive $N^{\text{th}}$ root of unity $\zeta_N$. For an integer $a$ coprime to $N$ we denote by $\varphi_a \in \Gal(\Q(\zeta_N)/\Q)$ the automorphism defined by $\varphi_a(\zeta_N) = \zeta_N^a$. The Galois group $\Gal(\Q(\zeta_N)/\Q)$ is naturally isomorphic to $(\Z/N\Z)^\times$.

We denote the Kronecker product of matrices by $\otimes$, and proportionality of matrices up to a nonzero scalar by $\propor$. Now suppose $(q, |J|) = 1$. Then by Lemma \ref{lemma:form.denominators} the entries of $K_p$ lie in $\Q(\zeta_{p |J|})$, on which $\varphi_q$ acts as an automorphism. Meanwhile the entries of $C_q^\Z$ lie in $\Q(\zeta_q)$, on which $\varphi_p$ acts as an automorphism. Similar considerations apply in case $(p, |J|) = 1$ and Theorem \ref{thm:Smatrix.quasi-general} can now be summarised as
\begin{align}
S_{\text{sr}}^{p, q} &\propor \varphi_p(C_q^{\Z}) \otimes \varphi_q(K_p) \quad \text{if $(q, |J|) = 1$}, \label{eq:s.gal.etaQ} \\
\text{and} \quad S_{\text{sr}}^{p, q} &\propor \varphi_p(C_q) \otimes \varphi_q(K_p^{\Z}) \quad \text{if $(p, |J|) = 1$}. \label{eq:s.gal.nuQ}
\end{align}
In both cases we obtain a presentation of the fusion ring of $\W_k(\fing, \fsubreg)$ as the tensor product of the fusion ring of $V_{p-h^\vee}(\fing)$ (or its integral weight subalgebra) and the fusion ring associated with the $S$-matrix $C_q^{\Z}$ (or $C_q$).

\section{Fusion Rules of Subregular \texorpdfstring{$W$}{W}-Algebras in Type \texorpdfstring{$A$}{A}}

We now specialise the discussion of the preceding section to the type $A$ case.
We note that
subregular $W$-algebras of type $A$ were previously studied 
by Feigin and Semikhatov  \cite{FeiSem04}
as $W_{n}^{(2)}$-algebras 
and the isomorphism between them
was  established  by Genra \cite{Gen17}.

%So let  
%$\g=\mf{sl}_{n+1}$,
%$k = -h^\vee+p/q$ be admissible, with  $q = h^\vee-1=n$. 
%Observe that we have a bijection
%\begin{align*}
%P^{p-n}_+\isomap [\prin^k_{\circ}], \quad \lam\mapsto [\lam-\frac{p}{q}x_0]
%\end{align*}
%by Theorem \ref{thm:weight.param}
%or Theorem \ref{thm:Tomoyuki.thm.8.5}.
%It follows that 
%$\{\mathbf{L}(E_{J_\lam-\frac{p}{q}x_0})\mid \lam\in P_+^{p-n}\}$
% gives a complete set of simple $\W_k(\g,f_q)$-modules.
\begin{lemma}\label{cor:Smat.An}
Let $\fing$ be the simple Lie algebra of type $A_n$ and let $k = -h^\vee+p/q$ where $q = h^\vee-1=n$ and $p \geq h^\vee$ coprime to $q$. There exists a bijection between the sets of irreducible modules of $\W_{-h^\vee+p/q}(\fing, \fsubreg)$ and $V_{p-h^\vee}(\fing)$ which induces an equality $S^{p,q}_{\text{sr}} = \varphi_{q}(K_p)$.
\end{lemma}
\begin{proof}
The set $\cpps^q$ consists of $\rho$ and the weights $\rho-\varpi_i$ for $i = 1,\ldots, \ell$. Let us put $\eta = \rho$ if $n$ is even, and $\eta = \rho-\varpi_m$ if $n = 2m+1$ is odd. Then $\cpps^q \cap Q = \{\eta\}$. Let us fix a good even grading on $\fing$ with associated simple root $\alstar \in \Delta_{0, +}$ and let $\varpi_*$ be the fundamental weight corresponding to $\alstar$. Finally we fix $y = y^{(\eta)} \in W$ such that condition (\ref{eq:yeta.condition}) is satisfied, so that $\widehat{\la} = (y t_{-\eta}) \circ \widehat\mu$ defines a bijection $\mu \mapsto \lambda$ from $P_{+}^{p-h^\vee}$ to $[\prin_\circ^k]$. By Theorem \ref{thm:Smatrix.quasi-general} we now have
\begin{align}\label{eq:Smatrix.An}
(S_{\text{sr}}^{p, q})_{\la, \la'} = C \cdot \frac{i}{|P/pqQ|^{1/2}} \sum_{u \in \finW} \eps(u) e^{-2\pi i \frac{q}{p} (u(\mu+\rho), \mu'+\rho)},
\end{align}
where the factor $C = C^{\eta, \eta}$, given by formula (\ref{eq:Smatrix.quasi-general.constant}), is independent of $\mu, \mu'$. By Lemma \ref{lemma:form.denominators} the exponential sums of (\ref{eq:Smatrix.WZW}) and (\ref{eq:Smatrix.An}) lie in $\Q(\zeta_{ph^\vee})$. Since $q$ is coprime to $p$ and $h^\vee$ the latter is the conjugate of the former by $\varphi_q \in \Gal(\Q(\zeta_{ph^\vee})/\Q)$.
\end{proof}

\begin{thm}\label{thm:fusion.An}
The fusion rules of the {exceptional} subregular $\W$-algebras of type $A_n$, for $q = n$ odd, coincide with those of simple affine vertex algebras at positive integer level. More precisely the assignment
\begin{align*}
\affL_{p-h^\vee}(\lam) \mapsto 
\mathbf{L}(E_{\lam-\frac{p}{q}x_0})\cong H_{DS,f}^0(\widehat{L}_k(\lam)),
%H^0_{f, -}(\affL_k(\la)) \quad \text{where} \quad \widehat\la = t_{-x_0}\phi(\widehat\mu+\widehat\rho)-\widehat\rho
\end{align*}
induces an isomorphism
\begin{align*}
\CF(\W_{-h^\vee+p/q}(\fing, \fsubreg)) \cong \CF(V_{p-h^\vee}(\fing))
\end{align*}
of fusion rings.
\end{thm}

\begin{proof}
We let $n = 2m+1$ and take $\alstar = \al_m$. The subregular denominator is $q=n$. In the proof of Lemma \ref{cor:Smat.An} we have $\eta = x_0$ and since $\left<\alstar, x_0\right> = 0$ we may take $y = 1$. This gives the bijection between irreducible modules stated in the theorem and gives the $S$-matrix of $\W = \W_{-h^\vee+p/q}(\fing, \fsubreg)$ as $\varphi_{q}(K_p)$. Since $\W$ is self-contragredient by Proposition \ref{prop:An.selfdual} (as well as rational and lisse) the fusion rules can be computed via the Verlinde formula. The fusion rules and coefficients of the charge conjugation matrix (which specify duality) are integers and therefore invariant under the Galois group $\Gal(\Q(\zeta_{ph^\vee})/\Q)$. It follows immediately that the bijection 
{$\affL_{p-h^\vee}(\lam) \mapsto 
\mathbf{L}(E_{\lam-\frac{p}{q}x_0})$}
%$\mu \mapsto \la$ 
induces an isomorphism from the fusion ring of $V_{p-h^\vee}(\fing)$ to that of $\W$.
\end{proof}

\begin{rem}
The rationality of subregular $W$-algebras of type $A$ has also been proven in another way by 
Creutzig and Linshaw \cite{CL:Trialities}
after the first version of the present article was submitted.
There it was shown that
 $\W^k(\mf{sl}_n, f_{\text{subreg}})$ at level $k=-n+(n+m)/(n-1)$ is isomorphic
 to a simple current extension of the rational vertex algebra
 $V_{\sqrt{mn}\Z}\* \W_{\ell}(\mf{sl}_m,f_{\text{prin}})$,
 where $V_{\sqrt{mn}\Z}$
is the lattice vertex algebra associated with the lattice $\sqrt{mn}\Z$
and $\ell=-m+(m+n)/(n+1)$.
\end{rem}

\begin{rem}\label{rem:equivalence}
Using a result in \cite{Urod}, which appeared after the first version of the present article was submitted, the statement of Theorem \ref{thm:fusion.An}
can be strengthened as we now describe.
\\
Let $k$ be an admissible level for $\fing = \mf{sl}_n$, 
let $\on{KL}_k$ be the  full subcategory of 
$\widehat{\BGG}_k$ consisting of modules isomorphic to a direct sum of finite dimensional $\g$-modules,
and let $\on{KL}(V_k(\g))$ be the full subcategory of 
the category of $V_k(\g)$-modules consisting of those objects which belong to $\KL_k$ as $\affg$-modules.
%Then $\on{KL}(V_k(\g))$ is a semisimple category whose simple objects are 
%$\widehat{L}_k({\lam})$ with \textcolor{red}{$\lam\in \prin^k_{\Z}=P_+^{p-n}$}. 
Then
  $\on{KL}(V_k(\g))$  is naturally a  fusion category (\cite{CreHuaYan18,CreAdm}).
  By \cite[Theorem 10.4, Remark 10.7]{Urod},
  the functor 
\begin{align}
\on{KL}(V_k(\g))\ra \W_k(\g,f_q)\on{-mod},\quad
M\mapsto H_f^0(M),
\label{eq:braided-functor}
\end{align}
  is an equivalence of 
fusion categories
for $q=n-1$ and $n$ odd
since it induces a bijection between simple objects.
Theorem \ref{thm:fusion.An} then follows from the fact \cite{CreAdm}
 $\mc{F}(\on{KL}(V_k(\g)))$  is
isomorphic to $\mc{F}(V_{p-h^{\vee}}(\g))$.
\end{rem}

\begin{rem}
The argument of Remark \ref{rem:equivalence}
can be also used to describe  the fusion categories of exceptional rectangular $W$-algebras (\cite{AM}).
Namely, let $\g=\mf{sl}_n$,
let $k = -n+p/q$ be an admissible level and suppose that  $n=qr$ for some $r\in \Z_{\geq 1}$, and let $f_q\in \mathbb{O}_q$.
 Note that $f_{q}$ is necessarily an even nilpotent.
We have
$\check{P}_{+,f}^q=\{x_0\}$.
It follows from Theorem \ref{thm:weight.param} that we have a bijection
$$P^{p-n}_+/\Z_{q}\isomap [{\prin^k_{\circ}}],\quad [\lam]\mapsto [\lam-\frac{p}{q}x_0],$$
where the cyclic group $\Z_q$ acts on $P^{p-n}_+$ as follows: $i + q\Z \in \Z_q$ sends $\lam \mapsto \bar \pi_{ir} \lam + (p-n)\varpi_{ir}$. Therefore, 
\eqref{eq:braided-functor}
gives a quotient functor of fusion categories
such that
$H_{DS,f}^0(\widehat{L}_k(\lam))\cong H_{DS,f}^0(\widehat{L}_k(\mu))$ 
if and only if $\mu\in \Z_q \lam$.
In particular,
$\mc{F}(\W_k(\g,f_q))\cong \mc{F}(V_{p-n}(\g))/\Z_q$.
\\
In a subsequent paper \cite{AvEM} with Anne Moreau,
we will show that in the special case
where $p=h^{\vee}+1$, we have
 $\W_{-n+(n+1)/q}(\g,f_q)\cong V_1(\mf{sl}_r)$. This is compatible with the 
fact that
$\mc{F}(V_{1}(\mf{sl}_{qr}))/\Z_q\cong \mc{F}(V_{1}(\mf{sl}_r))$.
%above description of the category of $\W_{(n+1)/q-n}(\mf{sl}_n,f_q)$-modules.
\\
We also note that 
Ueda \cite{Ueda} has recently  constructed a surjective homomorphism
from the affine Yangian \cite{BoyLev94,GuaNakWen18} of type $A_r$
to the current algebra
of the rectangular $W$-algebra $\W^k(\mf{gl}_n,f_q)$.
It would be very interesting to 
understand the modular category
of exceptional rectangular $W$-algebras
in terms of the representation theory of the affine Yangian.
\end{rem}

\section{Fusion Rules of Subregular \texorpdfstring{$W$}{W}-Algebras in Types \texorpdfstring{$D$}{D} and \texorpdfstring{$E$}{E}}\label{sec:DE.type}

Now suppose $\fing$ to be of type $D$ or $E$. In this section we reduce the description of fusion rings of the associated rational subregular $\W$-algebras to a finite number of cases which can be computed explicitly. We denote by $\alstar$ the root associated with the trivalent node of the Dynkin diagram of $\fing$, and by $\varpi_*$ the corresponding fundamental weight.
\begin{prop}\label{prop:DEbigsmall}{\ }

\textup{(}a\textup{)} Let $q$ be a subregular denominator for $\fing$ coprime to $h^\vee+1$ and let $k = -h^\vee+p/q$ where $p \geq h^\vee$ is coprime to $q$ and to $|J|$. Then there exists an isomorphism of fusion rings
\begin{align*}
\CF(\W_k(\fing, \fsubreg)) \cong \CF(\W_{-h^\vee+(h^\vee+1)/q}(\fing, \fsubreg)) \otimes \CF(V_{p-h^\vee}(\fing))^{\text{int}}.
\end{align*}
\textup{(}b\textup{)} Let $q=h^\vee-1$ and let $k = -h^\vee+p/q$ where $p \geq h^\vee$ is coprime to $q$. Then there exists an isomorphism of fusion rings
\begin{align*}
\CF(\W_k(\fing, \fsubreg)) \cong \CF(\W_{-h^\vee+h^\vee/(h^\vee-1)}(\fing, \fsubreg)) \otimes \CF(V_{p-h^\vee}(\fing)).
\end{align*}
\end{prop}

We remark that the only subregular denominator $q$ excluded by the condition $(q, h^\vee+1) = 1$ is $\fing$ of type $D_n$ for $n \equiv 2 \pmod{3}$ and $q = 2n-4$.

\begin{proof}
(a) Since $(p, |J|)=1$ we have the relation (\ref{eq:s.gal.nuQ}). By Lemma \ref{lemma:KZ.Kint} the fusion ring of $\W_k(\fing, \fsubreg)$ is the tensor product of $\CF(V_{p-h^\vee}(\fing))^{\text{int}}$ and the fusion ring associated with the $S$-matrix $C_q$. Now in general $P_+^1$ is in bijection with the finite set $J$ and carries a transitive action of $\eW_+$, hence $|P_{+}^{1} \cap Q|=1$ and $K_{h^\vee+1}^\Z$ is a $1 \times 1$ matrix. Since we assume $(q, h^\vee+1) = 1$, we have
\begin{align*}
S_{\text{sr}}^{h^\vee+1, q} \propor \varphi_{h^\vee+1}(C_q) \otimes \varphi_q(K_{h^\vee+1}^\Z) \propor \varphi_{h^\vee+1}(C_q),
\end{align*}
so $C_q$ is a cyclotomic conjugate of the $S$-matrix of $\W_{-h+(h+1)/q}(\fing, \fsubreg)$ and the result follows.

(b) Since $(q, |J|)=1$ we have the relation (\ref{eq:s.gal.etaQ}). The fusion ring of $\W_k(\fing, \fsubreg)$ is therefore the tensor product of the fusion ring of $V_{p-h}(\fing)$ and the fusion ring associated with the $S$-matrix $C_q^\Z$. Evidently $|P_+^0|=1$ and so $K_{h^\vee}$ is a $1 \times 1$ matrix. Since $(q, h^\vee) = 1$ we have
\begin{align*}
S_{\text{sr}}^{h^\vee, q} \propor \varphi_{h^\vee}(C_q) \otimes \varphi_q(K_{h^\vee}) = \varphi_{h^\vee}(C_q),
\end{align*}
so $C_q$ is a cyclotomic conjugate of the $S$-matrix of $\W_{-h+h/(h-1)}(\fing, \fsubreg)$ and the result follows.
\end{proof}
In the following table we record some data on the $\W$-algebras of minimal numerator. The isomorphisms of the final column will be proved in Section \ref{sec:excep}. See that section also for background on the effective central charge.
\begin{table}[H]%[tp]%
\caption{Subregular $\W$-Algebras $\W_{-h+p/q}(\fing, \fsubreg)$} \label{table:php1}\centering% 
\begin{tabular}{|c|c|c|c|c|c|}
\hline 
$\fing$ & $p/q$ & $c$ & $c^{\text{eff}}$ & \#(Irred. rep.) & Isom. type \\
\hline
%  & & & & & \\
$A_n$ & $\frac{n+1}{n}$ & $0$ & $0$ & $1$ & $\mathbf{1}$ \\
$D_n$ & $\frac{2n-2}{2n-3}$ & $-\frac{12n^2-62n+78}{2n-3}$ & $1 - \frac{3}{2n-3}$ & $n-2$ & $\vir_{2, 2n-3}$ \\
$D_n\,\,\, (n \not\equiv 2 \bmod{3})$ & $\frac{2n-1}{2n-4}$ & $-\frac{2n^2-21n+52}{n-2}$ & $1-\frac{2}{n-2}$ & $n-3$ & $\vir_{3, n-2}$ \\
$E_6$ & $13 / 9$ & $0$ & $0$ & $1$ & $\mathbf{1}$ \\
$E_6$ & $13 / 10$ & $4/5$ & $4/5$ & $6$ & $\W_{-7/4}(A_2, \freg)$ \\
$E_6$ & $12 / 11$ & $-350/11$ & $10/11$ & $7$ & $\vir_{3, 22} \oplus L(21, 1)$ \\
$E_7$ & $19 / 14$ & $0$ & $0$ & $1$ & $\mathbf{1}$ \\
$E_7$ & $19 / 15$ & $-3/5$ & $3/5$ & $4$ & $\vir_{3, 5}$ \\
$E_7$ & $19 / 16$ & $-135/8$ & $9/8$ & $13$ & \\
$E_7$ & $18 / 17$ & $-1420/17$ & $20/17$ & $16$ &  \\
$E_8$ & $31 / 24$ & $0$ & $0$ & $1$ & $\mathbf{1}$ \\
$E_8$ & $31 / 25$ & $-22/5$ & $2/5$ & $2$ & $\vir_{2, 5}$ \\
$E_8$ & $31 / 26$ & $-350/13$ & $10/13$ & $6$ & $\vir_{2, 13}$ \\
$E_8$ & $31 / 27$ & $-590/9$ & $10/9$ & $12$ &  \\
$E_8$ & $31 / 28$ & $-830/7$ & $10/7$ & $25$ &  \\
$E_8$ & $30 / 29$ & $-7518/29$ & $42/29$ & $44$ &  \\
\hline 
\end{tabular}
\end{table}
%All values of $c$ in the table above are verified against formula (\ref{eq:cc.formula}) by \texttt{MAGMA}.

We observe that for $\fing$ of type $E_6$, $E_7$, $E_8$ of rank $\ell$ the cardinality of $\cpps^q$ coincides with that of the set of regular dominant integral weights of level $q-\ell+2$ for the root system $\D^\perp = \{\al \in \D | (\al, \alstar) = 0\}$ (which is, respectively, $A_5$, $D_6$, $E_7$). A similar pattern extends to types $A_n$ and $D_n$ but is more complicated since, for $\D$ of type $D_n$ for example, $\D^\perp$ has type $D_{n-2} \times A_1$. These bijections do not correspond to isomorphisms of fusion rings however. Indeed the fusion ring of $V_{1}(E_7)$ is isomorphic to the group ring of $\Z/2$, while the fusion ring of $\W_{-719/25}(E_8, \fsubreg) \cong \vir_{2, 5}$ is different.

We close this section with a remark on the case $\fing = D_n$ where $n \equiv 2 \bmod{3}$, and $k = -h+p/q$ where $q=2n-4$, and $p \geq h^\vee$ is coprime to $q$. By the proof of Proposition \ref{prop:DEbigsmall} we have an isomorphism of fusion rings
\begin{align*}
\CF(\W_k(\fing, \fsubreg)) \cong \CF(C_q) \otimes \CF(V_{p-h^\vee}(\fing))^{\text{int}}.
\end{align*}
Since $h^\vee+1$ is not coprime to $q$ we cannot set $p=h^\vee+1$ and thereby identify the fusion ring $\CF(C_q)$ as the fusion ring of a subregular $W$-algebra. Nevertheless we may compute the fusion rules of $\CF = \CF(C_q)$ explicitly for small $n$. For $\fing = D_5$ for instance $\CF$ is the group ring of $\Z/2\Z$ with its canonical basis. In general $\CF$ contains $n-3$ simple objects (naturally indexed by $\eta = \rho-\varpi_i$, where $\varpi_i$ are those fundamental weights of $D_n$ with Kac label $a_i^\vee = 2$). Based on explicit computation of $\CF$ for low ranks, we propose the following conjecture.
\begin{conj}\label{conj:Dtype}
For each positive integer $m$ there exists a rational lisse vertex algebra of central charge $c = 13 - 6(m + 1/m)$ whose $r = 3m-1$ irreducible modules, denoted $[i]$ for $i=0,1,\ldots, r$, have the following fusion rules:
\begin{align*}
[i] \boxtimes [j] \cong \bigoplus_{\substack{|i-j| \leq k \leq \min\{i+j, r-i-j\} \\ k \equiv i+j \bmod 2 }} [k].
\end{align*}
\end{conj}
Curiously the central charges that appear here are the central charges of the triplet vertex algebras.

% For example for $D_8$ we have $5$ modules and the following fusion matrices.
% \begin{align*}
% F_1 = 
% \left[\begin{array}{ccccc}
% 1 & 0 & 0 & 0 & 0 \\
% 0 & 1 & 0 & 0 & 0 \\
% 0 & 0 & 1 & 0 & 0 \\
% 0 & 0 & 0 & 1 & 0 \\
% 0 & 0 & 0 & 0 & 1 \\
% \end{array}\right]
% %
% F_2 = 
% \left[\begin{array}{ccccc}
% 0 & 1 & 0 & 0 & 0 \\
% 1 & 0 & 1 & 0 & 0 \\
% 0 & 1 & 0 & 1 & 0 \\
% 0 & 0 & 1 & 0 & 1 \\
% 0 & 0 & 0 & 1 & 0 \\
% \end{array}\right]
% %
% F_3 = 
% \left[\begin{array}{ccccc}
% 0 & 0 & 1 & 0 & 0 \\
% 0 & 1 & 0 & 1 & 0 \\
% 1 & 0 & 1 & 0 & 1 \\
% 0 & 1 & 0 & 1 & 0 \\
% 0 & 0 & 1 & 0 & 0 \\
% \end{array}\right] \\
% %
% F_4 = 
% \left[\begin{array}{ccccc}
% 0 & 0 & 0 & 1 & 0 \\
% 0 & 0 & 1 & 0 & 1 \\
% 0 & 1 & 0 & 1 & 0 \\
% 1 & 0 & 1 & 0 & 0 \\
% 0 & 1 & 0 & 0 & 0 \\
% \end{array}\right]
% %
% F_5 = 
% \left[\begin{array}{ccccc}
% 0 & 0 & 0 & 0 & 1 \\
% 0 & 0 & 0 & 1 & 0 \\
% 0 & 0 & 1 & 0 & 0 \\
% 0 & 1 & 0 & 0 & 0 \\
% 1 & 0 & 0 & 0 & 0 \\
% \end{array}\right]
% \end{align*}

\section{Sporadic Isomorphisms}\label{sec:excep}

In this section we explain the final column of Table \ref{table:php1}. First we fix some notation regarding the Virasoro minimal models. Recall that
\begin{align*}
\vir_{p, q} = H^0_{\freg}(V_{-2+p/q}(A_1))
\end{align*}
is a rational vertex algebra of central charge $c_{p, q} = 1 - 6(p-q)^2/pq$. It has $(p-1)(q-1)/2$ irreducible modules, all of the form
\begin{align*}
L(r, s) = H^0_{\freg, -}(L(\la)) \quad \text{where} \quad \la = k\La_0 + \left[s-1-r(p/q)\right]\varpi_1.
\end{align*}
Here $1 \leq r \leq q-1$ and $1 \leq s \leq p-1$ and $L(r, s) \cong L(q-r, p-s)$. The conformal dimensions of the irreducible modules are given by
\begin{align}\label{eq:min.mod.cw}
\D(L(r, s)) = h_{r, s} = \frac{(pr-qs)^2-(p-q)^2}{4pq}.
\end{align}

The effective central charge of a rational vertex algebra $V$ is by definition
\begin{align*}
c^{\text{eff}} = c - 24h_{\text{min}},
\end{align*}
where $h_{\text{min}}$ is the minimal conformal dimension of the irreducible $V$-modules.

Next let $M = \bigoplus_{n \in h + \Z_+} M_n$ be a graded vector space for which $\chi_M(\tau) = \sum_{n=0}^\infty \dim(M_n) q^{h+n}$ is convergent. If for some constants $A$, $\beta$ and $g$, the character $\chi_M$ has the following asymptotic behaviour
\begin{align}\label{eq:asy.datum}
\chi(it) \sim A t^\beta e^{\pi g / 12t},
\end{align}
then one says that $M$ has \emph{asymptotic growth} $g$ and \emph{asymptotic dimension} $A$.

The following proposition is a well known consequence of the modular invariance of characters of rational vertex algebras \cite{Zhu96}.
\begin{prop}
Let $V$ be a rational lisse vertex algebra of CFT type. The asymptotic growth of an irreducible $V$-module does not exceed the effective central charge of $V$, and equality occurs if all entries of the $S$-matrix of $V$ are nonzero.
\end{prop}
The asymptotic growth $g$ and asymptotic dimension $A$ of $\W_{-h^\vee+p/q}(\fing, f)$ are given in {\cite[Theorem 2.16]{KacWak08}} as
\begin{align}\label{eq:as.dim.fla}
g = \dim(\fing^f) - \frac{h^\vee}{pq}\dim(\fing)
\end{align}
and, if $f$ possesses a good even grading,
\begin{align}\label{eq:amp.fla}
A = \frac{1}{(pq)^{\ell/2} |P/Q^\vee|^{1/2}} \cdot \frac{1}{q^{|\D_{+, 0}|}} \prod_{\al \in \D_{>0}} 2 \sin\left(\frac{\pi}{q}(\al, x_0)\right) \cdot \prod_{\al \in \D_{+}} 2 \sin\left(\frac{\pi}{p}(\al, \rho)\right).
\end{align}
The asymptotic dimension  of the $\vir_{p, q}$-module $L(r, s)$ is
\begin{align}\label{eq:Vir.amp}
A_{p, q}^{L(r, s)} = (8/pq)^{1/2} (-1)^{(r+s)(r_0+s_0)} \sin\left(\pi\frac{p-q}{q} rr_0\right) \cdot \sin\left(\pi\frac{p-q}{p} ss_0\right),
\end{align}
where $r_0$ and $s_0$ are positive integers characterised by $r_0 < q$, $s_0 < p$ and $r_0p - s_0 q = 1$ \cite{KWPNAS}.

\begin{prop}\label{prop:sporadic.isom}
For each entry in Table \ref{table:php1} for which $c^{\text{eff}} < 1$, the isomorphism type of $\W_{-h^\vee+p/q}(\fing, \fsubreg)$ is as listed in the final column.
\end{prop}

\begin{proof}
First we consider the case $\fing = D_n$, $q = 2n-4$ (for $n \not\equiv 2 \bmod{3}$). By (\ref{eq:cc.formula}) and (\ref{eq:as.dim.fla}) the central charge $c$ and asymptotic growth $g$ of $\W = \W_{k}(\fing, \fsubreg)$ are
\begin{align*}
c = -\frac{2n^2-21n+52}{n-2} \quad \text{and} \quad g = 1 - \frac{2}{n-2}.
\end{align*}
These coincide with the values for $\vir_{3, n-2}$.

Let $N \subset \W$ be the vertex subalgebra generated by the Virasoro vector. 
By the representation theory of the Virasoro algebra, 
$N$ is either the universal Virasoro vertex algebra 
$\vir^{2, 2n-3}$ of the central charge $c$, 
or its simple quotient $\vir_{3, n-2}$. However by {\cite[Proposition 4(b)]{KWPNAS}} the asymptotic growth of $\vir^{3, n-2}$ is $1$ which is greater than the asymptotic growth of $\W$. Thus we cannot have $N \cong \vir^{3, n-2}$ and so in fact $N \cong \vir_{3, n-2}$. Since $\vir_{3, n-2}$ is a rational vertex algebra, its embedding into $\W$ induces a decomposition of the latter into a direct sum of irreducible $\vir_{3, n-2}$-modules. The direct sum is finite since $\W$ has finite dimensional graded pieces.

Each irreducible $\vir_{3, n-2}$-module has $\beta = 0$ in \eqref{eq:asy.datum} and nonzero asymptotic dimension, and all have the same asymptotic growth $g$. It follows that the asymptotic dimension of $\W$ equals the sum of the asymptotic dimensions of the modules in its decomposition, and so to prove $\W \cong \vir_{3, n-2}$ is suffices to prove equality of their asymptotic dimensions. By equation (\ref{eq:Vir.amp}) the asymptotic dimension of $\vir_{3, n-2}$ is
\begin{align*}
A_{3, n-2}^{L(1, 1)} = \frac{2}{q^{1/2}} \sin\left(\frac{2\pi}{q} \right).
\end{align*}

%If $n \equiv 0 \bmod 3$ (resp. $n \equiv 1 \bmod 3$) then $(r_0, s_0) = (2n/3-1, 2)$ (resp. $(r_0, s_0) = ((n-1)/3, 1)$) and 

% \begin{align*}
% A_{3, n-2}^{L(1, 1)} = \sqrt{\frac{8}{3(n-2)}} \sin\left(\pi\frac{(5-n)(2n-3)}{3(n-2)}\right) \cdot \sin\left(\pi\frac{2(5-n)}{3} \right),
% \end{align*}
% (verified in MAGMA)
% If $n \equiv 1 \bmod 3$ then $(r_0, s_0) = ((n-1)/3, 1)$ and equation (\ref{eq:Vir.amp}) yields
% \begin{align*}
% A_{3, n-2}^{L(1, 1)} = \sqrt{\frac{8}{3(n-2)}} \sin\left(\pi\frac{(5-n)(n-1)}{3(n-2)}\right) \cdot \sin\left(\pi\frac{5-n}{3} \right),
% \end{align*}
% (verified in MAGMA)

To compute the asymptotic dimension $A$ of $\W$ we recall that $\D_{>0} = \D_+ \backslash \{\alstar\}$ and $x_0 = \rho - \varpi_*$. We count, for each $m \in \Z_+$, the number of $\al \in \D_{>0}$ satisfying $(\al, x_0) = m$ and we deduce
\begin{align*}
\prod_{\al \in \D_{>0}} 2\sin\left(\frac{\pi}{q}(\al, x_0)\right) = 4 \sin\left(\frac{\pi}{q}\right) \sin\left( \frac{(q/2-1)\pi}{q}\right) \left[ \prod_{k=1}^{q-1} 2\sin\left(\frac{\pi k}{q}\right) \right]^{n/2+1} \cdot \left[ \prod_{k=1}^{q/2-1} 2\sin\left(\frac{2 \pi k}{q}\right) \right]^{-1/2}.
\end{align*}
Using the identity $\prod_{k=1}^{r-1} 2\sin(\pi k / r) = r$, the product reduces to
\begin{align*}
4(q)^{n/2+1} (q/2)^{-1/2} \sin\left(\frac{2\pi}{q} \right)
\end{align*}
Similarly the product over $\D_+$ reduces to $p^{n/2}$. Substituting into (\ref{eq:amp.fla}) yields $A = A_{3, n-2}^{L(1, 1)}$ as required.

Next we consider the case $\fing = D_n$, $q = 2n-3$. As above we conclude that $\W = \W_{k}(\fing, \fsubreg)$ is an extension of $\vir_{2, 2n-3}$ and computation of asymptotic dimensions reveals $\W \cong \vir_{2, 2n-3}$. 
% though in this case the same conclusion can be reached by a simpler argument, namely that no irreducible $\vir_{2, 2n-3}$-module has positive conformal dimension.
For $(\fing, q)$ one of the pairs $(E_7, 15)$, $(E_8, 25)$ or $(E_8, 26)$ the proof is again the same.

Now we consider the case $\W = \W_{-h^\vee+12/11}(E_6, \fsubreg)$ which, by the same arguments as above, is an extension of $\vir_{3, 22}$. The unique irreducible $\vir_{3, 22}$-module with conformal dimension $\D \in \Z_{\geq 1}$ is $L(21, 1)$ with $\D=5$. Thus $\W \cong \vir_{3, 22} \oplus L(21, 1)^{\oplus n}$ for some $n \in \Z_+$. Comparison of asymptotic dimensions reveals $n = 1$.

The final case $\fing = E_6$, $q = 10$ is quite similar to the last case. Analysis of central charges, asymptotic growths and asymptotic dimensions reveals that $\W_{-h^\vee+13/10}(E_6, \fsubreg)$ decomposes as $\vir_{5, 6} \oplus L(5, 1)$. The same arguments imply that $\W_{-7/4}(A_2, \freg) = \vir_{5, 6} \oplus L(5, 1)$ (a fact well known in the physics literature \cite[pp. 227]{DMS.CFT}). Now the module $L(5, 1)$ is a simple current, so by uniqueness of simple current extensions \cite[Proposition 5.3]{DM2004} we obtain the claimed isomorphism.
\end{proof}

\section{New Modular Tensor Categories}

We have computed the fusion rules of $\W_{-h^\vee+p/q}(\fing, \fsubreg)$ in terms of the fusion rules of the affine vertex algebra $V_{p-h^\vee}(\fing)$ and those of the vertex algebras listed in Table \ref{table:php1}. We have identified most of these vertex algebras, the remaining cases being $E_6$, $p/q=12/11$ and the five algebras with asymptotic growth greater than $1$. We compute the fusion rings of these vertex algebras from their $S$-matrices. In this section we present in detail the cases $E_6$, $p/q=12/11$ and $E_7$, $p/q=19/16$.

A vertex algebra $V$ is said to be \emph{positive} if every irreducible $V$-module besides $V$ itself has positive conformal dimension. We observe that the vertex algebras $U_6 = \W_{-h^\vee+15/11}(E_6, \fsubreg)$ and $U_7 = \W_{-h^\vee+21/16}(E_7, \fsubreg)$ are positive. It is natural to expect the following.
\begin{conj}\label{conj}
The vertex operator algebras $U_6$ and $U_7$ 
are unitary.
\end{conj}
We remark that the analogue $U_8 = \W_{-h^\vee+32/25}(E_8, \fsubreg)$ is isomorphic to the unitary theory $\vir_{4, 5}$. 
The MTCs associated with $U_6$ and $U_7$, which we also believe to be unitary, appear to be new and interesting. We recall that the quantum dimension of an irreducible $V$-module $L$ is $S_{V,L}/S_{V,V}$ where $S$ is the $S$-matrix of $V$. By results of \cite{DJX.qdim} the quantum dimensions of all irreducible $U_6$- and $U_7$-modules $L$ must satisfy $\qdim(L) \geq 1$. By Theorem \ref{prop:DEbigsmall} the $S$-matrices are
\begin{align}\label{eq:S.U}
\begin{split}
S^{(U_6)} &= S^{15, 11}_{\text{sr}} = \varphi_3(S^{12, 11}_{\text{sr}}) \otimes \varphi_{11}(K_{15}(E_6)) \\
\text{and} \quad S^{(U_7)} &= S^{21, 16}_{\text{sr}} = \varphi_3(S^{19, 16}_{\text{sr}}) \otimes \varphi_{16}(K_{21}(E_7)^{\text{int}}),
\end{split}
\end{align}
and the relation $\qdim(L) \geq 1$ is verified by direct computation of quantum dimensions from (\ref{eq:S.U}).

To compute the fusion rules of $U_6$ and $U_7$ it suffices to compute the fusion rules of $V_6 = \W_{-h^\vee+12/11}(E_6, \fsubreg)$ and $V_7 = \W_{-h^\vee+19/16}(E_7, \fsubreg)$.

The vertex algebra $V_6$ is an extension of $V^0 = \vir_{3, 22}$ by its simple current $M = L(21, 1)$. We note that $M \boxtimes L(r, 1) \cong L(22-r, 1)$. The $7$ irreducible $V_6$-modules are obtained as follows: the fusion product $V_6 \boxtimes_{V^0} L(r, 1)$ is an irreducible $V_6$-module for $1 \leq r \leq 9$ odd, while for $r=11$ it decomposes into two irreducible $V_6$-modules consisting of symmetric resp. antisymmetric tensors. We also note that $V_6$ carries a $\Z/2$-action by virtue of the decomposition $V_6 = V^0 \oplus M$, and for $r$ even $V_6 \boxtimes_{V^0} L(r, 1)$ is an irreducible $\Z/2$-twisted $V$-module.

The fusion rules of $V_6$ can be compactly described as follows. We write $[i] = V \boxtimes_{V^0}L(2i+1, 1)$ for $0 \leq i \leq 5$ and $[5] = [5_+] \oplus [5_-]$. Then $[5_+]' = [5_-]$ and all other $[i]$ are self dual. The conformal dimensions are given by $\D([i]) = i(3i-19)/22$. For $0 \leq i, j \leq 4$ we have $[i] \boxtimes [j] = \bigoplus_{|i-j|\leq k \leq i+j} [k]$, where we identify $[5+\epsilon]$ with $[5-\epsilon]$. Also $[5_{\pm}] \boxtimes [i] = [5_{\pm(-1)^i}] + \bigoplus_{5-i \leq k \leq 4} [k]$. Finally $[5_\pm] \boxtimes [5_\pm] = [1] \oplus [3] \oplus [5_{\mp}]$ and $[5_+] \boxtimes [5_-] = [0] \oplus [2] \oplus [4]$.

The quantum dimensions of the irreducible $V_6$-modules lie in the cyclotomic field $\Q(\zeta)$ of degree $11$ and are given explicitly as:
\begin{align*}
\qdim([0]) &= 1, &
\qdim([1]) &= -\zeta^7 - \zeta^4 + 1, \\
\qdim([2]) &= \zeta^8 - \zeta^7 - \zeta^4 + \zeta^3 + 1, &
\qdim([3]) &= \zeta^9 + 2\zeta^8 + \zeta^6 + \zeta^5 + 2\zeta^3 + \zeta^2 + 2, \\
\qdim([4]) &= \zeta^9 + 2\zeta^8 + 2\zeta^6 + 2\zeta^5 + 2\zeta^3 + \zeta^2 + 2, &
\qdim([5_\pm]) &= \zeta^8 + \zeta^6 + \zeta^5 + \zeta^3 + 1.
\end{align*}

We now examine the vertex algebra $V_7$. Since the asymptotic growth of $V_7$ is $9/8 > 1$, it is not a finite extension of a Virasoro minimal model. The central charge is $c = -135/8$ and it has $13$ irreducible modules, parametrised by $\cpps^{16}$. We denote by $\sigma$ the nontrivial diagram automorphism of the Dynkin diagram of $E_7$.
\begin{center}
  \begin{tikzpicture}[scale=.4]
    \draw (-1,1) node[anchor=east]  {};
% Circles
    \foreach \x in {-1,...,5}
    \draw[thick,xshift=\x cm] (\x cm,0) circle (3 mm);
% Edges
    \foreach \y in {0,...,4}
    \draw[thick,xshift=\y cm] (\y cm,0) ++(.3 cm, 0) -- +(14 mm,0);
% Top circle and edge
    \draw[thick] (4 cm,2 cm) circle (3 mm);
    \draw[thick] (4 cm, 3mm) -- +(0, 1.4 cm);
% Dotted line
    \draw[thick,dashed,xshift=-1 cm] (-1 cm,0) ++(.3 cm, 0) -- +(14 mm,0);
% Labels
    \draw[xshift=0 cm,thick] node[below=8pt,fill=white] {\scriptsize{$1$}};
    \draw[xshift=2 cm,thick] node[below=8pt,fill=white] {\scriptsize{$3$}};
    \draw[xshift=4 cm,thick] node[below=8pt,fill=white] {\scriptsize{$4$}};
    \draw[xshift=6 cm,thick] node[below=8pt,fill=white] {\scriptsize{$5$}};
    \draw[xshift=8 cm,thick] node[below=8pt,fill=white] {\scriptsize{$6$}};
    \draw[xshift=10 cm,thick] node[below=8pt,fill=white] {\scriptsize{$7$}};
    \draw[xshift=4 cm, yshift=2cm, thick] node[left=8pt,fill=white] {\scriptsize{$2$}};
    \draw[xshift=-2 cm,thick] node[below=8pt,fill=white] {\scriptsize{$0$}};
  \end{tikzpicture}
\end{center}
We list the $13$ weights $\wh\eta = 16\Lambda_0 + \eta$ where $\eta \in \cpps^{16}$: they are
\begin{align*}
\widehat\eta_{0} &= (1; 1, 1, 1, 0, 1, 1, 3 ), &
\widehat\eta_{1} &= (1; 2, 1, 1, 0, 1, 1, 1 ), \\
\widehat\eta_{2} &= (1; 0, 1, 1, 1, 1, 1, 1 ), &
\widehat\eta_{3} &= (1; 1, 1, 1, 1, 0, 1, 2 ), \\
\widehat\eta_{4} &= (1; 1, 1, 0, 1, 1, 1, 2 ), &
\widehat\eta_{5} &= (1; 1, 2, 1, 0, 1, 1, 1 ), \\
\widehat\eta_{6} &= (1; 1, 0, 1, 1, 1, 1, 1 ), &
\widehat\eta_{7} &= (2; 1, 1, 1, 0, 1, 1, 2 ),
\end{align*}
together with $\widehat\eta_{12-i} = \sigma(\widehat\eta_{i})$ for $i=0,1,2,3,4$. Let $M_i$ denote the irreducible $V$-module associated with the weight $\widehat\eta_i$. Then $M_{12}$ is a simple current of order $2$ and conformal dimension $3/2$. The fusion product with $M_{12}$ acts as $\sigma$ at the level of weights. All irreducible $V_7$-modules are self-contragredient, and the fusion rules $F_i = M_{i} \boxtimes (-)$ are as follows.
\begin{align*}
F_1 = \left[\begin{array}{ccccccccccccc}
0 & 1 & 0 & 0 & 0 & 0 & 0 & 0 & 0 & 0 & 0 & 0 & 0 \\
1 & 1 & 1 & 0 & 0 & 0 & 0 & 0 & 0 & 0 & 0 & 0 & 0 \\
0 & 1 & 1 & 0 & 0 & 0 & 0 & 0 & 0 & 1 & 0 & 0 & 0 \\
0 & 0 & 0 & 1 & 0 & 0 & 0 & 0 & 0 & 1 & 1 & 0 & 0 \\
0 & 0 & 0 & 0 & 0 & 0 & 1 & 0 & 1 & 0 & 0 & 0 & 0 \\
0 & 0 & 0 & 0 & 0 & 1 & 1 & 1 & 0 & 0 & 0 & 0 & 0 \\
0 & 0 & 0 & 0 & 1 & 1 & 1 & 0 & 1 & 0 & 0 & 0 & 0 \\
0 & 0 & 0 & 0 & 0 & 1 & 0 & 1 & 0 & 0 & 0 & 0 & 0 \\
0 & 0 & 0 & 0 & 1 & 0 & 1 & 0 & 0 & 0 & 0 & 0 & 0 \\
0 & 0 & 1 & 1 & 0 & 0 & 0 & 0 & 0 & 1 & 0 & 0 & 0 \\
0 & 0 & 0 & 1 & 0 & 0 & 0 & 0 & 0 & 0 & 1 & 1 & 0 \\
0 & 0 & 0 & 0 & 0 & 0 & 0 & 0 & 0 & 0 & 1 & 1 & 1 \\
0 & 0 & 0 & 0 & 0 & 0 & 0 & 0 & 0 & 0 & 0 & 1 & 0 \\
\end{array}\right] 
F_2 = \left[\begin{array}{ccccccccccccc}
0 & 0 & 1 & 0 & 0 & 0 & 0 & 0 & 0 & 0 & 0 & 0 & 0 \\
0 & 1 & 1 & 0 & 0 & 0 & 0 & 0 & 0 & 1 & 0 & 0 & 0 \\
1 & 1 & 1 & 1 & 0 & 0 & 0 & 0 & 0 & 1 & 0 & 0 & 0 \\
0 & 0 & 1 & 1 & 0 & 0 & 0 & 0 & 0 & 1 & 1 & 1 & 0 \\
0 & 0 & 0 & 0 & 1 & 1 & 1 & 0 & 0 & 0 & 0 & 0 & 0 \\
0 & 0 & 0 & 0 & 1 & 1 & 1 & 1 & 1 & 0 & 0 & 0 & 0 \\
0 & 0 & 0 & 0 & 1 & 1 & 2 & 1 & 1 & 0 & 0 & 0 & 0 \\
0 & 0 & 0 & 0 & 0 & 1 & 1 & 0 & 0 & 0 & 0 & 0 & 0 \\
0 & 0 & 0 & 0 & 0 & 1 & 1 & 0 & 1 & 0 & 0 & 0 & 0 \\
0 & 1 & 1 & 1 & 0 & 0 & 0 & 0 & 0 & 1 & 1 & 0 & 0 \\
0 & 0 & 0 & 1 & 0 & 0 & 0 & 0 & 0 & 1 & 1 & 1 & 1 \\
0 & 0 & 0 & 1 & 0 & 0 & 0 & 0 & 0 & 0 & 1 & 1 & 0 \\
0 & 0 & 0 & 0 & 0 & 0 & 0 & 0 & 0 & 0 & 1 & 0 & 0 \\
\end{array}\right]
\end{align*}
\begin{align*}
F_3 = \left[\begin{array}{ccccccccccccc}
0 & 0 & 0 & 1 & 0 & 0 & 0 & 0 & 0 & 0 & 0 & 0 & 0 \\
0 & 0 & 0 & 1 & 0 & 0 & 0 & 0 & 0 & 1 & 1 & 0 & 0 \\
0 & 0 & 1 & 1 & 0 & 0 & 0 & 0 & 0 & 1 & 1 & 1 & 0 \\
1 & 1 & 1 & 1 & 0 & 0 & 0 & 0 & 0 & 1 & 1 & 1 & 0 \\
0 & 0 & 0 & 0 & 1 & 1 & 1 & 1 & 0 & 0 & 0 & 0 & 0 \\
0 & 0 & 0 & 0 & 1 & 1 & 2 & 0 & 1 & 0 & 0 & 0 & 0 \\
0 & 0 & 0 & 0 & 1 & 2 & 2 & 1 & 1 & 0 & 0 & 0 & 0 \\
0 & 0 & 0 & 0 & 1 & 0 & 1 & 0 & 1 & 0 & 0 & 0 & 0 \\
0 & 0 & 0 & 0 & 0 & 1 & 1 & 1 & 1 & 0 & 0 & 0 & 0 \\
0 & 1 & 1 & 1 & 0 & 0 & 0 & 0 & 0 & 1 & 1 & 1 & 1 \\
0 & 1 & 1 & 1 & 0 & 0 & 0 & 0 & 0 & 1 & 1 & 0 & 0 \\
0 & 0 & 1 & 1 & 0 & 0 & 0 & 0 & 0 & 1 & 0 & 0 & 0 \\
0 & 0 & 0 & 0 & 0 & 0 & 0 & 0 & 0 & 1 & 0 & 0 & 0 \\
\end{array}\right] 
F_4 = \left[\begin{array}{ccccccccccccc}
0 & 0 & 0 & 0 & 1 & 0 & 0 & 0 & 0 & 0 & 0 & 0 & 0 \\
0 & 0 & 0 & 0 & 0 & 0 & 1 & 0 & 1 & 0 & 0 & 0 & 0 \\
0 & 0 & 0 & 0 & 1 & 1 & 1 & 0 & 0 & 0 & 0 & 0 & 0 \\
0 & 0 & 0 & 0 & 1 & 1 & 1 & 1 & 0 & 0 & 0 & 0 & 0 \\
1 & 0 & 1 & 1 & 0 & 0 & 0 & 0 & 0 & 0 & 0 & 1 & 0 \\
0 & 0 & 1 & 1 & 0 & 0 & 0 & 0 & 0 & 1 & 1 & 0 & 0 \\
0 & 1 & 1 & 1 & 0 & 0 & 0 & 0 & 0 & 1 & 1 & 1 & 0 \\
0 & 0 & 0 & 1 & 0 & 0 & 0 & 0 & 0 & 1 & 0 & 0 & 0 \\
0 & 1 & 0 & 0 & 0 & 0 & 0 & 0 & 0 & 1 & 1 & 0 & 1 \\
0 & 0 & 0 & 0 & 0 & 1 & 1 & 1 & 1 & 0 & 0 & 0 & 0 \\
0 & 0 & 0 & 0 & 0 & 1 & 1 & 0 & 1 & 0 & 0 & 0 & 0 \\
0 & 0 & 0 & 0 & 1 & 0 & 1 & 0 & 0 & 0 & 0 & 0 & 0 \\
0 & 0 & 0 & 0 & 0 & 0 & 0 & 0 & 1 & 0 & 0 & 0 & 0 \\
\end{array}\right]
\end{align*}
\begin{align*}
F_5 = \left[\begin{array}{ccccccccccccc}
0 & 0 & 0 & 0 & 0 & 1 & 0 & 0 & 0 & 0 & 0 & 0 & 0 \\
0 & 0 & 0 & 0 & 0 & 1 & 1 & 1 & 0 & 0 & 0 & 0 & 0 \\
0 & 0 & 0 & 0 & 1 & 1 & 1 & 1 & 1 & 0 & 0 & 0 & 0 \\
0 & 0 & 0 & 0 & 1 & 1 & 2 & 0 & 1 & 0 & 0 & 0 & 0 \\
0 & 0 & 1 & 1 & 0 & 0 & 0 & 0 & 0 & 1 & 1 & 0 & 0 \\
1 & 1 & 1 & 1 & 0 & 0 & 0 & 0 & 0 & 1 & 1 & 1 & 1 \\
0 & 1 & 1 & 2 & 0 & 0 & 0 & 0 & 0 & 2 & 1 & 1 & 0 \\
0 & 1 & 1 & 0 & 0 & 0 & 0 & 0 & 0 & 0 & 1 & 1 & 0 \\
0 & 0 & 1 & 1 & 0 & 0 & 0 & 0 & 0 & 1 & 1 & 0 & 0 \\
0 & 0 & 0 & 0 & 1 & 1 & 2 & 0 & 1 & 0 & 0 & 0 & 0 \\
0 & 0 & 0 & 0 & 1 & 1 & 1 & 1 & 1 & 0 & 0 & 0 & 0 \\
0 & 0 & 0 & 0 & 0 & 1 & 1 & 1 & 0 & 0 & 0 & 0 & 0 \\
0 & 0 & 0 & 0 & 0 & 1 & 0 & 0 & 0 & 0 & 0 & 0 & 0 \\
\end{array}\right] 
F_6 = \left[\begin{array}{ccccccccccccc}
0 & 0 & 0 & 0 & 0 & 0 & 1 & 0 & 0 & 0 & 0 & 0 & 0 \\
0 & 0 & 0 & 0 & 1 & 1 & 1 & 0 & 1 & 0 & 0 & 0 & 0 \\
0 & 0 & 0 & 0 & 1 & 1 & 2 & 1 & 1 & 0 & 0 & 0 & 0 \\
0 & 0 & 0 & 0 & 1 & 2 & 2 & 1 & 1 & 0 & 0 & 0 & 0 \\
0 & 1 & 1 & 1 & 0 & 0 & 0 & 0 & 0 & 1 & 1 & 1 & 0 \\
0 & 1 & 1 & 2 & 0 & 0 & 0 & 0 & 0 & 2 & 1 & 1 & 0 \\
1 & 1 & 2 & 2 & 0 & 0 & 0 & 0 & 0 & 2 & 2 & 1 & 1 \\
0 & 0 & 1 & 1 & 0 & 0 & 0 & 0 & 0 & 1 & 1 & 0 & 0 \\
0 & 1 & 1 & 1 & 0 & 0 & 0 & 0 & 0 & 1 & 1 & 1 & 0 \\
0 & 0 & 0 & 0 & 1 & 2 & 2 & 1 & 1 & 0 & 0 & 0 & 0 \\
0 & 0 & 0 & 0 & 1 & 1 & 2 & 1 & 1 & 0 & 0 & 0 & 0 \\
0 & 0 & 0 & 0 & 1 & 1 & 1 & 0 & 1 & 0 & 0 & 0 & 0 \\
0 & 0 & 0 & 0 & 0 & 0 & 1 & 0 & 0 & 0 & 0 & 0 & 0 \\
\end{array}\right]
\end{align*}
\begin{align*}
F_7 = \left[\begin{array}{ccccccccccccc}
0 & 0 & 0 & 0 & 0 & 0 & 0 & 1 & 0 & 0 & 0 & 0 & 0 \\
0 & 0 & 0 & 0 & 0 & 1 & 0 & 1 & 0 & 0 & 0 & 0 & 0 \\
0 & 0 & 0 & 0 & 0 & 1 & 1 & 0 & 0 & 0 & 0 & 0 & 0 \\
0 & 0 & 0 & 0 & 1 & 0 & 1 & 0 & 1 & 0 & 0 & 0 & 0 \\
0 & 0 & 0 & 1 & 0 & 0 & 0 & 0 & 0 & 1 & 0 & 0 & 0 \\
0 & 1 & 1 & 0 & 0 & 0 & 0 & 0 & 0 & 0 & 1 & 1 & 0 \\
0 & 0 & 1 & 1 & 0 & 0 & 0 & 0 & 0 & 1 & 1 & 0 & 0 \\
1 & 1 & 0 & 0 & 0 & 0 & 0 & 0 & 0 & 0 & 0 & 1 & 1 \\
0 & 0 & 0 & 1 & 0 & 0 & 0 & 0 & 0 & 1 & 0 & 0 & 0 \\
0 & 0 & 0 & 0 & 1 & 0 & 1 & 0 & 1 & 0 & 0 & 0 & 0 \\
0 & 0 & 0 & 0 & 0 & 1 & 1 & 0 & 0 & 0 & 0 & 0 & 0 \\
0 & 0 & 0 & 0 & 0 & 1 & 0 & 1 & 0 & 0 & 0 & 0 & 0 \\
0 & 0 & 0 & 0 & 0 & 0 & 0 & 1 & 0 & 0 & 0 & 0 & 0 \\
\end{array}\right]
\end{align*}

\def\cprime{$'$} \def\cprime{$'$}

\bibliographystyle{alpha}
\bibliography{/Users/tomoyuki/Documents/Dropbox/bib/math}

\end{document}